\documentclass[11pt]{article}

\usepackage{amsmath, amsfonts, amssymb, amsthm,  graphicx, mathtools, enumerate}
\usepackage{bm}
\usepackage[blocks, affil-it]{authblk}
\usepackage{subcaption}
\usepackage{bbm}
\usepackage{enumitem}
\usepackage{tikz}
\usetikzlibrary{fit,positioning,arrows,automata,calc}
\usepackage[numbers, square]{natbib}
\usepackage{float}
\usepackage[CJKbookmarks=true,
            bookmarksnumbered=true,
			bookmarksopen=true,
			colorlinks=true,
			citecolor=red,
			linkcolor=blue,
			anchorcolor=red,
			urlcolor=blue]{hyperref}
\usepackage[letterpaper, left=1.2truein, right=1.2truein, top = 1.2truein, bottom = 1.2truein]{geometry}
\usepackage[ruled, vlined, lined, commentsnumbered]{algorithm2e}

\usepackage{prettyref,soul}

\def\hat{\widehat}
\newtheorem{lemma}{Lemma}[section]
\newtheorem{proposition}{Proposition}[section]
\newtheorem{thm}{Theorem}[section]

\newtheorem{corollary}{Corollary}[section]
\newtheorem{remark}{Remark}[section]

\newrefformat{eq}{(\ref{#1})}
\newrefformat{chap}{Chapter~\ref{#1}}
\newrefformat{sec}{Section~\ref{#1}}
\newrefformat{algo}{Algorithm~\ref{#1}}
\newrefformat{fig}{Fig.~\ref{#1}}
\newrefformat{tab}{Table~\ref{#1}}
\newrefformat{rmk}{Remark~\ref{#1}}
\newrefformat{clm}{Claim~\ref{#1}}
\newrefformat{def}{Definition~\ref{#1}}
\newrefformat{cor}{Corollary~\ref{#1}}
\newrefformat{lmm}{Lemma~\ref{#1}}
\newrefformat{lemma}{Lemma~\ref{#1}}
\newrefformat{prop}{Lemma~\ref{#1}}
\newrefformat{app}{Appendix~\ref{#1}}
\newrefformat{ex}{Example~\ref{#1}}
\newrefformat{exer}{Exercise~\ref{#1}}
\newrefformat{soln}{Solution~\ref{#1}}
\newrefformat{cond}{Condition~\ref{#1}}

\makeatletter
\DeclareRobustCommand\widecheck[1]{{\mathpalette\@widecheck{#1}}}
\def\@widecheck#1#2{%
    \setbox\z@\hbox{\m@th$#1#2$}%
    \setbox\tw@\hbox{\m@th$#1%
       \widehat{%
          \vrule\@width\z@\@height\ht\z@
          \vrule\@height\z@\@width\wd\z@}$}%
    \dp\tw@-\ht\z@
    \@tempdima\ht\z@ \advance\@tempdima2\ht\tw@ \divide\@tempdima\thr@@
    \setbox\tw@\hbox{%
       \raise\@tempdima\hbox{\scalebox{1}[-1]{\lower\@tempdima\box
\tw@}}}%
    {\ooalign{\box\tw@ \cr \box\z@}}}
\makeatother
\let\check\widecheck

\def\text#1{\mbox{\rm #1}}

\DeclarePairedDelimiter{\ceil}{\lceil}{\rceil}
\DeclarePairedDelimiter{\floor}{\lfloor}{\rfloor}

\newcommand{\argmax}{\mathop{\rm argmax}}

\newcommand{\indc}[1]{\mathbb{I}\left\{#1\right\}}

\newcommand{\wt}{\widetilde}
\newcommand{\Norm}[1]{\|{#1} \|}

\newcommand{\Prob}{\mathbb{P}}
\newcommand{\Expect}{\mathbb{E}}

\newcommand{\pth}[1]{\left( #1 \right)}

\newcommand{\TV}{{\sf TV}}

\newcommand{\iidsim}{\stackrel{\mathrm{iid}}{\sim}}

\newcommand*{\rom}[1]{\uppercase\expandafter{\romannumeral #1\relax}}

\newtheorem*{m2'}{Condition M2'}

\newtheorem{condition}{Condition}

\def\tPi{\wt \Pi}
\newcommand{\eps}{\varepsilon}
\newcommand{\bp}[1]{\left(#1\right)}
\newcommand{\cG}{\mathcal{G}}
\newcommand{\cT}{\mathcal{T}}
\newcommand{\TVdiff}[2]{\Norm{#1-#2}_{\TV}}
\newcommand{\cB}{\mathcal{B}}
\newcommand{\cA}{\mathcal{A}}
\newcommand{\cN}{\mathcal{N}}
\newcommand{\bb}[1]{\left\{#1\right\}}
\newcommand{\KL}[2]{D \left( #1 \middle\| #2 \right)}
\newcommand{\dO}{\bm{\mathit{\Delta}} O}
\newcommand{\dn}{\bm{\mathit{\Delta}} n}
\newcommand{\dtO}{\bm{\mathit{\Delta}}\wt O}
\newcommand{\dtn}{\bm{\mathit{\Delta}}\wt n}
\newcommand{\abs}[1]{\left|#1\right|}
\newcommand{\ignore}[1]{}
\newcommand{\goto}{\rightarrow}
\newcommand{\PP}[1]{\mathbb{P}\left\{{#1}\right\}} 
\newcommand{\EE}[1]{\mathbb{E}\left[{#1}\right]} 

\newcommand{\PPst}[2]{\mathbb{P}\left\{{#1}\ \middle| \ {#2}\right\}}
\newcommand{\beps}{\bar\eps}
\DeclareMathOperator{\var}{Var}
\newcommand{\VV}[1]{\var\left({#1}\right)} 
\newcommand{\tP}{\wt P}
\newcommand{\tB}{\wt B}
\renewcommand{\O}[1]{\mathcal{O}\left({#1}\right)}
\newcommand{\cM}{{\mathcal M}}
\newcommand{\cl}{{C_\lambda}}
\newcommand{\tR}{\wt R}
\newcommand{\cE}{{\mathcal E}}

\title{Mixing Time of Metropolis-Hastings for Bayesian Community Detection
}

\author{Bumeng Zhuo}
\author{Chao Gao}
\affil{University of Chicago
}

\begin{document}
\maketitle

\begin{abstract}
We study the computational complexity of a Metropolis-Hastings algorithm for Bayesian community detection. We first establish a posterior strong consistency result for a natural prior distribution on stochastic block models under the optimal signal-to-noise ratio condition in the literature. We then give a set of conditions that guarantee rapid mixing of a simple Metropolis-Hastings algorithm. The mixing time analysis is based on a careful study of posterior ratios and a canonical path argument to control the spectral gap of the Markov chain.
\smallskip

\end{abstract}


\section{Introduction}

Markov Chain Monte Carlo (MCMC) is a popular sampling technique, in which the equilibrium distribution of the Markov chain matches the target distribution. Most attention to date has been focused on Bayesian applications in order to sample from posterior distributions. Despite its popularity in Bayesian statistics and many other areas, its theoretical properties are not well understood, not to mention the limited theory for computational efficiency of MCMC algorithms, where the pivotal interest lies in the analysis of mixing time. The mixing time of a Markov chain is the number of iterations required to get close enough to the target distribution, in the sense that the total variation distance is bounded above by some small constant $\varepsilon$. We call the Markov chain rapidly mixing (resp. slowly mixing) if the mixing time grows at most polynomially (resp. exponentially) with respect to the sample size of the problem. One central research interest is to determine whether a designed Markov chain is rapidly mixing or slowly mixing.

Though slow mixing is the central hurdle of Markov chains, to the best of our knowledge, there has been little work on the theoretical analysis of mixing time. A series of studies have made efforts to design efficient Markov chains \cite{thawornwattana2018designing,wu2004efficient,moller2006efficient,hutter2014efficient,tu2002image,wu2004efficient} without providing theoretical guarantees. Over the past fifteen years, a surge of research has led to breakthroughs in the understandings of geometric ergodicity of Markov chains \cite{roberts1997geometric,meyn1994computable,diaconis1991geometric,sinclair1992improved}, and several elegant techniques were developed to characterize the mixing property \cite{bubley1997path,guruswami2016rapidly,diaconis1991geometric,sinclair1992improved,meyn2012markov,randall2006rapidly,levin2017markov}. Canonical path is one main tool to show rapid mixing of Markov chains, and the idea is to design a set of paths between all pairs of points such that no edge is ``overloaded'' (congested). The method of canonical paths heavily relies on the graph structure, and the design of low congestion canonical paths remains a highly non-trivial artwork especially for exponentially large state space of Markov chains, which limits the range of applications. However, in some statistical problems, the construction of canonical paths may take advantage of the underlying model, and quantitative bounds for the convergence rate and mixing time of Markov chains can be obtained under some general conditions. Yang et al. \cite{yang2016computational} was one of the first to apply the canonical paths idea to a Bayesian variable selection problem, and obtained an explicit upper bound for the mixing time under some mild conditions. Inheriting their ideas, this paper applies the same technique to a Bayesian community detection problem.

Motivated by the computational advantages of Gibbs sampling, a Bayesian point of view of community detection was first suggested in \cite{snijders1997estimation} with only two communities. The approach was further extended in \cite{nowicki2001estimation,hofman2008bayesian} to incorporate adjusted priors on community proportions as well as edge probabilities and allow for the case of more than two communities. There has been little theoretical analysis of Bayesian community detection until very recently, when the consistency results of posterior distribution were obtained by \cite{van2017bayesian}. However, they required the expected degree of a node to be at least of order $\log^2 n$ to ensure the strong consistency of Bayesian posterior mode, which is a suboptimal condition for stro{}ng consistency \cite{abbe2016community,abbe2014exact,zhang2016minimax}. Compared with the work on statistical performance, little work has been done on the computational efficiency to sample from the posterior distribution, and it was once suggested that the mixing time for high dimensional Bayesian community detection should scale exponentially, because the Markov chain must eventually go over all possible states.

\ignore{\textbf{Goal of this paper}} This paper considers both the statistical and computational performances of the Metropolis-Hastings algorithm for community detection. The primary goal is to study the computational complexity for recovering the community memberships in a social network using a simple Markov chain. To be concrete, we construct a Bayesian model with some specific prior, analyze the performance of the corresponding posterior distribution, and provide a rapidly mixing time bound for the induced Metropolis-Hastings algorithm. In particular, we introduce a uniform prior for the community label assignments and a Beta prior for the connectivity probabilities. In order to sample from the induced posterior distribution, we adopt a ``tampered'' Metropolis-Hastings algorithm that takes a ``single flip'' procedure as the proposal distribution and accepts the update with a scaled posterior ratio with respect to some temperature parameter.

\ignore{\textbf{The result and contribution.}} The Bayesian model considered in this paper yields great statistical performance, in the sense that the posterior strong consistency holds when the expected degree is of order $\log n$, compared with the previous expected degree condition $\log ^2n$ in \cite{van2017bayesian}. In particular, it is shown in \cite{abbe2016community,abbe2014exact,zhang2016minimax} that this condition for strong consistency is optimal. A simple Metropolis-Hastings algorithm is adopted to sample from the posterior distribution with temperature parameter adjusted. We are able to show that the Markov chains induced by the algorithm is rapidly mixing. To be specific, when the expected degree is of order $\log n$, the mixing time of the Markov chain scales as $O(n^3\log n)$. It is worth noting that the analysis of this paper is based on a frequentist point of view that the social network data are generated according to an underlying true model.

\textbf{Organization.} The rest of the paper is organized as follows. Section \ref{sec:BayComDetect} formally sets up the community detection problem and presents the posterior strong consistency property. Section \ref{sec:rapidmixingalg} introduces a Metropolis-Hastings algorithm and provides an explicit mixing time bound, followed by some numerical results demonstrating its competitive performance on simulated datasets in Section \ref{sec:NumResults}. Section \ref{sec:mainProofs} is devoted to the proofs of the technical results of the paper.

\textbf{Notations.} We close this section by introducing some notations. For an integer $d$, we use $[d]$ to denote $\{1,2,\ldots,d\}$. For a set $S$, we write $\indc{S}$ as its indicator function and $|S|$ as its cardinality. For a vector $v\in \mathbb{R}^d$, its norms are defined by $\|v\|_1=\sum_{i=1}^d |v_i|$, $\Norm{v}^2 = \sum_{i=1}^d v_i^2$, and $\Norm{v}_\infty = \max_{1\leq i\leq d}|v_i|$. The Hamming error of two binary vectors $v_1,v_2\in \{0,1\}^d$ is defined by $H(v_1,v_2) = \sum_{i=1}^d \indc{v_1(i)\neq v_2(i)}.$ For a matrix $A\in \mathbb{R}^{K\times K}$, its norms are defined by $\Norm{A}_{\infty} = \max_{i,j\in[K]}\abs{A_{ij}}$, and $\Norm{A}_{1} = \sum_{i,j\in[K]}\abs{A_{ij}}$. The notation $\Prob$ and $\Expect$ are generic probability and expectation operators whose distribution is determined from the context. We use $o(1)$ to denote any positive sequence tending to 0. Throughout the paper, unless otherwise noticed, we use $C$, $c$ and their variants to denote absolute constants, and the values may vary from line to line. For any two distributions $P$ and $Q$, the total variation distance is defined by $\left\| P-Q \right\|_{\TV} = \frac{1}{2}\int \abs{dP-dQ} $, and the KL divergence is defined by $\KL{P}{Q}=\int dP\log\frac{dP}{dQ}$. For simplicity, we write $\KL{p}{q}$ to denote $\KL{\text{Bernoulli}(p)}{\text{Bernoulli}(q)}$ for $p,q\in[0,1]$. For any two numbers $a$ and $b$, we use $a\wedge b$ and $a\vee b$ to denote $\min\{a,b\}$ and $\max\{a,b\}$ respectively.

\section{Bayesian Community Detection}\label{sec:BayComDetect}

Networks have arisen in various areas of applications and have attracted a surge of research interests in fields such as physics, computer science, social sciences, biology, and statistics \cite{goldenberg2010survey,newman2010networks,wasserman1994social,fortunato2010community,chen2006detecting}. In the realm of network analysis, community detection has emerged as a fundamental task that provides insights of the underlying structure. Great advances have been made on community detection recently with a remarkable diversity of models and algorithms developed in different areas \cite{girvan2002community,newman2007mixture,handcock2007model}. Among various statistical models, the stochastic block model (SBM), first proposed in \cite{holland1983stochastic}, is one of the most prominent generative model that depicts the network topologies and incorporates the community structure. It is arguably the simplest model of a graph with communities and has been widely applied in social, biological and communication networks. Much effort has been devoted to SBM-based methods and their asymptotic properties have also been studied recently \cite{bickel2009nonparametric,celisse2012consistency,bickel2013asymptotic}.

In this section, we give a precise formulation of the community detection problem and introduce a Bayesian approach. Then, we present the posterior strong consistency result.

\subsection{Problem formulation}

Consider an unweighted and undirected network with $n$ nodes and $K$ communities. The adjacency matrix is denoted by $A\in \{0,1\}^{n\times n}$, $A=A^T$, and $A_{ii}=0$, for all $i\in[n]$. The edges are independently generated as Bernoulli variable with $\Expect{A_{ij}} = P_{ij}$, for all $i<j$. Here, $P_{ij}$ denotes the connectivity probability for nodes $i$ and $j$, and depends on the communities that the two nodes are assigned to. In this paper, we focus on a homogeneous SBM and assume $P_{ij} = p$ if two nodes are from the same community and $P_{ij} = q$ otherwise. We call $p$ (resp. $q$) as the within-community (resp. between-community) connectivity probability and assume $p>q$ to satisfy the ``assortative'' property. Extensions to heterogenous SBMs are straightforward, but will not be considered in the paper for the sake of the presentation.

Let $Z\in [K]^n$ denote a label assignment vector, where $Z_i$ is the community label for the $i$th node. Let $B\in[0,1]^{K\times K}$ be a symmetric connectivity probability matrix and thus $P_{ij} = B_{Z_i Z_j}$ with $B_{aa}=p$ for all $a\in [K]$, and $B_{ab} = q$ for all $a\neq b$. According to the description of the model, the likelihood formula can be written as
\begin{align}\label{eq:like}
p(A|{Z,B}) = \prod_{i<j}B_{Z_iZ_j}^{A_{ij}}\pth{1-B_{Z_iZ_j}}^{1-A_{ij}}.
\end{align}

We use $Z^*$ to denote the underlying true label assignment vector, and further assume that 
\begin{equation}\label{eq:beta}
	\frac{n}{\beta K}\leq \sum_{i=1}^n \indc{Z^*_i = k}\leq\frac{\beta n}{K},~\text{for all~} k\in [K],
\end{equation}
where $\beta\geq 1$ is an absolute constant. It indicates that the all community sizes are of the same order. When $\beta = 1+o(1)$, all communities have almost the same sizes. Furthermore, we assume $K$ is a known constant, $p,q\rightarrow 0$ and $p\asymp q$ throughout the paper. To conclude, this paper focuses on a sparse homogeneous SBM with a finite number of communities.

Note that community detection is a clustering problem, and thus any label assignment gives an equivalent result after a label permutation. To be specific, let 
\begin{equation}\label{eq:GammaZ}
\Gamma(Z) = \{\sigma\circ Z: \sigma\in \mathcal{P}_K\},    
\end{equation}
where $\mathcal{P}_K$ stands for the set of all permutations on $[K]$, and then any $Z'\in \Gamma(Z)$ leads to an equivalent clustering structure. Hence, with the identifiability issue, our ultimate goal is to reconstruct the community structure, or equivalently, to recover the community label assignment $Z^*$ up to a label permutation.

\subsection{A Bayesian model for community detection}
In addition to the likelihood formula of the adjacency matrix $A$ given in \prettyref{eq:like}, we put a uniform prior on $Z$ over a set $S_\alpha$, where $S_\alpha$ is the set of all feasible label assignments depending on a hyperparameter $\alpha$. The connectivity probabilities $B_{ab}$ for $1\leq a\leq b\leq K$ receive independent Beta priors. More precisely, the Bayesian model is given by

\begin{align*}
&\text{stochastic block model:} && p(A|Z,B) = \prod_{i<j} B_{Z_i Z_j}^{A_{ij}}\pth{1-B_{Z_i Z_j}}^{1-A_{ij}},\\[5pt]
&\text{label assignment prior:} && \pi(Z)\propto \indc{Z\in S_\alpha},\\[5pt]
&\text{connectivity probability prior:} && B_{ab}\iidsim \text{Beta}(\kappa_1,\kappa_2), ~~~ 1\leq a\leq b\leq K,
\end{align*}
where $\kappa_1,\kappa_2>0$ measure the prior information of the connectivity probabilities and have negligible effects on the results when the sample size is large enough. This is essentially the same set-up in \cite{van2017bayesian}, except that we introduce a uniform prior over set $S_\alpha$. The key set $S_\alpha$ is defined by
\begin{equation}\label{eq:S_alpha}
	S_\alpha = \left\{Z:\sum_{i=1}^n \indc{Z_i=k}\in \left[\frac{n}{\alpha K},\frac{\alpha n}{K}\right],\text{ for all } k \in [K] \right\},
\end{equation}
where the hyperparameter $\alpha$ controls the size of the feasible set $S_\alpha$, which rules out those models whose group sizes differ too much. We require $\alpha>\beta$ so that $Z^*\in S_\alpha$. \ignore{Note that the similar trick is also used in \cite{yang2016computational}, where a restricted set was introduced to eliminate models in extreme cases.} As will be clarified in Section \ref{sec:NumResults}, this additional constraint seems to be necessary for the rapidly mixing according to our practical experiments.

The induced posterior distribution can be expressed as
\begin{align*}
\Pi(Z|A)&\propto \int_{[0,1]^{K(K+1)/2}}\prod_{a\leq b} B_{ab}^{ O_{ab}(Z)}(1-B_{ab})^{ n_{ab}(Z)- O_{ab}(Z)}d\Pi(B)\\
&\propto \prod_{a\leq b}\text{Beta}( O_{ab}(Z)+\kappa_1, n_{ab}(Z)- O_{ab}(Z)+\kappa_2), & \text{for $Z\in S_\alpha$},
\end{align*}
and it follows that for $Z\in S_\alpha$, 
\begin{align}\label{eq:postunknown}
\log\Pi(Z|A) &= \sum_{a\leq b}\log \text{Beta}( O_{ab}(Z)+\kappa_1, n_{ab}(Z)- O_{ab}(Z)+\kappa_2)+Const,
\end{align}
where $n_{ab}(Z) = n_a(Z)n_b(Z)$, $n_{aa}(Z)=n_a(Z)(n_a(Z)-1)/2$ for all $a\neq b\in[K]$. We use $n_a(Z)$ to denote the size of community $a$, i.e., $n_a(Z) = |\{i:Z_i=a\}|$. We use $O_{ab}(Z)$ to denote the number of connected edges between communities $a$ and $b$, which takes the formula $O_{ab}(Z)=\sum_{i,j}A_{ij}\indc{Z_i=a, Z_j=b}$ and $O_{aa}(Z) = \sum_{i<j}A_{ij}\indc{Z_i=Z_j=a}$ for all $a\neq b\in[K]$. Note that the posterior distribution is permutation symmetric, i.e., 
\begin{equation}
    \Pi(Z|A) = \Pi(Z'|A), ~~~\text{for all }Z'\in \Gamma(Z). \label{eq:per-symm}
\end{equation}

\subsection{Posterior strong consistency}\label{sec:postStrongConsis}

Before stating the theoretical properties of the proposed Bayesian model, we introduce some useful quantities. The first quantity $I$ plays a crucial part in the minimax theory \cite{zhang2016minimax},
\[
	I = -2\log (\sqrt{pq}+\sqrt{(1-p)(1-q)}),
\]
which is the R\'enyi divergence of order 1/2 between Bernoulli($p$) and Bernoulli($q$). It can be shown that when $p,q\rightarrow 0$,
\[
	I = (1+o(1)) (\sqrt{p}-\sqrt{q})^2.
\]
Then, we introduce an effective sample size to simplify the presentation of the results. As mentioned in \cite{zhang2016minimax}, the minimax misclassification error rate is determined by that of classifying two communities of the smallest sizes. When $K=2$, the hardest case is when one has two communities of the same size $n/2$. When $K>2$, the hardest case is when one has two communities of sizes $n/K\beta$. Thus, we define 
\begin{equation}
	\bar n  = \begin{dcases}	
		\frac{n}{2}, &\text{for~} K=2,\\[5pt]
		\frac{n}{K\beta}, &\text{for~}K>2,
			\end{dcases}	\displaystyle
\end{equation}
as the effective sample size of the problem. The following result characterizes the statistical performance of the posterior distribution $\Pi(Z|A)$ under mild conditions. 

\begin{thm}[Posterior strong consistency]\label{thm:postcontract}
	Recall that $\Gamma(Z) = \{\sigma\circ Z: \sigma\in \mathcal{P}_K\}$, where $\mathcal{P}_K$ stands for the set of all permutations on $[K]$. Suppose that
	\begin{equation}
		\liminf_{n\rightarrow \infty} \frac{\bar nI}{\log n}>1,
	\end{equation}
	and the feasible set $S_\alpha$ satisfies that $\alpha-\beta$ is a positive constant. Then, we have that
	\[
		\Expect{[\Pi(Z\in \Gamma(Z^*)|A)]} \geq 1-n\exp(-(1-\eta_n)\bar nI)=1-o(1)
	\]
	for a large $n$ and some positive sequence $\eta_n$ tending to 0 as $n\goto \infty$, and the expectation is with respect to the data-generating process.
\end{thm}

We defer the proof of the theorem to Section \ref{sec:mainProofs}. It is worth noting that the condition required in \prettyref{thm:postcontract} is identical to the fundamental limits required for exact label recovery \cite{abbe2016community,abbe2014exact,zhang2016minimax}. In the special case of two communities of equal sizes, we require $nI>2\log n$ to guarantee the strong consistency result. Hence, \prettyref{thm:postcontract} implies that under our Bayesian framework, posterior strong consistency holds under the optimal condition.

We can also compare the statistical performance of our model with other Bayesian community detection approaches. The first Bayesian SBM was suggested by \cite{snijders1997estimation}, who considered two communities and proposed a uniform prior for both community proportions and the connectivity probabilities. It was further extended for more communities with Dirichlet priors on community proportions and Beta priors on the connectivity probabilities. However, the field of Bayesian SBM grows in a slow pace due to lack of theoretical analysis in terms of statistical consistency. Recently, van der Pas and van der Vaart \cite{van2017bayesian} proved that the strong consistency result holds under a condition where the expected degree satisfies $\lambda_n\gg\log^2 n$. In contrast, our model introduces a feasible set $S_\alpha$ and proposes a uniform prior for label assignment $Z$ on set $S_\alpha$. It results in the strong consistency of posterior distribution under the condition that ${n(p-q)^2}/{p}\gtrsim \log n$, much weaker than the condition required in \cite{van2017bayesian}. 

\ignore{
\textbf{review strong consistency results in literature} \url{http://www.jmlr.org/papers/volume18/16-245/16-245.pdf} Page 2
\url{https://arxiv.org/pdf/1608.04242.pdf}}

\section{Rapidly mixing of a Metropolis-Hastings algorithm}\label{sec:rapidmixingalg}

In this section, we propose a modified Metropolis-Hastings random walk, and analyze its statistical performance as well as the computational complexity. Due to the identifiability of the problem, the rapidly mixing property is analyzed in clustering space that will be defined in the sequel.

\subsection{A Metropolis-Hastings algorithm}\label{sec:alg}

A general Metropolis-Hastings algorithm is an iterative procedure consisting of two steps:

\begin{description}[leftmargin = 0cm,labelsep = 0.5cm]
	\item[Step 1] For the current state $X_t$, generate an $X' \sim q(x|X_t)$, where $q(x|X_t)$ is the proposal distribution defined on the same state space.
	\item[Step 2] Move to the new state $X'$ with acceptance probability $\rho(X_t,X')$, and stay in the original state $X_t$ with probability $1-\rho(X_t,X')$, where the acceptance probability is given by
	\[
		\rho(X_t,X') = \min \left\{1,\frac{p(X')q(X_{t}|X')}{p(X_{t})q(X'|X_{t})}\right\},
	\]
    where $p(\cdot)$ is the target distribution.
\end{description}
In this paper, we are sampling the community label assignment $Z\in [K]^n$. In particular, we take the \textit{single flip update} as the proposal distribution, which is to choose an index $j\in[n]$ uniformly at random, and then randomly choose $c\in [K]\setminus \{Z_{t}(j)\}$ to assign a new label. The whole algorithm is presented in \prettyref{algo:Ag1}.

\bigskip
\begin{algorithm}[H]
	\label{algo:Ag1}
	\SetArgSty{textnormal}
	\SetKwInOut{Input}{Input}
	\caption{A Metropolis-Hastings algorithm for Bayesian community detection}
	\Input{Adjacency matrix $A\in\{0,1\}^{n\times n}$,\\
	number of communities $K$,\\
initial community assignment $Z_0$,\\
inverse temperature parameter $\xi$,\\
maximum number of iterations $T$.}
\KwOut{Community label assignment $Z_{T}$.}
\For{ each $t\in \{0,1,2,\ldots,T\}$ }{
		Choose an index $j\in[n]$ uniformly at random\;
		Randomly assign a new label for index $j$ from the set $[K]\setminus \{Z_{t}(j)\}$ to get a new assignment $Z'$\;
		$Z_{t+1} = Z'$ with probability
		\[\rho(Z_t,Z') = \min\left\{1,\frac{\Pi^\xi(Z'|A)}{\Pi^\xi(Z_t|A)}\right\},\]
		otherwise set $Z_{t+1} = Z_{t}$.
}

\end{algorithm}
\bigskip

The Markov chain induced by Algorithm \ref{algo:Ag1} is characterized by the transition matrix, which takes the form as
\begin{equation}\label{eq:tran_mat}
    P(Z,Z') = \begin{dcases}
\frac{1}{n(K-1)}\min \left\{1,\frac{\Pi^\xi(Z'|A)}{\Pi^\xi(Z|A)}\right\}, &\text{if~} H(Z,Z') = 1,\\
1-\sum_{Z'\neq Z}{P(Z,Z')},&\text{if~} Z' = Z,\\
0, &\text{if~} H(Z, Z')>1,
\end{dcases}
\end{equation}
where $H(Z,Z')$ is the Hamming error between the two label assignments $Z,Z'$. The inverse temperature parameter $\xi$ satisfies that $\xi\geq 1$. The algorithm is sampling from the scaled distribution $\tPi(Z|A)$, where $\tPi(Z|A)\propto\Pi^\xi(Z|A)$ for any $Z\in [K]^n$. As $\xi\goto \infty$, the probability mass of $\tPi(\cdot|A)$ concentrates on the global maximum of $\Pi(\cdot|A)$, in which case the algorithm is deterministic and reduces to a label switching algorithm, as discussed in \cite{bickel2009nonparametric}. When $\xi=1$, asymptotically the algorithm is sampling from the true posterior distribution. The possible choices of $\xi$ will be discussed in the sequel. 

The parameter $T$ in \prettyref{algo:Ag1} is the total number of iterations required. As long as the Markov chain mixes after $T$, according to Theorem \ref{thm:postcontract}, $Z_T$ recovers the true community label assignment up to a label permutation with high probability, i.e., $Z_T\in \Gamma(Z^*)$ where $\Gamma(Z^*)$ is as defined in \eqref{eq:GammaZ}. Even though Theorem \ref{thm:postcontract} is only stated for $\xi=1$, it is easy to see that its conclusion also holds for $\wt{\Pi}(\cdot|A)$ for a general $\xi\geq 1$.

Due to the identifiability issue, the theoretical analysis of mixing time will be performed in the clustering space $\{\Gamma(Z): Z\in S_\alpha\}$, where $\Gamma(Z)$ is defined in \eqref{eq:GammaZ}. We denote the state in the clustering space at time $t$ as $\Gamma_t=\Gamma(Z_t)$, where $Z_t$ is generated from Algorithm \ref{algo:Ag1}. The graphical model of the sequence $\{\Gamma_t\}_{t\geq 0}$ is illustrated by Figure \ref{fig:update_clustering_space}.

\begin{figure}[!htb]
\begin{center}
\tikzset{
  main/.style={circle, minimum size = 5mm, thick, draw =black!80, node distance = 10mm},
  connect/.style={-latex, thick},
  box/.style={rectangle, draw=black!100}
}
\begin{tikzpicture}
  \node[box,draw=white!100] (Latent) {\textbf{State $Z_t$}};
  \node[main] (L1) [right=of Latent] {$Z_0$};
  \node[main] (L2) [right=of L1] {$Z_1$};
  \node[main] (L3) [right=of L2] {$Z_2$};
  \node[main] (Lt) [right=of L3] {$Z_t$};
  \node[main, scale=0.75] (Ltp1) [right=of Lt] {$Z_{t+1}$};

  \node[main,fill=black!10] (O1) [below=of L1] {$\Gamma_0$};
  \node[main,fill=black!10] (O2) [below=of L2] {$\Gamma_1$};
  \node[main,fill=black!10] (O3) [below=of L3] {$\Gamma_2$};
  \node[main,fill=black!10] (Ot) [below=of Lt] {$\Gamma_t$};
  \node[main,fill=black!10, scale=0.75] (Otp1) [below=of Ltp1] {$\Gamma_{t+1}$};

  \node[box,draw=white!100,left=of O1] (Observed) {\textbf{State $\Gamma_t$}};
  
  \path (L3) -- node[auto=false]{\ldots} (Lt);
  \path (L1) edge [connect] (L2)
        (L2) edge [connect] (L3)
        (Lt) edge [connect] (Ltp1)
        (L3) -- node[auto=false]{\ldots} (Lt);
  \path (L1) edge [connect] (O1);
  \path (L2) edge [connect] (O2);
  \path (L3) edge [connect] (O3);
  \path (Lt) edge [connect] (Ot);
  \path (Ltp1) edge [connect] (Otp1);
\end{tikzpicture}
\end{center}
\caption{Updating process of $\Gamma(Z_t)$.}
\label{fig:update_clustering_space}
\end{figure}
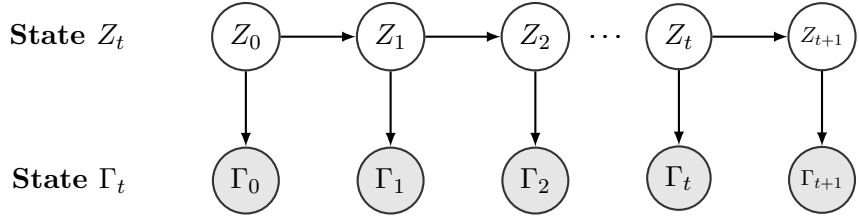

\begin{proposition}
The sequence $\{\Gamma_t\}_{t\geq 0}$ induced by Algorithm \ref{algo:Ag1} is a Markov chain.
\end{proposition}
\begin{proof}
The proof relies on the permutation symmetry of the posterior distribution given by (\ref{eq:per-symm}). We first introduce a distance between two clustering structures $\Gamma$ and $\Gamma'$, defined by
\begin{equation}\label{eq:hammingGamma}
    \check{H}(\Gamma,\Gamma') = \min_{Z\in \Gamma, Z'\in \Gamma'}H(Z,Z').
\end{equation}
When $\check{H}(\Gamma_{t+1},\Gamma_{t})\leq 1$, we have
\begin{align}\label{eq:MCproof}
    &\PPst{\Gamma_{t+1}}{\Gamma_s,~s\leq t}= \sum_{Z\in \Gamma_t}\PPst{\Gamma_{t+1}}{Z_t=Z}\cdot \PPst{Z_t=Z}{\Gamma_s,~s\leq t}.
\end{align}
The equality holds since given $Z_t$, $\Gamma_{t+1}$ and $\{\Gamma_s:s\leq t\}$ are independent. We proceed to calculate $\PPst{\Gamma_{t+1}}{Z_t=Z}$. In the case of $\check{H}(\Gamma_{t+1},\Gamma_{t})\leq 1$, it is obvious that for any $Z\in \Gamma_t$, there exists a unique $Z'\in \Gamma_{t+1}$ such that $H(Z,Z')\leq 1$. Thus, we have that
\begin{equation}\label{eq:tran_mat_proof}
    \PPst{\Gamma_{t+1}}{Z_t=Z} = \sum_{\wt Z\in \Gamma_{t+1}}P(Z,\wt Z) = P(Z,Z').
\end{equation}
By \eqref{eq:tran_mat}, the transition probability $P(Z,Z')$ only depends on the ratio of $\Pi(Z'|A)$ and $\Pi(Z|A)$, and 
\begin{equation}
    \frac{\Pi(Z'|A)}{\Pi(Z|A)}=  \frac{\Pi(\Gamma(Z')|A)}{\Pi(\Gamma(Z)|A)} = \frac{\Pi(\Gamma_{t+1}|A)}{\Pi(\Gamma_t|A)},
\end{equation}
which only depends on $\Gamma_t$ and $\Gamma_{t+1}$. It follows that
\[
    \PPst{\Gamma_{t+1}}{Z_t} = \PPst{\Gamma_{t+1}}{\Gamma_t}.
\]
Hence, plug the above identity into \eqref{eq:MCproof}, and we have
\begin{align*}
    \PPst{\Gamma_{t+1}}{\Gamma_s,~s\leq t} &= \sum_{Z\in \Gamma_t}\PPst{\Gamma_{t+1}}{\Gamma_t} \cdot \PPst{Z_t=Z}{\Gamma_s,~s\leq t}\\
    & = \PPst{\Gamma_{t+1}}{\Gamma_t} \cdot\sum_{Z\in \Gamma_t} \PPst{Z_t=Z}{\Gamma_s,~s\leq t}\\
    & = \PPst{\Gamma_{t+1}}{\Gamma_t}.
\end{align*}
When $\check{H}(\Gamma_{t+1}, \Gamma_t)>1$, it is obvious that
\[
    \PPst{\Gamma_{t+1}}{\Gamma_s,~s\leq t} = 0 = \PPst{\Gamma_{t+1}}{\Gamma_t}.
\]
Therefore, $\{\Gamma_t\}_{t\geq 0}$ is a Markov chain by combining the conclusions of the two cases. 
\end{proof}

According to the above proposition and its proof, we can define the transition matrix $\check P$ from state $\Gamma$ to $\Gamma'$ as
\begin{equation}\label{eq:tran_mat_gamma}
\begin{split}
    \check P(\Gamma,\Gamma') = \begin{dcases}
    \frac{1}{n(K-1)}\min \left\{1,\left[{\frac{\Pi(\Gamma'|A)}{\Pi(\Gamma|A)}}\right]^\xi \right\}, &\text{if~} \check H(\Gamma,\Gamma') = 1,\\
1-\sum_{\Gamma'\neq \Gamma}{\check P(\Gamma,\Gamma')},&\text{if~} \Gamma' = \Gamma,\\
0, &\text{if~} \check H(\Gamma, \Gamma')>1.
    \end{dcases}
\end{split}
\end{equation}
We perform the analysis of mixing time for the Markov chain $\{\Gamma_t\}_{t\geq 0}$. Write $\check S_\alpha = \{\Gamma(Z):Z\in S_\alpha\}$ for simplicity, and we define the target distribution in the clustering space as $\check{\Pi}(\Gamma|A)=\sum_{Z\in \Gamma}\tPi(Z|A)$ for any $\Gamma\in \check S_\alpha$. We show in the next section that $\{\Gamma_t\}_{t\geq 0}$ is rapidly mixing to the target distribution $\check{\Pi}(\cdot |A)$.

\subsection{Main results}
Before stating the main theorem, we first review the definition of $\eps$-mixing time, as well as the loss function that we need for the community detection problem.

\begin{description}[leftmargin = 0cm,labelsep = 0.5cm]
	\item[$\eps$-mixing time.] Let $\Gamma_0=\Gamma(Z_0)$ be the initial state of the chain. The total variation distance to the stationary distribution after $t$ iterations is
\[
	\Delta_{Z_0}(t) = \left\| \check{P}^t(\Gamma_0,\cdot )-\check\Pi(\cdot|A) \right\|_{\TV} = \frac{1}{2}\sum_{\Gamma\in\{\Gamma(Z):Z\in S_\alpha\}} \abs{\check{P}^t(\Gamma_0,\Gamma)-\check\Pi(\Gamma|A)},
\]
where $\check{P}^t(\Gamma_0,\cdot)$ and $\check{\Pi}(\cdot|A)$ are both distributions defined in the clustering space. The $\eps$-mixing time for Algorithm \ref{algo:Ag1} starting at $Z_0$ is defined by
\begin{equation}\label{eq:mixingFormulaEps}
    \tau_\eps(Z_0) = \min \left\{t\in \mathbb N: \Delta_{Z_0}(t')\leq \eps ~\text{for all } ~t'\geq t\right\}.
\end{equation}
It is the minimum number of iterations required to ensure the total variation distance to the stationary distribution is less than some tolerance threshold $\eps$.
	\item[Loss function.] We introduce the misclassification proportion as a loss function, which is defined by
	\begin{equation}
		\ell(Z,Z^*) = \frac{1}{n}\check H(\Gamma(Z), \Gamma(Z^*)),
	\end{equation}
	where $\check H(\cdot, \cdot)$ in defined in \eqref{eq:hammingGamma}.
\end{description}
To this end, let us show that the proposed modified Metropolis-Hastings algorithm in \prettyref{sec:alg} gives a rapidly mixing Markov chain $\{\Gamma_t\}_{t\geq 0}$ under the following conditions.

\begin{condition}\label{con:Z0}
There exist some positive sequences $\eta=\eta(n)$ and $\gamma_0=\gamma_0(n)$ such that 
\[
    \inf_{B,Z^*}\PP{\ell(Z_0,Z^*)\leq \gamma_0} \geq 1-\eta.
\]
\end{condition}
We proceed to state the conditions for $\gamma_0$.
\begin{condition}\label{con:gamma0}
Suppose the sequence $\gamma_0$ in Condition \ref{con:Z0} satisfies one of the following cases:
\begin{itemize}
    \item Case 1: there are only two communities, i.e., $K=2$, and
\begin{equation}
    (1-K\gamma_0)^4 nI\goto \infty, ~~(1-K\gamma_0)(1-K\beta\gamma_0) n\goto \infty,
\end{equation}
where $\beta\geq 1$ is defined in \prettyref{eq:beta}.
\item Case 2: there are more than 2 communities, i.e., $K\geq 3$, and
\begin{equation}
    \gamma_0=o(1).
\end{equation}
\end{itemize}
\end{condition}

Condition \ref{con:Z0} and Condition \ref{con:gamma0} require that the misclassification number of the initial label assignment is less than the minimum community size $n/K\beta$ with high probability. Consider the special situation where $K=2$ and $\beta =1+o(1)$, i.e., the underlying two community share the same size asymptotically, Condition \ref{con:gamma0} is satisfied when the initial misclassification proportion is $1/2-\eps$ for some sequence $\eps\goto 0$. When $K\geq 3$, we require a stronger condition that the initial label assignment need to be weakly consistent, i.e., the initial misclassification error goes to $0$ as $n\goto \infty$. The initial condition can be easily satisfied by algorithms such as spectral clustering \cite{mcsherry2001spectral,rohe2011spectral,coja2010graph,fishkind2013consistent}.

\begin{condition}\label{con:xi}
Suppose $\limsup_{n\goto \infty}{\log n}/{\bar nI}=1-\eps_0$. With the hyperparameter $\xi$ defined in \prettyref{algo:Ag1}, one of the following cases holds:
\begin{itemize}
	\item Case 1: there are only two communities, i.e., $K=2$, and
	\begin{equation}
		\xi >(1-\eps_0)\left\{\frac{1}{2\eps_0}\vee \frac{\alpha^2}{(1-K\gamma_0)^4}\right\},
	\end{equation}
	where $\alpha$ is defined in \eqref{eq:S_alpha}, and $\gamma_0$ is defined in Condition \ref{con:Z0}.
	\item Case 2: there are more than 2 communities, i.e., $K\geq 3$, and
	\begin{equation}
		\xi > \frac{1-\eps_0}{2\eps_0}.
	\end{equation}
\end{itemize}
\end{condition}
Note that the condition for the inverse temperature hyperparameter $\xi$ also depends on the signal condition ($\eps_0$) and initialization condition ($\gamma_0$). The condition of $\xi$ is provided to ensure the strong rapidly mixing property in the worst scenario. Note that with stronger initialization condition for the case of $K\geq 3$, the condition of $\xi$ is slightly weaker than the case of $K=2$.

Here are some intuitive understandings of Condition \ref{con:xi}. Theorem \ref{thm:postcontract} shows that posterior strong consistency holds under the condition $\liminf_{n\goto \infty}\bar nI/\log n>1$. Suppose for two label assignments $Z_1,Z_2$ that $\Pi(Z_1|A)>\Pi(Z_2|A)$, with hyperparameter $\xi\geq 1$, the posterior ratio $\Pi^\xi(Z_1|A)/\Pi^\xi(Z_2|A)$ gets enlarged, and the Markov chain is more certain to move towards the maximum point. However, the value of $\xi$ is also constrained by the initialization $Z_0$. With a larger $\xi$, $\Pi^\xi(Z_0|A)$ is smaller and it takes longer for the Markov chain $\{\Gamma_t\}_{t\geq 0}$ to get mixed. The special case is that when the initialization $Z_0$ is weakly consistent, or equivalently, $\ell(Z_0,Z^*)\goto 0$ as $n\goto \infty$, then the value of $\xi$ only depends on $\eps_0$. It gives the following alternative condition that can replace Condition \ref{con:gamma0} and Condition \ref{con:xi}.

\begin{condition}\label{con:gamma0xi_small}
Denote $\limsup_{n\goto\infty}\log n/\bar nI=1-\eps_0$. The positive sequence $\gamma_0$ defined in Condition \ref{con:Z0} and the hyperparameter $\xi$ satisfy that
\begin{equation}
    \gamma_0=o(1),~~    \xi\geq \frac{1-\eps_0}{2\eps_0}.
\end{equation}
\end{condition}

\begin{thm}[Rapidly mixing]\label{thm:mixing}
	The initial label assignment is denoted by $Z_0$. Suppose Conditions (\ref{con:Z0}, \ref{con:gamma0}, \ref{con:xi}) or Conditions (\ref{con:Z0}, \ref{con:gamma0xi_small}) are satisfied. Then, the $\eps$-mixing time of the modified Metropolis-Hastings algorithm is upper bounded by
	\begin{equation}\label{eq:mixing}
		\tau_\eps(Z_0)\leq 4Kn^2\max\left\{\gamma_0,n^{-\tau}\right\}\cdot \left(\xi\log \left(\Pi(Z_0|A)^{-1}\right)+\log(\eps^{-1})\right)
	\end{equation}
	with probability at least $1-C_1n^{-C_2}-\eta$ for some constant $C_1,C_2>0$, where $\tau$ is a sufficiently small constant, and $\eta$ is defined in Condition \ref{con:Z0}.
\end{thm}
\begin{remark}
It is classical to perform theoretical analysis on a lazy version of Markov chain, which has probability 1/2 of staying unchanged, and the other probability 1/2 of updating the state. Theorem \ref{thm:mixing} is proved for the lazy Markov chain induced by Algorithm \ref{algo:Ag1}, i.e., the corresponding transition matrix is $(\check P+I)/2$. The same tricks are widely used in \cite{yang2016computational,dabbs2009markov,berestycki2016mixing,montenegro2006mathematical}. It is worth noting that this is only for the proof, and in practice, we still use the original transition matrix in Algorithm \ref{algo:Ag1}.
\end{remark}
Theorem \ref{thm:mixing} implies the mixing time depends on the initialization $Z_0$ and the choice of $\xi$. In order to show that the mixing time is at most a polynomial of $n$, we still need the following lemma to lower bound the initial posterior value $\Pi(Z_0|A)$.

\begin{lemma}\label{lm:minpost}
Under the conditions of Theorem \ref{thm:postcontract}, we have
\begin{equation}
	\log\Pi(Z_0|A)\geq -C_3n^2I\cdot \ell(Z_0,Z^*),
\end{equation}
with probability at least $1-C_4n^{-C_5}$ for some positive constants $C_3,C_4,C_5$.
\end{lemma}
Theorem \ref{thm:mixing} and Lemma \ref{lm:minpost} jointly imply that $\tau_\eps(Z_0)\lesssim n^2(n^2I+\log(\eps^{-1}))$ with high probability, which demonstrates that the Markov chain of Metropolis-Hastings algorithm is rapidly mixing. To the best of our knowledge, \eqref{eq:mixing} is the first explicit upper bound on the mixing time of the Markov chain for Bayesian community detection. 

Note that the target distribution of Algorithm \ref{algo:Ag1} is $\tPi(\cdot|A)\propto \Pi^\xi(\cdot |A)$. Since $\xi\geq 1$ and the posterior strong consistency property still holds for $\tPi(\cdot|A)$, Theorem \ref{thm:mixing} shows that Algorithm \ref{algo:Ag1} will find the maximum a posteriori in polynomial time with high probability.

\begin{corollary}\label{cor:opt}
Under the condition of Theorem \ref{thm:mixing}, for any iteration number $T$ such that $T\geq C_6n^2(n^2I+\log (\eps^{-1}))$ for some constant $C_6$, the output $Z_T$ of the Algorithm \ref{algo:Ag1} satisfies that $Z_T\in \Gamma(Z^*)$ with high probability, or equivalently, $\ell(Z_T,Z^*)=0$.
\end{corollary}

The following corollary focuses on the case of $\xi=1$, and gives explicit conditions for the Markov chain to converge to the posterior distribution $\Pi(\cdot|A)$.

\begin{corollary}
When $nI/\log n\goto \infty$, suppose Condition \ref{con:gamma0} holds, and we can take $\xi =1$ in Algorithm \ref{algo:Ag1}, which reduces to the standard Metropolis-Hastings algorithm sampling from $\Pi(\cdot|A)$. We have that the $\eps$-mixing time of the Markov chain is upper bounded by $O(n^2(n^2I+\log (\eps^{-1})))$ with high probability.
\end{corollary}

The conditions of the above results can be weakened in the case where the connectivity probability matrix $B$ is known. When $B$ is known, there is no need to put a prior on $B$. Thus, the posterior distribution can be simplified as
\begin{equation}\label{eq:postknown}
\begin{split}
  \log \Pi(Z|A) = &\log\frac{p(1-q)}{q(1-p)}\sum_{i<j}A_{ij}\indc{Z_i=Z_j} - \\
  &\log \frac{1-q}{1-p}\sum_{i<j}\indc{Z_i=Z_j}+ Const,  ~~~~~\text{for~}Z\in S_\alpha.
\end{split}
\end{equation}
The posterior formula is essentially the same as likelihood, while we restrict $Z$ inside the feasible set $S_\alpha$. It can be shown that the posterior strong consistency property still holds in this case. 

\begin{thm}[posterior strong consistency]\label{thm:postknown}
Suppose that $\limsup_{n\goto\infty}\bar nI/\log n>1$, and the feasible set $S_\alpha$ satisfies that $\alpha-\beta$ is a positive constant, then it follows that
\[
    \Expect\left[\Pi(Z\in \Gamma(Z^*)|A)\right] \geq 1-o(1),
\]
with high probability, and the expectation is with respect to the data-generating process.
\end{thm}

\begin{condition}\label{con:pqknowncon}
    Suppose $\limsup_{n\goto \infty}\log n/\bar nI=1-\eps_0$. Assume the positive sequence $\gamma_0$ defined in Condition \ref{con:Z0} and the hyperparameter $\xi$ satisfy one of the following conditions:
\begin{itemize}
    \item Case 1: 
\begin{equation}
    \gamma_0=o(1),~~    \xi\geq \frac{1-\eps_0}{2\eps_0}.
\end{equation}

\item Case 2:
\begin{equation}
    (1-K\alpha\gamma_0)^2nI\goto\infty,~~ 
    \xi > \begin{dcases}
    (1-\eps_0)\left(\frac{1}{2\eps_0}\vee \frac{\alpha}{4(1-K\alpha\gamma_0)}\right), &\text{for~} K=2,\\
    (1-\eps_0)\left(\frac{1}{2\eps_0}\vee \frac{\alpha}{4\beta(1-K\alpha\gamma_0)}\right), &\text{for~} K\geq 3.
    \end{dcases}
\end{equation}
\end{itemize}
\end{condition}

Condition \ref{con:pqknowncon} yields the rapidly mixing property when the connectivity matrix $B$ is known.

\begin{thm}[Rapidly mixing]\label{thm:mixingknown}
Suppose we start the algorithm at $Z_0$, and Conditions (\ref{con:Z0}, \ref{con:pqknowncon}) hold. Then, the $\eps$-mixing time of the Metropolis-Hastings algorithm induced by Equation \eqref{eq:postknown} is upper bounded by
\[
	\tau_\eps(Z_0)\leq 4Kn^2 \max\left\{\gamma_0,n^{-\tau}\right\}\cdot \left(\xi \log \Pi(Z_0|A)^{-1}+\log (\eps^{-1})\right),
\]
with probability at least $1-C_7n^{-C_8}-\eta$ for some constants $C_7,C_8$, where $\tau$ is a sufficiently small constant, and $\eta$ is as defined in Condition \ref{con:Z0}.
\end{thm}
Compare the result with Theorem \ref{thm:mixing}, and we can see that Theorem \ref{thm:mixingknown} obtains the same upper bound with slightly weaker conditions for the initialization $Z_0$ and hyperparameter $\xi$.

\section{Numerical Results}\label{sec:NumResults}

In this section, we study the numerical performance of the Metropolis-Hastings algorithm.

\begin{description}[leftmargin = 0cm,labelsep = 0.5cm]
\item[Balanced networks.]  In this setting, we generate networks with 2500 nodes, and 5 communities, each of which consists of 500 nodes. Figure \ref{fig:general} shows the trajectories of the Markov chains (each denoted by a black line). By posterior strong consistency, the true label assignment receives the highest posterior probability (denoted by the red line), and the Markov chains converge rapidly to the stationarity (within $40n$ iterations), demonstrating the rapidly mixing property.

\begin{figure}[H]
\begin{center}
\begin{subfigure}[b]{0.48\textwidth}
    \includegraphics[width=\textwidth]{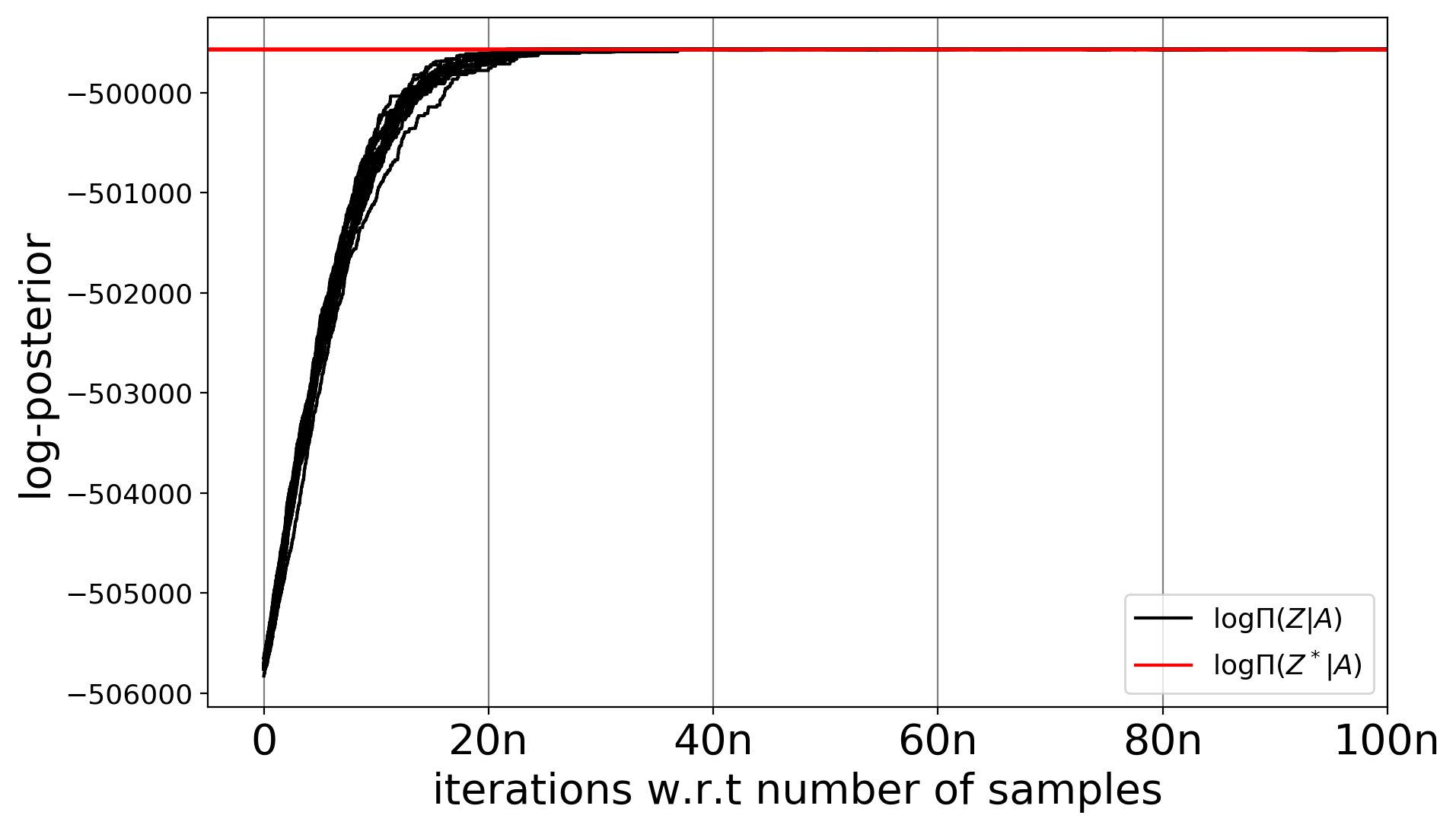}
    \caption{}
\end{subfigure}
\begin{subfigure}[b]{0.48\textwidth}
    \includegraphics[width=\textwidth]{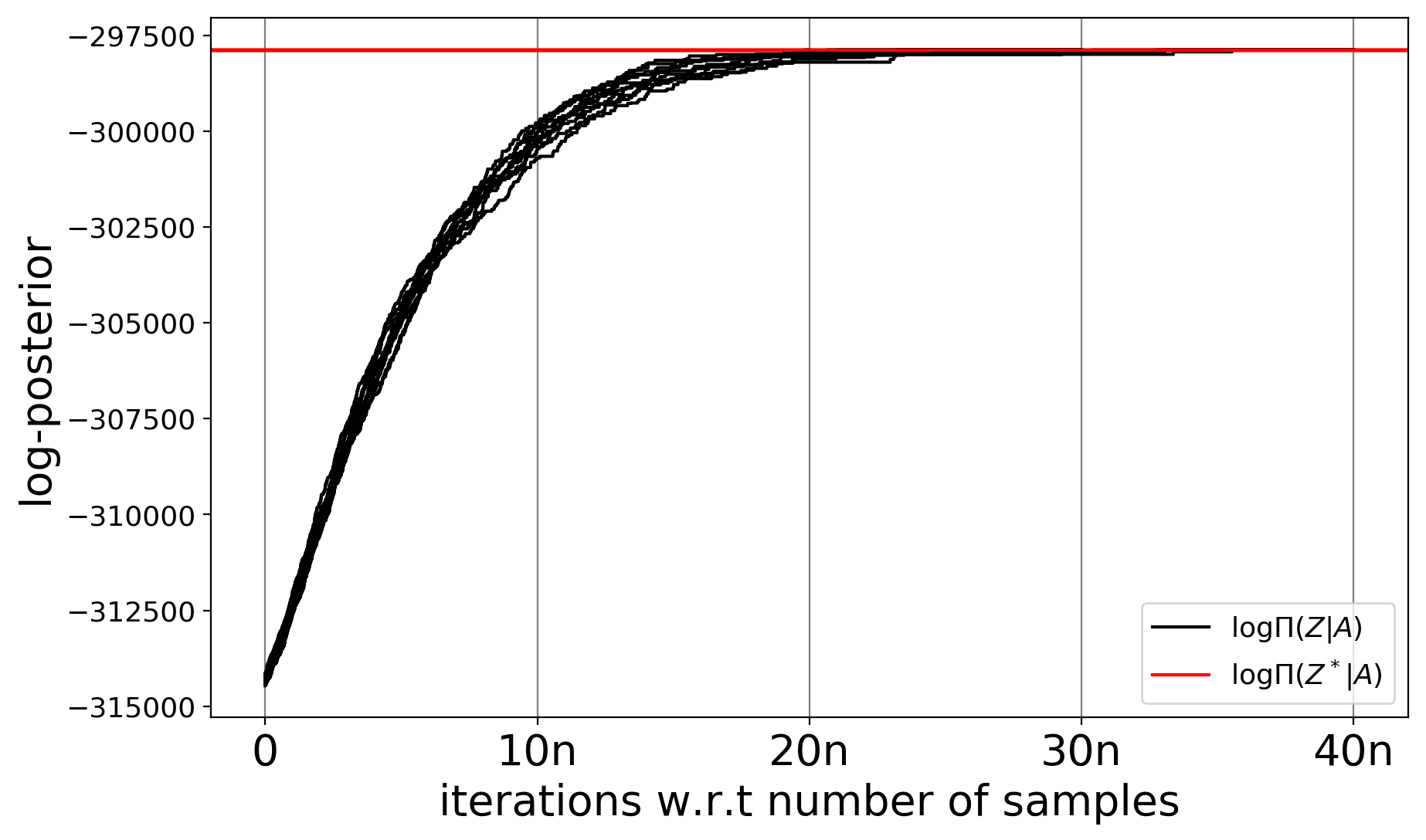}
    \caption{}
\end{subfigure}
\end{center}
\caption{Log-posterior probability versus the number of iterations. Each black curve corresponds to a trajectory of the chain (20 chains in total), and the red horizontal line represents the log-posterior probability at the true label assignment. (a) A network with $p=0.48$ and $q=0.32$. (b) A network with $p=0.3$ and $q=0.1$.}\label{fig:general}
\end{figure}

\item[Heterogeneous networks.] In this setting, we generate networks with 2000 nodes and 4 communities of sizes 200, 400, 600, and 800, respectively. The connectivity matrix is set as
\[
    B = \left(\begin{matrix}
        0.50 & 0.29 & 0.35 & 0.25\\
        0.29 & 0.45 & 0.25 & 0.30\\
        0.35 & 0.25 & 0.50 & 0.35\\
        0.25 & 0.30 & 0.35 & 0.45
    \end{matrix}\right).
\] 
The algorithm still performs well. As shown in Figure \ref{fig:hetero}, the posterior strong consistency still hold, and the Markov chains rapidly converge to the stationarity.

\begin{figure}[H]
\begin{center}
\includegraphics[width=0.6\textwidth]{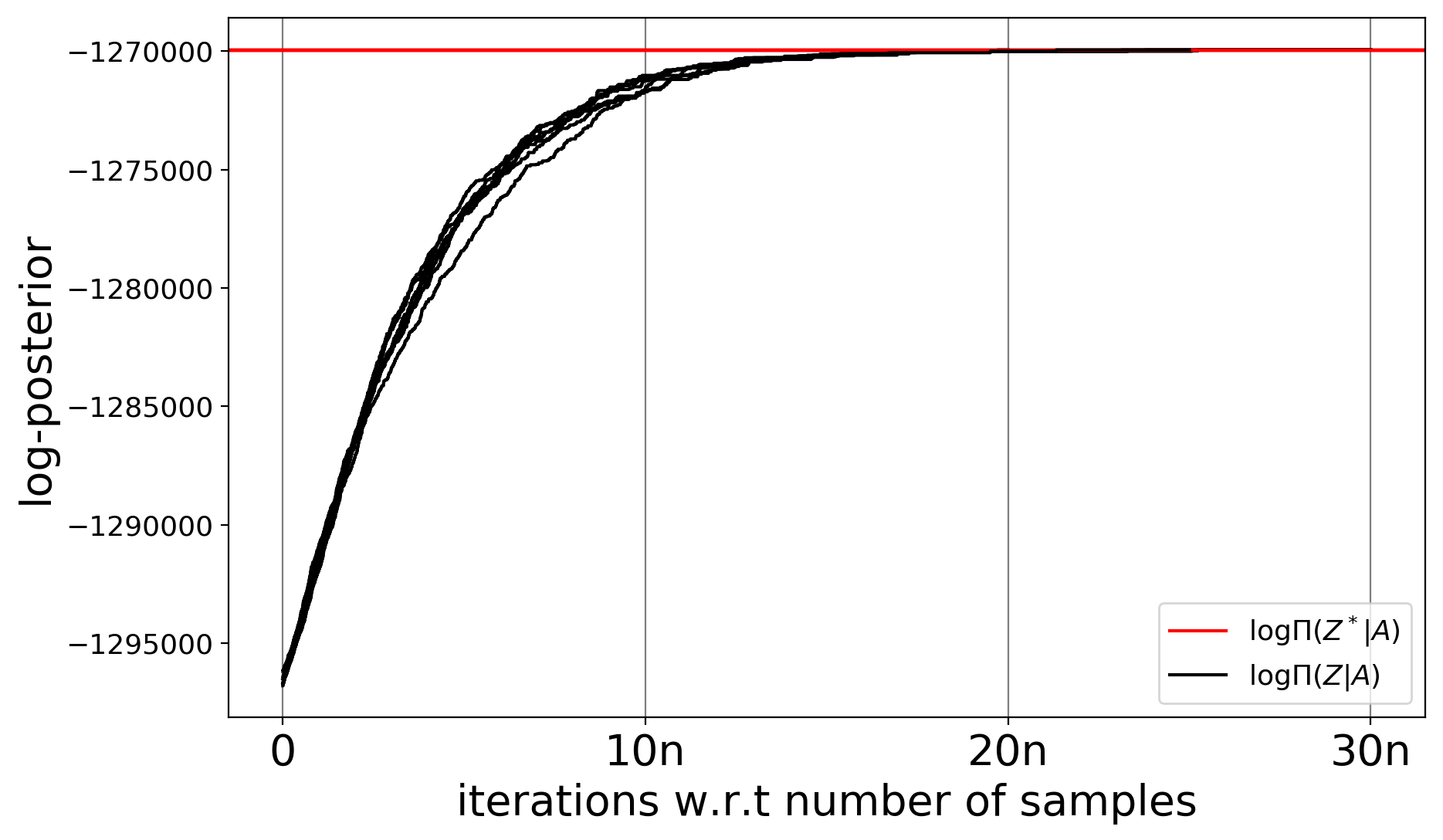}
\caption{Log-posterior probability versus the number of iterations. Each black curve corresponds to a trajectory of the chain (20 chains in total), and the red horizontal line represents the log-posterior probability at the true label assignment.}
\label{fig:hetero}
\end{center}
\end{figure}

\item[Necessity of the initialization condition.] We show that the initialization condition required by our main theorems is necessary by numerical experiments. Consider the network with two communities of size 270 and 460, and the connectivity probabilities are set to be $p=10^{-1}$, $q=10^{-8}$. The initial label assignment $Z_0$ satisfies $\ell(Z_0,Z^*) = (1-\eps)/2\alpha$, and then Condition \ref{con:pqknowncon} is equivalent to $\eps>0$ and $\eps^2nI\goto\infty$. In simulations, we run experiments for $\eps=0.2,0.1,-0.1,-0.2$, and the results are shown as below.

\begin{figure}[H]
\begin{center}
\begin{subfigure}[b]{0.48\textwidth}
    \includegraphics[width=\textwidth]{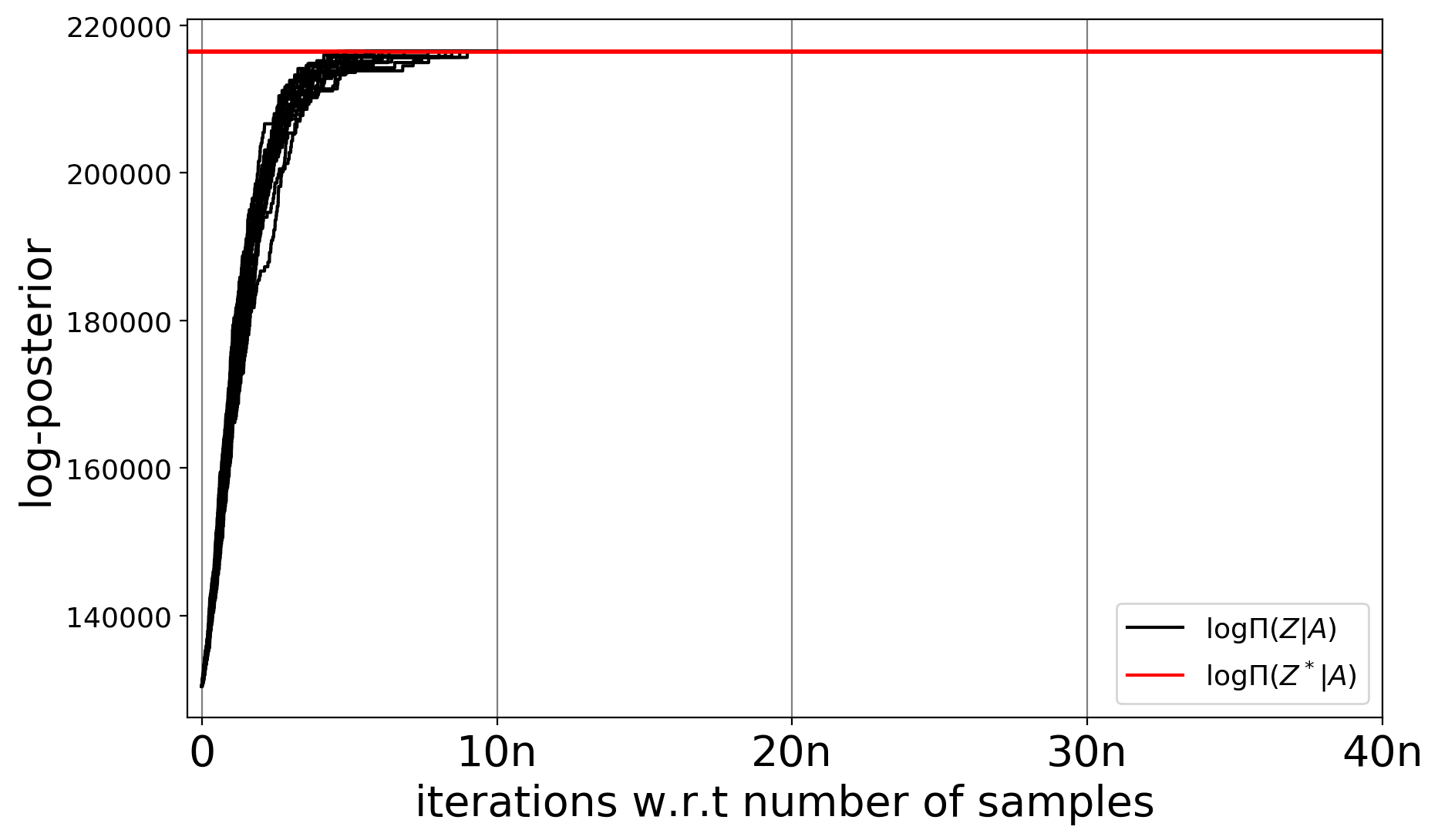}
    \caption{$\eps=0.2$}
\end{subfigure}
~
\begin{subfigure}[b]{0.48\textwidth}
    \includegraphics[width=\textwidth]{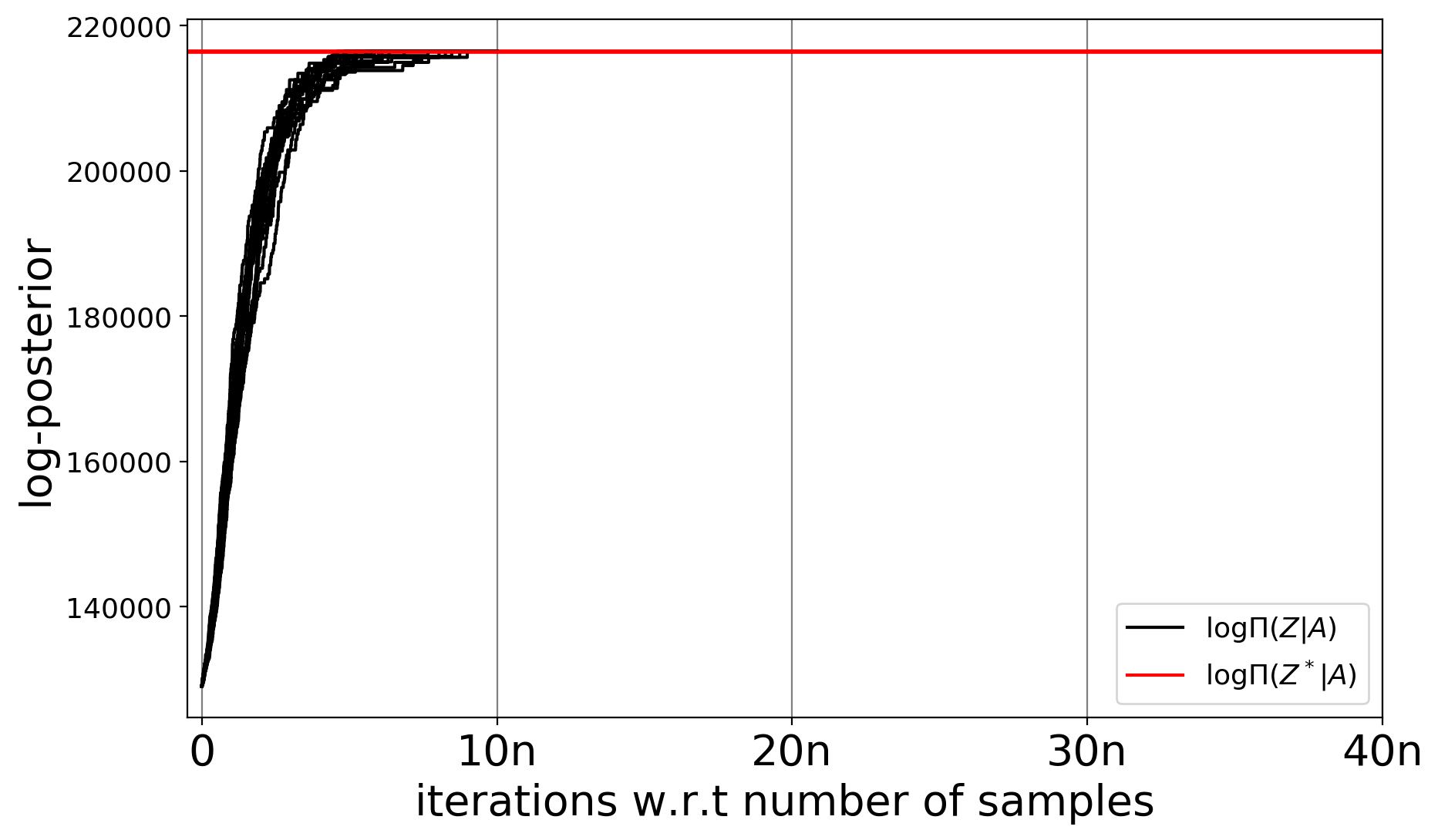}
    \caption{$\eps=0.1$}
\end{subfigure}
\\
\begin{subfigure}[b]{0.48\textwidth}
    \includegraphics[width=\textwidth]{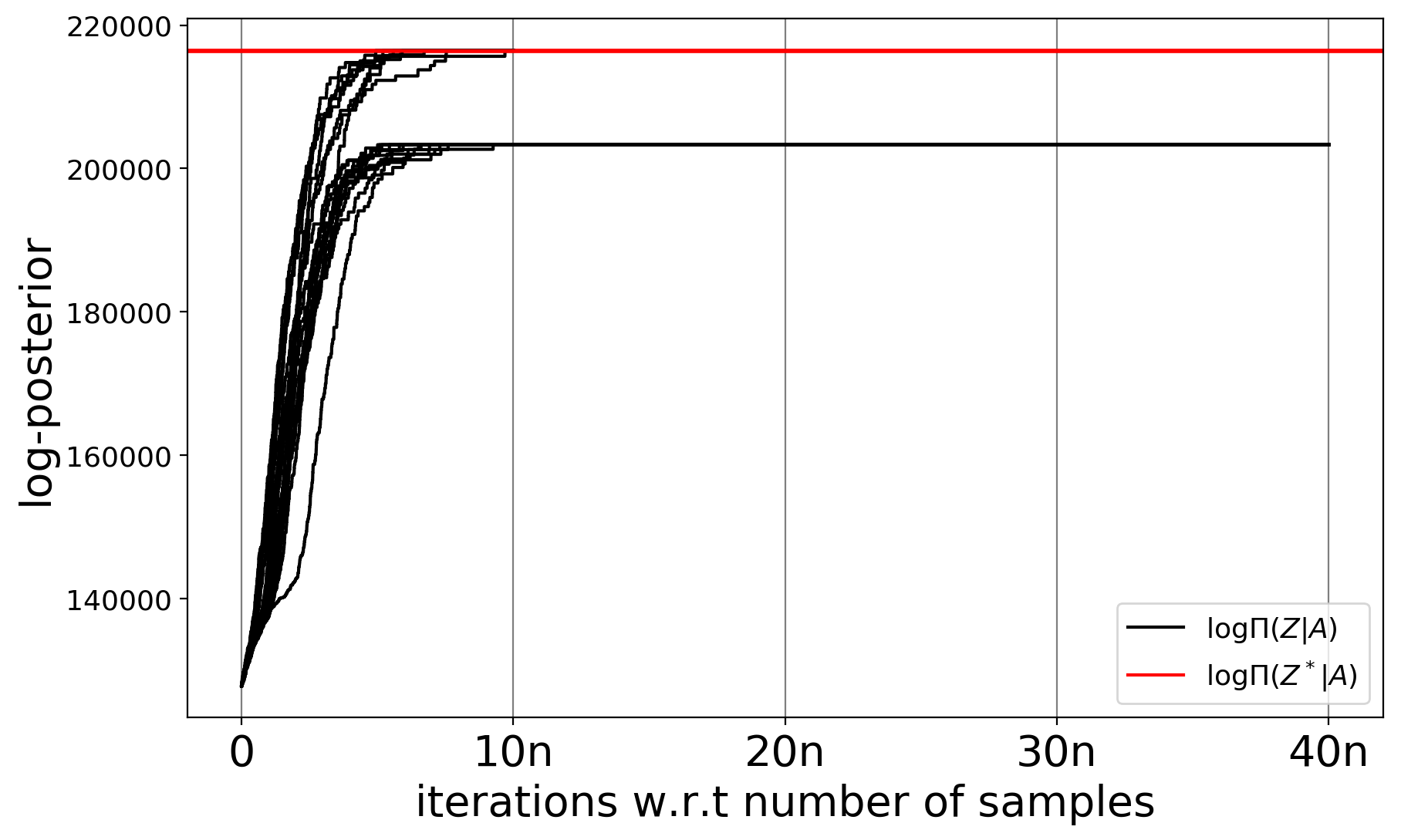}
    \caption{$\eps=-0.1$}
\end{subfigure}
~
\begin{subfigure}[b]{0.48\textwidth}
    \includegraphics[width=\textwidth]{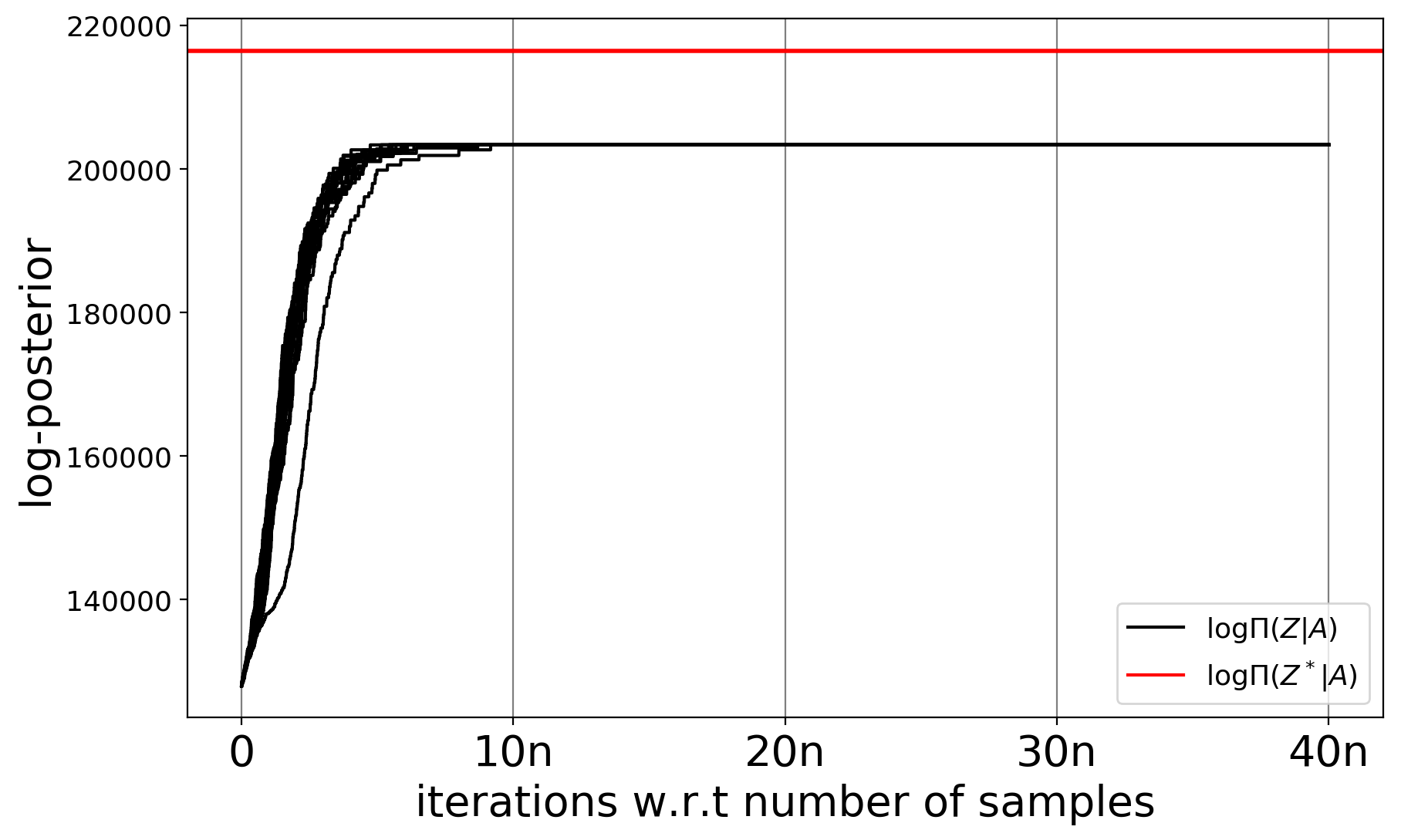}
    \caption{$\eps=-0.2$}
\end{subfigure}
\end{center}
\caption{Log-posterior probability versus the number of iterations. The initial label assignment $Z_0$ is constructed so that the labels of the community of size $270$ are all correct, and there are $n(1-\epsilon)/2\alpha$ labels in the community of size $460$ are incorrect. Each black curve corresponds to a trajectory of the chain (20 chains in total), and the red horizontal line represents the log-posterior probability at the true label assignment.}
\label{fig:bad}
\end{figure}
Figure \ref{fig:bad} shows that when $\eps<0$, it is very likely for the algorithm to get stuck at some local maximum, and does not converge to the stationary distribution.

\item[Fundamental limit of the signal condition.] We check that the fundamental limit of the signal condition can be achieved by the Metropolis-Hastings algorithm. We generate homogeneous networks with 1000 nodes and 2000 nodes, and each has two communities of equal sizes. Figure \ref{fig:limit} is the heatmap of the number of misclassified samples, where every rectangular block represents one setting with different values of $p$ and $q$. In each setting, we run 20 experiments with independent initializations and adjacency matrices, and the value of each block is the average number of misclassified samples in the 20 experiments. Figure \ref{fig:limit} shows that when $nI>2\log n$, we are able to exactly recover the underlying true label assignment, and the result of simulation coincides with the posterior strong consistency property in Section \ref{sec:postStrongConsis}.

\begin{figure}[H]
\begin{subfigure}[b]{0.48\textwidth}
    \includegraphics[width=\textwidth]{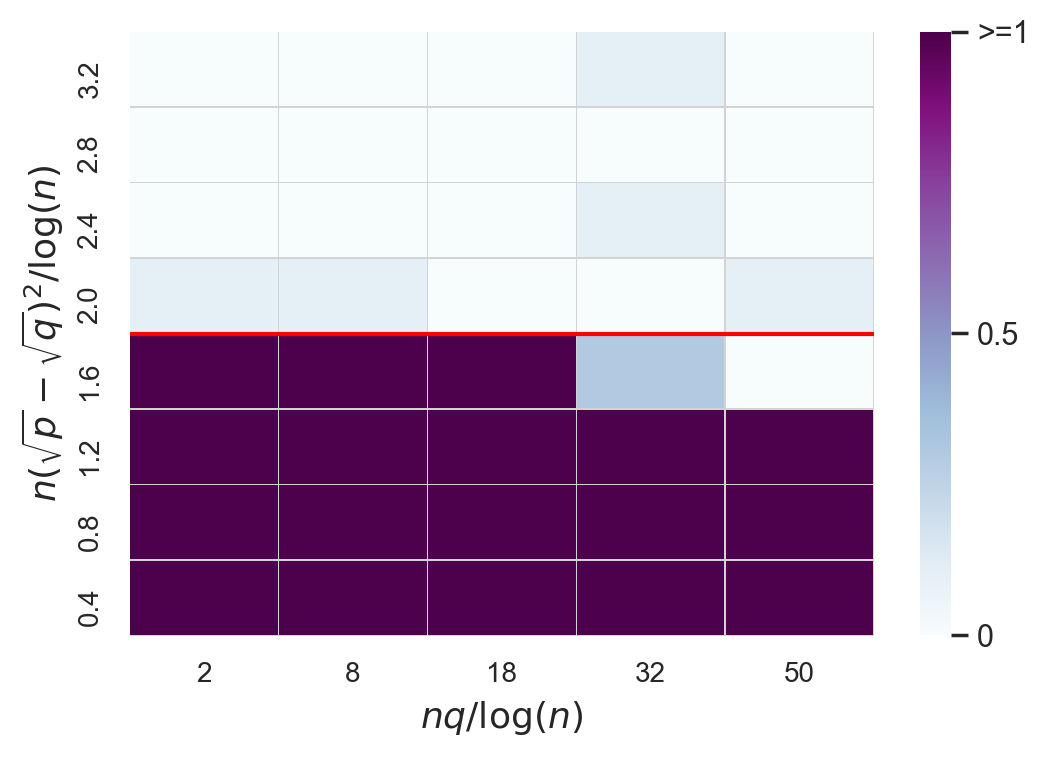}
    \caption{Network with 1000 nodes}
\end{subfigure}
\begin{subfigure}[b]{0.48\textwidth}
    \includegraphics[width=\textwidth]{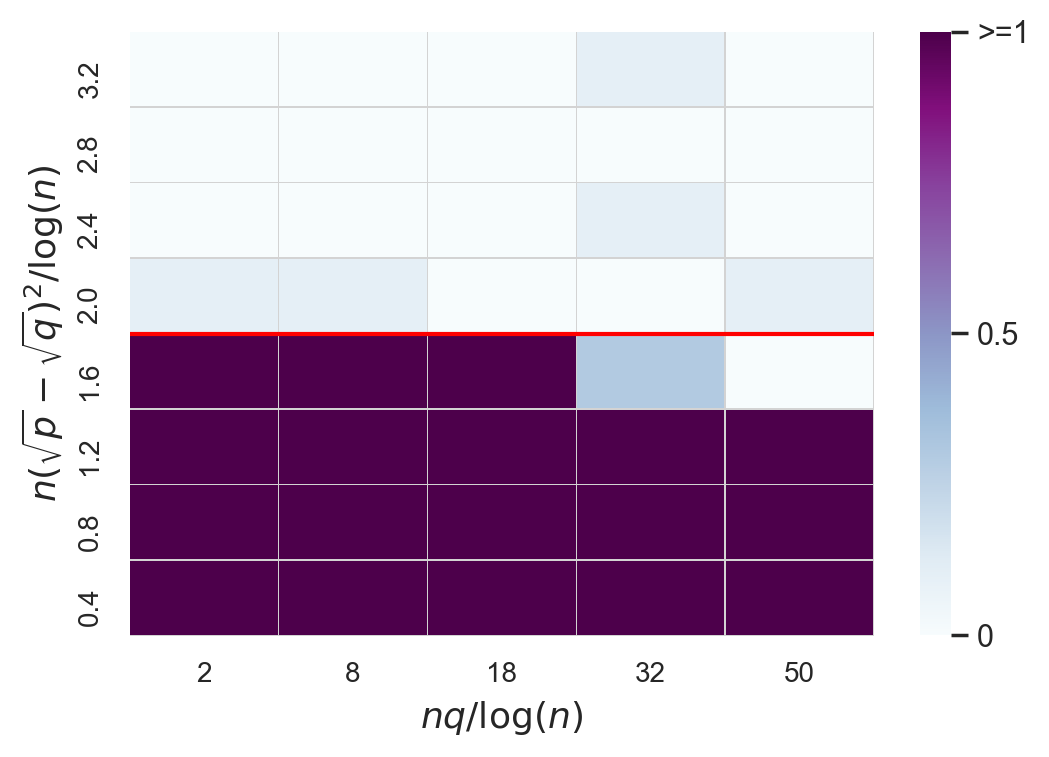}
    \caption{Network with 2000 nodes}
\end{subfigure}
\caption{The heatmap of the number of misclassified samples. The red line in each plot represents the fundamental limit with $K=2$.}
\label{fig:limit}
\end{figure}

\end{description}

\section{Proofs}\label{sec:mainProofs}

The posterior strong consistency property, Theorem \ref{thm:postcontract} and Theorem \ref{thm:postknown}, is proved in Section \ref{sec:proof_contract}. The main result of the paper, Theorem \ref{thm:mixing}, is proved in Section \ref{sec:proof_mixing}.

\subsection{Proof of posterior strong consistency}\label{sec:proof_contract}

We first state the proof in the case where the connectivity probability matrix $B$ is known (Theorem \ref{thm:mixingknown}). Then, by similar techniques, we have the result of Theorem \ref{thm:postcontract}. To distinguish the two cases, we denote the posterior distribution as $\Pi_0(\cdot|A)$ with a known connectivity probability matrix. In this section, we use $d(Z,Z^*)= n \ell(Z,Z^*) = m(Z)$ to denote the number of mistakes for the label assignment $Z$. For simplicity, we also write $m$ for $m(Z)$ with a slight abuse of notation.

\subsubsection{Proof of Theorem \ref{thm:postknown}}

We first state a lemma in order to prove the theorem.

\begin{lemma}[Lemma 5.4 in \cite{zhang2016minimax}]\label{lm:postcon}
	For any constants $\alpha>\beta\geq 1$, let $Z\in S_\alpha$ be an arbitrary assignment satisfying that $d(Z,Z^*)= m$ with $0<m<n$. Then, for the $\Pi_0(Z|A)=\Pi(Z|A)$ defined in \eqref{eq:postknown}, we have
	\[\Prob\left\{\Pi_0(Z|A)>\Pi_0(Z^*|A)\right\}\leq\begin{dcases}
	\exp\bp{-(\bar nm -m^2)I}, & m\leq \frac{n}{2K},\\∫
	\exp\bp{-d_{\alpha,\beta}\frac{nmI}{K}}, & m>\frac{n}{2K},
	\end{dcases}\]
	where $d_{\alpha,\beta}$ is some positive constant that only depends on $\alpha,\beta$.
\end{lemma}
\begin{proof}[Proof of Theorem \ref{thm:postknown}]
Recall that for any $Z'\in \Gamma(Z)$, we have $\Pi_0(Z'|A)=\Pi_0(Z|A)$.\ignore{To avoid ambiguity, we condense all state space into small subset ignoring label permutation, simply write $\Pi(\Gamma(Z)|A)$ as $\Pi(Z|A)$ for clearer presentation, and denote $\Gamma(Z)\neq \Gamma(Z^*)$ by $Z\neq Z^*$. This trick is widely used in this paper.} For each $Z\in S_\alpha$, let $G_Z = \left\{\Pi_0(Z|A)>\Pi_0(Z^*|A)\right\}$, and define $G = \cup_{Z\in S_\alpha} G_Z$.  Let $P_Z$ denote the likelihood function for the assignment $Z$. With the uniform prior on $S_{\alpha}$, we have
\begin{align*}
\Expect{\Pi_0(Z\not\in\Gamma(Z^*)|A)} =& P_{Z^*} \Pi_0(Z\not\in\Gamma(Z^*)|A) \indc{G^c} + P_{Z^*} \Pi_0(Z\not\in\Gamma(Z^*)|A) \indc{G}&\\[7pt]
\leq &P_{Z^*} \sum_{Z\not\in\Gamma(Z^*)} \frac{P_Z}{P_{Z^*}+\sum_{Z} P_Z} \indc{G_Z^c}+ \sum_{Z\not\in\Gamma(Z^*)} P_{Z^*}(G_Z)&\\[7pt]
\leq &P_{Z^*} \sum_{Z\not\in\Gamma(Z^*)} \frac{P_Z}{P_{Z^*}} \indc{G_Z^c}+ \sum_{Z\not\in\Gamma(Z^*)} P_{Z^*}(G_Z)&\\[7pt]
= & \sum_{Z\not\in\Gamma(Z^*)} P_Z(G_Z^c) + P_{Z^*}(G_Z) &\\[7pt]
= & 2\sum_{Z\not\in\Gamma(Z^*)} \Prob\left\{\Pi_0(Z|A)>\Pi_0(Z^*|A)\right\},
\end{align*}
where the last inequality is due to symmetry. We also have
\[
	\left|{\{\Gamma: \exists Z \in \Gamma, s.t.~d(Z,Z^*)=m\}}\right|\leq  {n\choose m} (K-1)^m\leq \min\left\{\bp{\frac{enK}{m}}^m,K^n\right\}.
\]

Note that $\{Z:Z\not\in \Gamma(Z^*)\}$ is equivalent to set $\{Z:m(Z)\geq 1\}$. With the condition that $\bar nI>\log n$, it follows by Lemma \ref{lm:postcon} that
\begin{align}
\Expect{\Pi_0(Z\not \in\Gamma(Z^*)|A)}&\leq 2\sum_{1\leq m\leq n/2K}{n\choose m}K^m\exp\bp{-(\bar n m-m^2)I} + 2\sum_{m>n/2K}K^n\exp\bp{-d_{\alpha,\beta} mnI/K} \nonumber\\
&\leq 2\sum_{1\leq m\leq n/2K}{n\choose m}K^m\exp\bp{-(\bar n m-m^2)I} + 2nK^n\exp\bp{-Cn^2I}\label{eq:postconsisknown}
\end{align}
for some constant $C$. We proceed to upper bound the first term in \eqref{eq:postconsisknown}. It follows that
\begin{align*}
\sum_{1\leq m\leq {n}/{2K}}{n\choose m}K^m \exp\bp{-(\bar nm-m^2)I}\leq \sum_{1\leq m\leq {n}/{2K}}\bp{{enK}}^m\exp\bp{-(\bar nm-m^2)I} = \sum_{m}P_m
\end{align*}
where $P_m = \bp{enK}^m\exp\bp{-(\bar nm-m^2)I}$. The ratio of $P_m$ and $P_1$ is calculated as
\begin{align*}
\frac{P_m}{P_1} &= {(enK)^{m-1}}\exp(-\bar nI(m-1)+(m^2-1)I) = \bp{enK\exp\bp{-\bar nI+(m+1)I}}^{m-1}.
\end{align*}
Define $m' = \eps' n$ for some positive sequence $\eps'=\eps_n'$ with $\eps'\goto 0$ and $\eps'nI\goto \infty$. Then, $\sum_{1\leq m\leq n/2K}P_m$ can be split into summation of $\sum_{1\leq m< m'}P_m$ and $\sum_{m'\leq m\leq n/2K}P_m$, where
\begin{align*}
\sum_{m=1}^{m'-1}P_m &= P_1\sum_{m=1}^{m'-1}\frac{P_m}{P_1}\leq P_1\sum_{m=1}^{m'-1}(enK\exp\bp{-\bar nI+m'I})^{m-1}\\
&\leq enK\exp\bp{-\bar nI+I}\cdot(1+2enK\exp(-\bar nI+\eps'nI)),
\end{align*}
and there exists some constant $C$ such that
\begin{align*}
\sum_{m'<m\leq n/2K}P_m \leq nK^n\exp(-\eps'(C-\eps')n^2I)\leq \exp(-n).
\end{align*}
Hence, by combining all parts and based on the condition that $\bar nI>\log n$, we have $\Pi_0(Z\not \in \Gamma (Z^*)|A)\leq Cn\exp(-\bar nI)$ for some constant $C$ and for a large $n$.

\end{proof}

\subsubsection{Proof of Theorem \ref{thm:postcontract}}
\begin{lemma}\label{lm:postlargemistakes}
	Let $Z\in S_\alpha$ be an arbitrary assignment with $d(Z,Z^*)=m>0$. If $p,q\goto 0$ and $p\asymp q$, there exists some positive sequence $\gamma =\gamma_n$ with $\gamma\goto 0$ and $\gamma^2 nI\goto \infty$, such that for the $\Pi(Z|A)$ defined in \eqref{eq:postunknown}, we have
	\begin{align*}
	&\Prob\left\{{\max_{Z\in S_\alpha:m>\gamma n}\log\frac{\Pi(Z|A)}{\Pi(Z^*|A)}\geq -C_1\gamma n^2I}\right\}\leq 4\exp(-n),
	\end{align*}
and
	\begin{align*}
	&\Prob\left\{{\max_{Z\in S_\alpha:m\leq\gamma n}\log\frac{\Pi(Z|A)}{\Pi(Z^*|A)}-\log \frac{\Pi_0(Z|A)}{\Pi_0(Z^*|A)}-C_2\gamma mnI>0}\right\}\leq n\exp\bp{-(1-o(1))\bar n I},
	\end{align*}
for some constants $C_1,C_2$. Here, $\Pi_0(Z|A)$ is the posterior probability with known connectivity probabilities.
\ignore{Here, $C_1 = \frac{\alpha^4}{16K^4\beta^2}\bp{\gamma^2\wedge 1/\beta^2}$, $C_2 = 4c_0\alpha K^2(\alpha+\beta+2K)$, and $(\sqrt{p}+\sqrt{q})^2/q\leq c_0$.}
\end{lemma}

The proof of Lemma \ref{lm:postlargemistakes} is deferred to Section \ref{sec:postContrSec}. We now state another lemma that is based on Proposition 5.1 in \cite{zhang2016minimax}.
\begin{lemma}[Proposition 5.1 in \cite{zhang2016minimax}]\label{lm:exp}
	For any $Z\in S_\alpha$ where $d(Z,Z^*)=m<{n}/{2K}$,
	\begin{align*}
	\Expect{\sqrt{\frac{\Pi_0(Z|A)}{\Pi_0(Z^*|A)}}}\leq \exp\bp{-\bar nmI+m^2I}.
	\end{align*}
\end{lemma}
\begin{proof}[Proof of Theorem \ref{thm:postcontract}]

With Lemma \ref{lm:postlargemistakes} and Lemma \ref{lm:exp}, we divide $S_\alpha$ into a large mistake region and a small mistake region according to whether $m>\gamma n$, where $\gamma$ is a positive sequence defined in Lemma \ref{lm:postlargemistakes}.

\begin{description}[leftmargin = 0cm,labelsep = 0.5cm]
\item[Large mistake region.] For $m>\gamma n$, by Lemma \ref{lm:postlargemistakes}, with probability at least $1-4\exp(-n)$,
\begin{align*}
\sum_{Z\in S_\alpha:m>\gamma n}\frac{\Pi(Z|A)}{\Pi(Z^*|A)}\leq nK^n\exp(-C_1\gamma n^2I)\leq \exp(-n).
\end{align*}
Denote $\cE = \{\sum_{Z\in S_\alpha:m>\gamma n}\frac{\Pi(Z|A)}{\Pi(Z^*|A)}\leq\exp(-n)\}$, and it follows directly that
\begin{align*}
\Expect\left[\sum_{Z\in S_\alpha :m>\gamma n}\Pi(Z|A)\right] \leq \Expect\left[\Pi(Z:m>\gamma n|A)\indc{\cE}\right] + \PP{\cE^c} \leq 5\exp(-n).
\end{align*}
\item[Small mistake region.] For $m\leq \gamma n$, let $G_Z = \{\Pi_0(Z|A)>\Pi_0(Z^*|A)\}$. Let $\theta$ denote all unknown parameters and $\theta_0$ denote the underlying true parameters respectively. \ignore{Let $\tP_e$ be the likelihood that integrates out the parameters for given label assignment $e$.}Define ${\cal F}=\{\max_{Z\in S_\alpha:m\leq \gamma n}\log\frac{\Pi(Z|A)}{\Pi(Z^*|A)}-\log \frac{\Pi_0(Z|A)}{\Pi_0(Z^*|A)}-C_2\gamma mnI>0\}$ as in Lemma \ref{lm:postlargemistakes}. Then, we have
\begin{align*}
&\Expect{\Pi(Z: 1\leq m\leq \gamma n|A)} \\
\leq &P_{Z^*,\theta_0}\Pi(Z: 1\leq m\leq \gamma n|A) \indc{{\cal F}^c}+\Prob\left\{{\cal F}\right\}\\[7pt]
\leq &\sum_{Z:1\leq m\leq \gamma n} P_{Z^*,\theta_0}\frac{\Pi(Z|A)}{\Pi(Z^*|A)}\indc{G_Z^c,{\cal F}^c} + \sum_{Z:1\leq m\leq \gamma n}P_{Z^*,\theta_0}\indc{G_Z} +\Prob\left\{{\cal F}\right\}\\[7pt]
\leq &\sum_{Z:1\leq m\leq \gamma n} P_{Z^*,\theta_0}\frac{\Pi_0(Z|A)}{\Pi_0(Z^*|A)}\indc{G_Z^c}\exp\bp{C_2\gamma mnI} + \sum_{Z:1\leq m\leq \gamma n}P_{Z^*,\theta_0}\indc{G_Z} +\Prob\left\{{\cal F}\right\}\\[7pt]
\leq &\sum_{Z:1\leq m\leq \gamma n}\exp\bp{C_2\gamma mnI}P_{Z,\theta_0}\indc{G_Z^c}+\sum_{Z:1\leq m\leq \gamma n}P_{Z^*,\theta_0}\indc{G_Z} +\Prob\left\{{\cal F}\right\}\\[7pt]
\leq & n\exp\bp{-(1-o(1))\bar nI)}.
\end{align*}
Recall that $\Pi_0(\cdot|A)$ denotes the posterior distribution with knowledge of the connectivity probabilities. The second inequality is due to $\Pi(Z^*|A)\leq 1$. The third inequality is due to the definition of the event $\cal F$. The last two inequalities hold by Lemma \ref{lm:postcon} and symmetry.
\end{description}
Combine the two regions, and then
\[\Expect{\Pi(Z\not\in \Gamma(Z^*)|A)}\leq n\exp(-(1-o(1))\bar nI).\]
The proof is complete.
\end{proof}

\subsection{Proofs of Theorem \ref{thm:mixing} and Theorem \ref{thm:mixingknown}}\label{sec:proof_mixing}

\subsubsection{Backgrounds on mixing time} 
Consider a reversible, irreducible, and aperiodic Markov chain on a discrete space $\Omega$ that is completely specified by a transition matrix $P\in [0,1]^{|\Omega|\times|\Omega|}$ with stationary distribution $\Pi$. Let $\omega\in \Omega$ be the initial state of the chain, and then the total variation distance to the stationary distribution after $t$ iterations is 
\begin{align*}
\Delta_\omega (t) = \TVdiff{P^t(\omega, \cdot)}{\Pi},
\end{align*}
where $P^t(\omega,\cdot)$ is the distribution of the chain after $t$ iterations. The $\eps$-mixing time starting at $\omega$ is given by 
\begin{align*}
\tau_\eps(\omega) = \min\left\{t\in {\mathbb N}: \Delta_\omega(t')\leq \eps \text{~for all $t'\geq t$}
\right\}.
\end{align*}
With this notation, we say a Markov chain is rapidly mixing if $\tau_\eps(\omega)$ is $O\bp{\text{poly}(\log(|\Omega|/\eps))}$ in the case where $|\Omega|$ scales exponentially to the problem size $n$. This means we only need to update the Markov chain for poly$(n)$ steps in order to obtain \textit{good} samples from the stationary distribution. The explicit bound for the mixing time through the spectral gap is
\begin{align}\label{eq:boundmixingfomula}
\tau_\eps(\omega )\leq \frac{-\log{\Pi(\omega)+\log(1/\eps)}}{Gap(P)},
\end{align}
where $Gap(P)$ represents the spectral gap of the transition matrix $P$, defined by $Gap(P) = 1-\max\{|\lambda_{2}(P)|,|\lambda_{\min}(P)|\}$, where $\lambda_2(P)$, $\lambda_{\min}(P)$ are the second largest and the smallest eigenvalues of the transition matrix $P$. See the paper \cite{woodard2013convergence} for this bound.

\subsubsection{Preparation}
Suppose $P(\cdot,\cdot)$ in \eqref{eq:tran_mat} is the transition matrix introduced in Algorithm \ref{algo:Ag1} defined in the label assignment space $S_\alpha$, and $\check{P}(\cdot,\cdot)$ in \eqref{eq:tran_mat_gamma} is the transition matrix of $\{\Gamma_t\}_{t\geq 0}$ defined in the clustering space $\check{S}_{\alpha}=\{\Gamma(Z):Z\in S_\alpha\}$.
 The stationary distribution for $P$ and $\check{P}$ are denoted as $\tPi$ and $\check{\Pi}$ respectively. We require a good initializer, and use the following lemma to guarantee that all possible states visited by the algorithm remain in a good region with high probability.

\begin{lemma}\label{lm:stayingood}
	Suppose we start at a fixed initializer $Z_0$ with $\ell(Z_0,Z^*)\leq \gamma_0$ where $\gamma_0$ satisfies Condition \ref{con:gamma0}, \ref{con:gamma0xi_small}, or \ref{con:pqknowncon}. Then, the number of misclassified nodes in any polynomial running time can be upper bounded by 
	\begin{equation}\label{eq:stayingood}
		m = n\cdot\ell(Z,Z^*)\leq n\max\left\{\gamma_0,n^{-\tau}\right\} +\log^2n,
	\end{equation}
	with probability at least $1-\exp(-\log^2 n)$, where $\tau>0$ is a sufficiently small constant.
\end{lemma}

The proof of Lemma \ref{lm:stayingood} is deferred to Section \ref{sec:stayingood}. Note that Lemma \ref{lm:stayingood} is stated conditioning on a fixed initial label assignment $Z_0$ with $\ell(Z_0,Z^*)\leq \gamma_0$. This is slightly different from the original initialization conditions where we use $Z_0$ dependent on data. A simple union bound will lead to the final conclusion. Lemma \ref{lm:stayingood} quantifies the maximum possible number of classification mistakes when starting at a good initializer. Here, $(\log n)^2$ is chosen for simplicity and can be replaced by any sequence $\nu_n \gg \log n$.

Let $\cG(\gamma_0)$ denote a good region with respect to the initial misclassification proportion, defined by
\begin{equation}\label{eq:defineGoodRegion}
    \cG(\gamma_0) = \{Z\in S_\alpha: m\leq n\max\{\gamma_0,n^{-\tau}\}+\log^2 n\},
\end{equation}
where $\tau$ is a sufficiently small constant. Accordingly, we can define a good region in the clustering space as
\[
    \check{\cG}(\gamma_0) = \{\Gamma(Z):Z\in \cG(\gamma_0)\},
\]
and Lemma \ref{lm:stayingood} ensures that for any $T$ that is a polynomial of $n$, $\{\Gamma_t\}_{0\leq t\leq T}$ stays inside $\check\cG(\gamma_0)$ with high probability. Sometimes we write $\cG(\gamma_0)$ and $\check{\cG}(\gamma_0)$ as $\cG$ and $\check{\cG}$ for simplicity. Then, we modify the distributions and transition matrices according to the regions $\cG$ and $\check\cG$. Denote the modified distributions as $\tPi_g(Z|A)\propto \Pi^\xi(Z|A)\indc{Z\in\cG}$ for all $Z\in S_\alpha$, and $\check\Pi_g(\Gamma|A)\propto \check\Pi(\Gamma|A)\indc{\Gamma\in\check\cG}$ for all $\Gamma\in \check S_\alpha$. Define in the label assignment space the new transition matrix $P_g(\cdot,\cdot)$ corresponding to $\tPi_g(\cdot |A)$, by replacing $\Pi^\xi(\cdot |A)$ with $\tPi_g(\cdot |A)$ in \eqref{eq:tran_mat}. Define in the clustering space the new transition matrix $\check P_g(\cdot,\cdot)$ corresponding to $\check \Pi_g(\cdot|A)$, by replacing $\Pi(\cdot|A)$ with $\Pi(\cdot|A)\indc{\cdot\in \check\cG}$ in \eqref{eq:tran_mat_gamma}.

With these notations, we proceed to bound the total variation error between the distribution of $\Gamma_T$ and $\check\Pi(\cdot|A)$ after $T$ steps for some $T$ that is a polynomial of $n$.

\begin{lemma}[TV difference]\label{lm:TV}
	\[\mathbb{E}{\TVdiff{\check\Pi_g(\cdot |A)}{\check\Pi(\cdot |A)}}\leq n\exp(-(1-o(1))\bar nI).\]
\end{lemma}
\begin{proof} We have
	\begin{align*}
	\mathbb{E}{\TVdiff{\check\Pi_g(\cdot |A)}{\check\Pi(\cdot |A)}}\leq 2\mathbb{E}\check\Pi(\Gamma\not\in\check \cG|A)\leq 2\mathbb{E}{\check\Pi(\Gamma\neq \Gamma(Z^*)|A)}
	\leq n\exp(-(1-o(1))\bar nI),
	\end{align*}
where the first inequality is due to Lemma \ref{supp-lm:TV}. The second inequality is due to the definition of $\check\cG$. The last inequality directly follows by Theorem \ref{thm:postcontract} or Theorem \ref{thm:mixingknown}, and the condition that $\xi\geq 1$.
\end{proof}
Thus, by triangle inequality, we can decompose the total variation bound at time $T$ as
\begin{equation}\label{eq:TV}
\begin{split}
    &\TVdiff{\check P^T(\Gamma_0,\cdot)}{\check\Pi(\cdot|A)}\\
    \leq &\TVdiff{\check P^T(\Gamma_0,\cdot)}{\check P_g^T(\Gamma_0,\cdot)}+\TVdiff{\check P_g^T(\Gamma_0,\cdot)}{\check \Pi_g(\cdot|A)}+\TVdiff{\check\Pi_g}{\check \Pi(\cdot |A)},
\end{split}
\end{equation}
where $T$ is the number of iterations. Lemma \ref{lm:stayingood} implies that the first term is 0 with high probability for $T\leq \text{poly}(n)$, since the algorithm stays in the region $\cG(\gamma_0)$. The third term can be upper bounded by Lemma \ref{lm:TV}. Therefore, the remaining proof is to adopt the canonical path approach to bound the second term in \eqref{eq:TV}.

For the purpose of the proof, we replace the transition matrix $\check P_g$ by its lazy version, which has a probability of 1/2 at staying at its current state, and another probability of 1/2 at updating the state. The same technique can be also found in \cite{yang2016computational,dabbs2009markov,berestycki2016mixing,montenegro2006mathematical}. It is worth noting that this technique is only for the proof.

\subsubsection{Canonical path}
Given an ergodic Markov chain $\cal C$ induced by the lazy transition matrix $\check P_g$ in the discrete state space $\check \cG$, we define a weighted directed graph $G(\mathcal{C}) = (V,E)$, where the vertex set $V=\check \cG$ and an edge between an ordered pair $(\Gamma,\Gamma')$ is included in $E$ with weight $Q(\Gamma,\Gamma') = \check\Pi_g(\Gamma)\check P_g(\Gamma,\Gamma')$ whenever $\check P_g(\Gamma,\Gamma')>0$. A canonical path ensemble $\cT$ is a collection of simple paths $\{\cT_{x,y}\}$ in the graph $G(\mathcal{C})$, one between each ordered pair $(x,y)$ of distinct vertices. As shown in \cite{sinclair1992improved}, for any choice of canonical path $\cT$, the spectral gap of the transition matrix $\check P_g$ can be lower bounded by
\begin{align*}
Gap(\check P_g)\geq \frac{1}{\rho(\cT)\ell(\cT)},
\end{align*}
where $\ell(\cT)$ is the length of the longest path in $\cT$, and $\rho(\cT)$ is the \textit{path congestion} parameter defined by $\rho(\cT) = \max_{(\Gamma,\Gamma')\in E(\mathcal{C})}\frac{1}{Q(\Gamma,\Gamma')}\sum_{\cT_{x,y}\ni (\Gamma,\Gamma')}\check \Pi_g(x)\check \Pi_g(y)$.

In order to apply the approach, we need to construct an appropriate canonical path ensemble $\cT$ in the discrete state space $\check\cG$. First, we construct a unique canonical path from any clustering $\Gamma$ to the underlying true clustering $\Gamma^*$, where $\Gamma^*=\Gamma(Z^*)$. Suppose for any label assignment $Z$, we define a function $g:S_\alpha\goto S_\alpha$ such that
\begin{equation}\label{eq:defGZ}
g(Z)= \begin{dcases}
\argmax_{Z'\in \cB(Z)\cap S_\alpha}\Pi(Z'|A), & ~\text{if~}Z\not \in \Gamma(Z^*),\\
Z, &~\text{if }Z\in \Gamma(Z^*),
\end{dcases}
\end{equation}
where
\begin{align*}
\cB(Z) &= \{Z': d(Z',Z^*) = d(Z,Z^*)-1, H(Z,Z')=1\}.
\end{align*}
We use $\cB(Z)\cap S_\alpha$ to denote the set of available states that have fewer mistakes than the current state $Z$. By Lemma \ref{lm:size}, $\cB(Z)\cap S_\alpha$ is always non-empty for $Z\not \in \Gamma(Z^*)$. Here, $g(Z)$ is the \textit{optimal} state in $\cB(Z)$ in the sense that $g(Z)$ maximizes the posterior distribution. Then, for any current state $\Gamma\in \check S_\alpha$, we define the next state $\check g(\Gamma)$ to be
\[
    \check g(\Gamma) = \{g(Z):Z\in \Gamma\}
\]
Since for any $Z\in \Gamma$, $g(Z)$ gives the equivalent result, and thus $\check g(\Gamma)\in \check S_\alpha$ is well defined. Hence, the canonical path from any current state $\Gamma\neq\Gamma(Z^*)$ is a greedy path, and the number of mistakes keeps decreasing along the canonical path.

Second, we construct a unique canonical path between any two states $\Gamma$ and $\wt \Gamma$, defined by $\cT_{\Gamma,\wt \Gamma} = \cT_{\Gamma,\Gamma^*}-\cT_{\wt \Gamma,\Gamma^*}$. The operations on simple paths are the same as defined in \cite{abbe2016community}. It is worth noting that the construction of the canonical path is data dependent, i.e., for different adjacency matrix $A$, the construction of the canonical path might be different.

Let $\Lambda(\Gamma) = \left\{\wt \Gamma\in \check\cG:\Gamma\in \cT_{\wt \Gamma,\Gamma^*}\right\}$ denote the set of all precedents states before $\Gamma$ along the canonical path. Let $\mathcal E = \left\{(\Gamma,\Gamma')\in E({\cal C}): \Gamma\in \Lambda(\Gamma') \right\}$ denote the ordered adjacent pairs along the canonical path. It follows that
\begin{align*}
\rho(\cT) &= \max_{(\Gamma,\Gamma')\in E(\mathcal{C})}\frac{1}{Q(\Gamma,\Gamma')}\sum_{\cT_{x,y}\ni (\Gamma,\Gamma')}\check\Pi_g(x)\check\Pi_g(y)\\
&\leq \max_{(\Gamma,\Gamma')\in \mathcal{E}} \frac{1}{Q(\Gamma,\Gamma')} \sum_{x\in \Lambda(\Gamma),y\in \check\cG}\check\Pi_g(x)\check\Pi_g(y)\\
&= \max_{(\Gamma,\Gamma')\in \mathcal{E}} \frac{\check\Pi_g(\Lambda(\Gamma))}{Q(\Gamma,\Gamma')} = \max_{\Gamma\in \check\cG}\frac{\check\Pi_g(\Lambda(\Gamma))}{Q(\Gamma,\check g(\Gamma))},
\end{align*}
where we simply take maximum only over all states in the discrete space $\check\cG$. By the definition of Algorithm \ref{algo:Ag1} and the lazy transition matrix, $Q(\Gamma,\Gamma')$ can be expressed as
\begin{align*}
Q(\Gamma,\check g(\Gamma)) = \check\Pi_g(\Gamma)\check P_g(\Gamma,\check g(\Gamma)) = \frac{1}{2(K-1)n}\min \left\{\check\Pi_g(\Gamma),\check\Pi_g(\check g(\Gamma))\right\}.
\end{align*}
It leads to the bound for the \textit{congestion parameter} as
\begin{align*}
\rho(\cT)&\leq 2(K-1)n \max_{\Gamma\in \check\cG}\frac{\check\Pi_g(\Lambda(\Gamma))}{\min\left\{\check\Pi_g(\Gamma),\check\Pi_g(\check g(\Gamma))\right\}}\\
& = 2(K-1)n \max_{\Gamma\in \check\cG}\frac{\check\Pi(\Lambda(\Gamma)|A)}{\min\left\{\check\Pi(\Gamma|A),\check\Pi(\check g(\Gamma)|A)\right\}}\\
& = 2(K-1)n \max_{\Gamma\in \check\cG}\left\{\frac{\check\Pi(\Lambda(\Gamma)|A)}{\check\Pi(\Gamma|A)}{\cdot\max\left\{1,\frac{\check\Pi(\Gamma|A)}{\check\Pi(\check g(\Gamma)|A)}\right\}}\right\},
\end{align*}
where $\check\Pi(\Gamma|A)=\sum_{Z\in \Gamma}\tPi(Z|A)$ for all $\Gamma\in \check S_\alpha$.

\begin{lemma}\label{lm:PathIncrease}
	Recall that $g(Z)$ is the next optimal state of $Z$ defined in \eqref{eq:defGZ}, and $\cG(\gamma_0)$ is defined in \eqref{eq:defineGoodRegion}. Suppose $Z_0$ is given with $\ell(Z_0,Z^*)\leq \gamma_0$ where $\gamma_0$ satisfies Condition \ref{con:gamma0}, \ref{con:gamma0xi_small}, or \ref{con:pqknowncon}. Suppose $\xi$ satisfies Condition \ref{con:xi}, \ref{con:gamma0xi_small} or \ref{con:pqknowncon}. Then, we have
	\begin{align*}
		\max_{Z\in \cG(\gamma_0)}\frac{\tPi(Z|A)}{\tPi(g(Z)|A)}\leq \exp(-C\bar nI)
	\end{align*}
	for some constant $C>1-\eps_0$ with probability at least $1-C_1n^{-C_2}$.
\end{lemma}

\ignore{Recall that $\cG$ is defined based on the initial state $Z_0$ of the algorithm with initial mistake number $m_0$, and $m_0 =o(n)$ implies $d(Z,Z^*)=o(n)$ for all $Z\in\cG(\gamma_0)$. Lemma \ref{lm:PathIncrease} says that the probability dramatically increases along the canonical path provided that all assumptions hold. With Lemma \ref{lm:PathIncrease}, we are able to bound the \textit{congestion parameter} $\rho(\cT)$.}
The proof of Lemma \ref{lm:PathIncrease} is deferred to Section \ref{sec:pqknown} and Section \ref{sec:pqunknown}. By Lemma \ref{lm:PathIncrease} and by permutation symmetry, we have
\begin{align*}
    \max_{\Gamma\in \check \cG}\frac{\check\Pi(\Gamma|A)}{\check\Pi(\check g(\Gamma)|A)} = \max_{Z\in \cG}\frac{\tPi(Z|A)}{\tPi(g(Z)|A)}\leq \exp(-C\bar nI) \leq \exp(-C\bar nI),
\end{align*}
for some constant $C>1-\eps_0$ with high probability. Denote $\wt m = n\max\{\gamma_0, n^{-\tau}\}+\log^2 n$ for simplicity, and it follows that
\begin{align*}
\max_{\Gamma\in \check \cG}\frac{\check\Pi(\Lambda(\Gamma)|A)}{\check\Pi(\Gamma|A)}\leq& 1+ \max_{m\leq\wt m}\sum_{l=1}^{\wt m-m}{n-m\choose l}(K-1)^l \exp(-Cl\bar nI)\\
 \leq& 1 + C'n\exp(-C\bar nI),
\end{align*}
for some constants $C'$ and $C>1-\eps_0$. Then, we have
\begin{align*}
\rho(\cT) & \leq 2(K-1)n \max_{\Gamma\in \check \cG}\left\{\frac{\check \Pi(\Lambda(\Gamma)|A)}{\check \Pi(\Gamma|A)}{\cdot\max\left\{1,\frac{\check \Pi(\Gamma|A)}{\check \Pi(g(\Gamma)|A)}\right\}}\right\}\\
 &\leq 2(K-1)n(1+C'n\exp(-C\bar nI)).
\end{align*}
Furthermore, since the canonical path is defined within $\check\cG$, we can upper bound the length of the longest path by
\begin{align*}
\ell(\cT)\leq 2n\max\{\gamma_0,n^{-\tau}\}+2\log^2 n.
\end{align*}

Recall that $\Gamma_0=\Gamma(Z_0)$. By Lemma \ref{lm:stayingood} and Lemma \ref{lm:TV}, together with \eqref{eq:boundmixingfomula} and the strong consistency property of $\check \Pi_g(\cdot|A)$, we have that for any constant $\eps\in(0,1)$,
\begin{equation}\label{eq:boundTVsmall}
\TVdiff{\check P^T(\Gamma_0,\cdot )}{\check \Pi(\cdot|A)}\leq \eps,
\end{equation}
holds for any
\begin{equation}\label{eq:timethresh}
    T \geq 4(K-1)n^2\max\{\gamma_0,n^{-\tau}\}\bp{-\xi \log \Pi(Z_0|A)+\log \eps^{-1}}(1+o(1)),
\end{equation}
for large $n$ with probability at least $1-C_3n^{-C_4}$ for some constants $C_3,C_4$, where $\Pi^\xi(Z_0|A)\leq \check\Pi_g(\Gamma_0|A)$ always holds. Finally, if $\PP{\ell(Z_0,Z^*)\leq \gamma_0}\geq 1-\eta$, then the conclusions of Theorem \ref{thm:mixing} and Theorem \ref{thm:mixingknown} can be obtained by a simple union bound argument.

\subsubsection{Coupling}

We require $T$ to be at most a polynomial of $n$ so that Lemma \ref{lm:stayingood} holds. Thus, the previous total variation bound \eqref{eq:boundTVsmall} holds only for $T\leq \text{poly}(n)$. In order to bound the mixing time defined in \eqref{eq:mixingFormulaEps}, we further use coupling approach to show the total variation bound holds for any $t\geq T$.

We call a probability measure $w$ over $\Omega\times\Omega$ is a coupling of $(u,v)$ if its two marginals are $u$ and $v$ respectively. Before the proof, we first state the following lemma to relate the total variation to the coupling.
\begin{lemma}[Proposition 4.7 in \cite{levin2017markov}]\label{lm:coupling}
     For any coupling $w$ of $(u,v)$, if the random variables $(X,Y)$ is distributed according to $w$, then we have
     \[
        \TVdiff{u}{v}\leq \PP{X\neq Y}.
     \]
\end{lemma}

Back to our problem, in order to upper bound $\TVdiff{\check P^t(\Gamma_0,\cdot )}{\check \Pi(\cdot|A)}$ for any $t\geq T$, we first create a coupling of these two distributions as follows. Consider two copies of the Markov chain $X_t$ and $Y_t$ both with the transition matrix $\check P$:
\begin{itemize}
    \item Let $X_0=\Gamma_0$, and $Y_0\sim \check \Pi(\cdot|A)$.
    \item If $X_t\neq Y_t$, then sample $X_{t+1}$ and $Y_{t+1}$ independently according to $\check P(X_t,\cdot)$ and $\check P(Y_t,\cdot)$ respectively.
    \item If $X_t=Y_t$, then sample $X_{t+1}\sim \check P(X_t,\cdot)$ and set $Y_{t+1}=X_{t+1}$.
\end{itemize}
Thus, it is obviously that for any $t\geq 1$, $Y_t\sim \check \Pi(\cdot|A)$, and $X_t\sim \check P^t(\Gamma_0,\cdot)$. Set $T=4Kn^2\max\{\gamma_0,n^{-\tau}\}\bp{-\xi \log \check\Pi(Z_0|A)+\log \eps^{-1}}(1+o(1))$ defined in \eqref{eq:timethresh}. By Lemma \ref{lm:coupling} and \eqref{eq:timethresh}, we have for any $t\geq T$,
\begin{align*}
    \TVdiff{\check P^t(\Gamma_0,\cdot )}{\check \Pi(\cdot|A)}&\leq \PP{X_t\neq Y_t}\leq \PP{X_T\neq Y_T}\\
    &= 1-\PP{X_T=Y_T}\\
    &\leq 1-\PP{X_T=Y_T=\Gamma^*}\\
    &\leq 2-\PP{X_T=\Gamma^*} - \PP{Y_T=\Gamma^*}.
\end{align*}
By \eqref{eq:boundTVsmall}, we have
\[
    \TVdiff{\check P^T(\Gamma_0,\cdot )}{\check \Pi(\cdot|A)}=\max_{S}\abs{\check P^T(\Gamma_0,S)-\check \Pi(S|A)}\leq \eps
\]
with high probability. Together with the strong consistency result, it yields
\begin{align*}
    \TVdiff{\check P^t(\Gamma_0,\cdot )}{\check \Pi(\cdot|A)}&\leq 2-\PP{X_T=\Gamma^*} - \PP{Y_T=\Gamma^*}\\
    &\leq 1- \bp{\check\Pi(\Gamma^*|A)-\eps}-\check \Pi(\Gamma^*|A) \\
    &\leq \eps(1+o(1)),
\end{align*}
with probability at least $1-Cn^{-C'}$ for some constants $C,C'$. Here, the high probability statement is with respect to the data generation process, i.e., adjacency matrix $A$.

To combine, we reach the result that for any constant $\eps\in(0,1)$,
\[
    \tau_\eps(Z_0)\leq 4Kn^2 \max\left\{\gamma_0,n^{-\tau}\right\}\cdot \left(\xi \log \Pi(Z_0|A)^{-1}+\log (\eps^{-1})\right),
\]
with high probability where $\tau_\eps(Z_0)$ is defined in \eqref{eq:mixingFormulaEps}.

\subsection{Proof of Lemma \ref{lm:stayingood}}
\label{sec:stayingood}
For any state $Z\in S_\alpha$, we define
\begin{equation}\label{eq:defABN}
    \begin{split}
        &\cN(Z) =\left\{{Z': H(Z',Z)=  1}\right\},\\
        &\cA(Z) = \left\{{Z'\in \cN(Z): d(Z',Z^*) = d(Z,Z^*)+1}\right\},\\
        &\cB(Z) = \left\{{Z'\in \cN(Z): d(Z',Z^*) = d(Z,Z^*)-1}\right\},
    \end{split}
\end{equation}
where $\cN(Z)$ denotes the neighborhood states of $Z$ with only one sample classified differently, and $\cA(Z)$ (resp. $\cB(Z)$) denotes the set of states with more mistakes (resp. fewer mistakes) in the neighbor. We further define 
\begin{equation}\label{eq:defpmqm}
\begin{split}
    p_m(Z) = P(Z,\cA(Z)) = \frac{1}{2(K-1)n}\sum_{Z'\in \cA(Z)\cap S_\alpha} \min \left\{{1,\frac{\tPi_g(Z'|A)}{\tPi_g(Z|A)}}\right\},\\
q_m(Z) = P(Z,\cB(Z)) = \frac{1}{2(K-1)n}\sum_{Z'\in \cB(Z)\cap S_\alpha} \min \left\{{1,\frac{\tPi_g(Z'|A)}{\tPi_g(Z|A)}}\right\},
\end{split}
\end{equation}
where $p_m(Z),q_m(Z)$ are the probabilities of $Z$ jumping to states with the number of mistakes equal to $m+1$, $m-1$ respectively.
We have the following lemma to bound the ratio of $p_m(Z)$ and $q_m(Z)$ for any $Z\in \cG$. Recall that $\cG = \cG(\gamma_0)$ is defined in \eqref{eq:defineGoodRegion}.

\begin{lemma}\label{lm:randomwalk_probratio}
Suppose $\gamma_0$ satisfies Condition \ref{con:gamma0}, \ref{con:gamma0xi_small}, or \ref{con:pqknowncon}. Let $\tau$ be any sufficiently small constant $\tau$, and denote $\cG^*=\{Z:n^{-\tau}\leq \ell(Z,Z^*)\leq \max\{\gamma_0,n^{-\tau}\}+(\log n)^2/n\}$. Then, we have
	\begin{align*}
\Prob\left\{\max_{Z\in \cG^*}\frac{p_m(Z)}{q_m(Z)}\geq{\eta} \right\}\leq
\exp\bp{-n^{1-\tau}}
	\end{align*}
	for some small constant $4\tau<\eta<1$. The probability is with respect to the data-generating process, i.e., the adjacency matrix $A$.
\end{lemma}

The proof of Lemma \ref{lm:randomwalk_probratio} is deferred to Section \ref{sec:randomwalk_probratio}. We take the $\tau$ in Lemma \ref{lm:randomwalk_probratio} to be the same as $\tau$ defined in $\cG$. In order to show that the Markov chain will stay in $\cG$ with high probability, we transform the original problem into an one dimensional random walk problem. Lemma \ref{lm:randomwalk_probratio} shows that the probability ratio of the one dimensional random walk on the region $\cG^*$ can be bounded with high probability. All the following analysis is conditioning on the adjacency matrix $A$ such that the event 
\begin{align*}
\mathcal E(A) = \bb{\max_{Z\in \cG^*}\frac{p_m(Z)}{q_m(Z)}\leq{\eta}}
\end{align*}
happens. We construct the following three types of Markov chains in order to prove Lemma \ref{lm:stayingood}.

\begin{description}[leftmargin = 0cm,labelsep = 0.5cm]

\item[Type \rom{1} MarKov chain.]

Consider a particle starting at the initial position $u$ on the $x$-axis where $0<u<b$ at time $t=0$, and it moves one unit to the left, to the right, or stay at the current position at time $t=1,2,\ldots$ with probability $q_t$, $p_t$, or $1-q_t-p_t$, where $\eta>p_t/q_t$ for all time $t$. It stops once it reaches the left or the right boundary, and we are interested in the probability of its stopping at the boundary $b$ or the boundary $0$.

Suppose the position of the particle at time $t$ is $X_t$, and $X_{t+1}=  X_t +\xi_t$, where $\xi_t$ follows the distribution
\begin{align*}
\Prob\bb{\xi_t=1} = p_t,~~\Prob\bb{\xi_t=-1} = q_t,~~\Prob\bb{\xi_t = 0} = 1-p_t-q_t.
\end{align*}
We define $Y_t = \exp((b-X_t)\log \eta)$. It is easy to verify that the stochastic process $Y_t$ is a super-martingale, due to the fact that
\begin{align*}
\Expect\bp{Y_{t+1}|Y_t} &= \exp((b-X_t)\log \eta) \cdot \Expect\bp{\exp(-\xi_t\log \eta)}\\
&=\exp((b-X_t)\log \eta)\cdot \bp{1-p_t-q_t+p_t/\eta+q_t\cdot \eta}\\
&\leq Y_t.
\end{align*}

Let $\tau = \min\{t\geq 1: X_t = b \text{~or~} X_t=0\}$, and $\tau$ is \textit{stopping time} of this random walk. It is evident that $\abs{Y_{t\wedge \tau}}\leq 1$ since $X_\tau<b$ for all time $t$. By Doob's optional stopping time theorem, it follows that $\Expect\bp{Y_{\tau}}\leq \Expect\bp{Y_0}$, i.e.,
\begin{align}\label{eq:probbound}
\Prob\bb{X_\tau = b}\leq \Prob\bb{X_\tau = 0}\cdot \eta^{b}+\Prob\bb{X_\tau = b}\leq \eta^{b-u},
\end{align}
where the first inequality holds since $\PP{X_\tau=0}\geq 0$, and the second inequality is due to Doob's optional stopping time theorem. By \eqref{eq:probbound}, the probability of the particle reaching boundary $b$ first is upper bounded by $\eta^{b-u}$, where $u$ is the starting position. Let $P_{u}^b=\PP{X_0=u, X_\tau = b}$ for $u\in (0,b)$ denote the probability of starting at $u$ and stopping at $b$. Then, we have $P_{1}^b \leq \eta^{b-1}$.

Now suppose a particle starts at 0, i.e., $X_0=0$. It moves to the right or stay at the current position with some fixed probability $p_0$ or $1-p_0$. Let $P_0^0 = \PP{X_0=0,~X_\tau=0}$, and $P_0^b = \PP{X_0=0,~X_\tau=b}$, where $\tau$ is the stopping time as defined before. Then, we have that
\begin{equation}\label{eq:P00}
P_{0}^0 = p_0\cdot P_{1}^0+(1-p_0),~~~P_{0}^b = p_0\cdot P_{1}^b,~~~P_0^0+P_0^b = 1.
\end{equation}

\item[Type \rom{2} Markov chain.] We now define another Markov chain that is similar to the previous one. Consider a particle starting at position $0$ at time 0, and follows the same updating rule as the previous chain but different stopping rule. We use $W_t$ to denote the position of the particle at time $t$. The particle will only stop when it reaches the boundary $b$. When it is at the position $0$, it still moves to the right or stay at 0 with fixed probability $p_0$ or $1-p_0$ (the same probability as defined in Type \rom{1} Markov chain). Thus, this newly defined Markov chain is a \textit{reflected random walk}.

It is worth noting that the Type \rom{2} Markov chain can always be decomposed into several Type I Markov chains. We use $\tau_W$ to denote the stopping time of Type \rom{2} Markov chain, defined by $\tau_W = \min\{t\geq 1:W_t = b\}$.

\item[Type \rom{3} Markov chain.] Now return to our original problem and construct Type \rom{3} Markov chain. Let $m_0=n\ell(Z_0,Z^*)$, and we use $H_t = n\ell(Z_t,Z^*)-m_0$ to denote the position of the particle, where $Z_t$ is the label assignment after $t$ steps. The state space is all integers between $-m_0$ and $b$, where we take $b=\max\{0, n^{1-\tau}-m_0\}+\log^2 n$. When $H_t\in (0,b)$, the particle moves to the left, to the right, or stay at the current position with probability $q_t$, $p_t$, or $1-q_t-p_t$, which is the same as the Type \rom{2} chains. \ignore{Note that $-m_0$ is the reflecting boundary, corresponding to the case of $\ell(Z_t,Z^*)=0$, and the }The particle will only stop when it reaches the boundary $b$. The stopping time of Type \rom{3} Markov chain is defined by $\tau_H = \min\{t\geq 1:H_t = b\}$.
\end{description}

\begin{proof}[Proof of Lemma \ref{lm:stayingood}]
Recall that $b=\max\{0, n^{1-\tau}-m_0\}+\log^2n$. In order to prove Lemma \ref{lm:stayingood}, it is equivalent to show that, for any $T$ that is a polynomial of $n$, the event $\{H_t<b, ~t\leq T\}$ happens with high probability, i.e., $\{\tau_H>T\}$ happens with high probability. By the definition of Type \rom{2} and \rom{3} Markov chains we have that 
\[
    \PP{\tau_H\leq T}\leq \PP{\tau_W\leq T}.
\]
The above inequality holds since $H_0=W_0$, $H_t\geq 0$ for all time $t$, and the updating rule of $H_t$ and $W_t$ are exactly the same when $W_t, H_t\in (0,b)$.

We now connect the Type \rom{2} Markov chain with multiple Type \rom{1} chains. The event $\{\tau_W\leq T\}$ means that the particle starts at 0 and reaches the boundary $b$ within $T$ steps, and it can be written as
\begin{equation}\label{eq:tauW}
    \left\{\tau_W\leq T\right\} = \bigcup_{k=1}^T\left\{\left[\bigcap_{i=1}^{k-1}\left\{X_0^{(i)}=0,~X_{\tau_i}^{(i)} = 0 \right\}\right]\bigcap \left\{X_0^{(k)} = 0,~X_{\tau_i}^{(k)}=b\right\}\bigcap\left\{\sum_{i=1}^k{\tau_i}\leq T\right\}\right\},
\end{equation}
where we use $X^{(i)}$ to denote the $i$th Type \rom{1} Markov chain, and $\tau_i$ is the stopping time of $X^{(i)}$. Note $X^{(i)}$ is independent with $X^{(j)}$ for $i\neq j$. The right hand side of \eqref{eq:tauW} can be interpreted as that the particle reaches the boundary 0 for $k-1$ times with $k\leq T$ before reaching the boundary $b$, and the total number of steps is less than $T$. Therefore, it follows directly that
\begin{align*}
\Prob\bb{\tau_W\leq T} &= \PP{\bigcup_{k=1}^T\left\{\left[\bigcap_{i=1}^{k-1}\left\{X_0^{(i)}=0,~X_{\tau_i}^{(i)} = 0 \right\}\right]\bigcap \left\{X_0^{(k)} = 0,~X_{\tau_i}^{(k)}=b\right\}\bigcap\left\{\sum_{i=1}^k{\tau_i}\leq T\right\}\right\}}\\
&\leq \sum_{k=1}^T\PP{\left[\bigcap_{i=1}^{k-1}\left\{X_0^{(i)}=0,~X_{\tau_i}^{(i)} = 0 \right\}\right]\bigcap \left\{X_0^{(k)} = 0,~X_{\tau_i}^{(k)}=b\right\}\bigcap\left\{\sum_{i=1}^k{\tau_i}\leq T\right\}}\\
&\leq \sum_{k=1}^T\PP{\left[\bigcap_{i=1}^{k-1}\left\{X_0^{(i)}=0,~X_{\tau_i}^{(i)} = 0 \right\}\right]\bigcap \left\{X_0^{(k)} = 0,~X_{\tau_i}^{(k)}=b\right\}}\\
& = \sum_{k=1}^T (P_0^0)^{k-1}P_0^b.
\end{align*}
The first inequality holds by a union bound. The third inequality is by the independence. By \eqref{eq:P00}, we have that
\begin{align*}
\sum_{k=1}^TP_0^b (P_0^0)^{k-1}& = 1-(P_0^0)^T\\
&\leq -T\log P_0^0\leq T\cdot \frac{1-P_0^0}{P_0^0} = T\cdot \frac{p_0\cdot P_{1}^b}{p_0\cdot P_1^0+1-p_0}\\
&\leq T\cdot \frac{P_1^b}{P_1^0}\leq \exp\bp{\log {\frac{\eta^{b-1}}{1-\eta^{b-1}}}+\log T}.
\end{align*}
Since $T$ is a polynomial of $n$, and $b\gg \log n$, then it follows that
\begin{align*}
\PP{\tau_H\leq T}\leq \PP{\tau_W\leq T}\leq \exp\left(-(1-o(1))b\log \frac{1}{\eta}\right).
\end{align*}
Thus, based on the result of Lemma \ref{lm:randomwalk_probratio}, for any given initial label assignment $Z_0$ with $\ell(Z_0,Z^*) \leq \gamma_0$ where $\gamma_0$ satisfies Condition \ref{con:gamma0}, \ref{con:gamma0xi_small}, or \ref{con:pqknowncon}, we have that in any polynomial running time, the number of mistakes is upper bounded by 
\[
  m\leq m_0+b = \max\{m_0,n^{1-\tau}\}+\log ^2n  \leq n\max\{\gamma_0,n^{-\tau}\}+\log ^2n,
\]
with probability at least $1-\exp(-\log^2 n)$.
\end{proof}

\subsection{Some preparations before the proofs of Lemma \ref{lm:postlargemistakes} and Lemma \ref{lm:PathIncrease}}
In this section, we will define some events and introduce some quantities to simplify the main proof.

\subsubsection{Basic events}

For any $Z\in S_\alpha$, recall that $O_{ab}(Z)=\sum_{i,j}A_{ij}\indc{Z_i=a, Z_j=b}$, and define $X_{ab}(Z) = O_{ab}(Z)-\EE{O_{ab}(Z)}$ for any $a,b\in[K]$. Let $X(Z)$ denote a $K\times K$ matrix with its $(a,b)$th element equal to $X_{ab}(Z)$ for any $a,b\in [K]$. For any positive sequences $\bar\eps = \bar \eps_n$, $\gamma=\gamma_n$, and $\theta=\theta_n$ satisfying that $\bar\eps,\gamma, \theta\goto 0, \bar\eps^2 nI\goto\infty$, and $\theta^2\gamma nI\goto\infty$, consider the following events:
\begin{equation}\label{eq:allevents}
\begin{split}
\mathcal{E}_1(\beps)& = \bb{\max_{Z\in S_\alpha}\Norm{X(Z)}_{\infty}\leq \bar\eps n^2(p-q)},\\
\mathcal{E}_2 &= \bb{\max_{Z\in S_\alpha}\Norm{X(Z)-X(Z^*)}_{\infty}-  (\alpha+\beta)mn(p-q)/K\leq 0},\\
\mathcal{E}_3(\beps) &= \bb{\max_{Z\in S_\alpha}\Norm{X(Z)-X(Z^*)}_{\infty}\leq \bar\eps n^2(p-q)},\\
\mathcal{E}_4(\gamma, \theta) & = \bb{\max_{Z\in S_\alpha: m>\gamma n}\Norm{X(Z)-X(Z^*)}_{\infty}\leq \theta mn(p-q)},\\
\mathcal{E}_5 & = \bb{\max_{Z\in S_\alpha} \frac{1}{\abs{\cB(Z)\cap S_{\alpha}}}\sum_{Z'\in \cB(Z)\cap S_\alpha}\sum_{a\leq a'}\abs{X_{aa'}(Z)-X_{aa'}(Z')}\leq 10n(p-q)},\\
\mathcal{E}_6(\gamma,\theta) & = \bb{\max_{Z\in S_\alpha: m> \gamma n} \frac{1}{\abs{\cB(Z)\cap S_\alpha}}\sum_{Z'\in \cB(Z)\cap S_\alpha}\sum_{a\leq a'}\abs{X_{aa'}(Z)-X_{aa'}(Z')}\leq \theta n(p-q)},
\end{split}
\end{equation}
where $\cB(Z)$ is defined in \eqref{eq:defABN}.
We may also use $\mathcal E_1$ to denote $\mathcal E_1(\beps)$ for simplicity, and such simplification also applies to any other event. Denote 
\begin{equation}\label{eq:defE}
\mathcal{E} = \mathcal{E}(\beps, \gamma, \theta)=\mathcal{E}_1\cap\mathcal{E}_2\cap\mathcal{E}_3\cap\mathcal{E}_4\cap\mathcal{E}_5\cap\mathcal{E}_6.    
\end{equation} 
By Lemmas \ref{lm:boundhatPtildeP}, \ref{boundXabe_xabc}, \ref{boundXabe_xabc_mlarge}, \ref{lm:bounddtOLargeMis}, \ref{lm:pathdiffall}, and \ref{lm:pathdifflargeMis}, it follows that for any $\beps, \gamma, \theta$ satisfying the conditions, 
\begin{align*}
\PP{\mathcal{E}(\beps,\gamma,\theta)}\geq 1-n\exp(-(1-o(1))\bar nI).
\end{align*}

\subsubsection{Likelihood modularity}
The posterior distribution is hard to deal with directly. Hence, we first analyze the performance of the likelihood modularity function, and then bound the difference between the likelihood modularity function and the posterior distribution to simplify the proof.

Likelihood modularity is first introduced in \cite{bickel2009nonparametric}, which takes the form as
\begin{equation}\label{eq:defineLikeMode}
    Q_{LM}(Z,A) = \sum_{a\leq b}n_{ab}(Z)\tau\bp{\frac{O_{ab}(Z)}{n_{ab}(Z)}},
\end{equation}
where $\tau(x) = x\log (x) + (1-x) \log (1-x)$. This criterion replaces the connectivity probabilities by maximum likelihood estimates. Instead of comparing the direct difference between the Bayesian expression and the likelihood modularity as in \cite{van2017bayesian}, we have the following lemma to bound the relative difference.

\begin{lemma}\label{lm:likemode}
    Under the event $\cal E(\beps,\gamma,\theta)$ defined in \eqref{eq:defE}, we have
    \begin{align*}
    \max_{Z\in S_\alpha}\abs{\log\frac{\Pi(Z|A)}{\Pi(Z^*|A)}-(Q_{LM}(Z,A)-Q_{LM}(Z^*,A))}\leq C_{LM}
    \end{align*}
    for some constant $C_{LM}$ only depending on $K, \alpha,\beta$.
\end{lemma}
The above lemma is rephrased and proved in Lemma \ref{lm:prooflikemode}. \ignore{Under Lemma \ref{lm:likemode}, the proof of Lemma \ref{lm:postlargemistakes} can be simplified by first working on the likelihood modularity and then show that the remaining difference is comparatively negligible. }

\subsubsection{Discrepancy matrix}
For any label assignment $Z\in S_\alpha$, let $R_Z$ be a \textit{discrepancy matrix}, which takes the form of
\begin{align}\label{eq:defRZ}
R_Z(a,b) = \sum_{i=1}^n \indc{Z_i=a,Z_j^*=b}, ~~~a,b\in[K],
\end{align}
where $R_Z(a,b)$ is the number of samples misclassified to group $a$ but actually from group $b$ based on the true label assignment. Note that the true label assignment is only unique up to a label permutation, and thus we always permute the rows of $R_Z$ to minimize the off diagonal sum. Later we write $R_Z(k,l)$ as $R_{kl}$ for simplicity.

Using the discrepancy matrix $R_Z$, we have
\begin{equation}\label{eq:RcalB}
\begin{split}   
    \EE{O_{ab}(Z)} &= \EE{\sum_{i,j}A_{ij}\indc{Z_i=a,Z_j=b}} = (RBR^T)_{ab},~~~ \text{for }a\neq b\in [K],\\
    \EE{O_{aa}(Z)} &= \EE{\sum_{i<j}A_{ij}\indc{Z_i=Z_j=a}} = \frac{1}{2}\bp{(RBR^T)_{aa}-\sum_{k}B_{kk} R_{ak}},~~~ \text{for }a\in [K].
\end{split} 
\end{equation}

\subsection{Proof of Lemma \ref{lm:postlargemistakes}}
\label{sec:postContrSec}

Before proving the lemma, we need to present some notations that will be frequently us{}ed:
\begin{align*}{}
    O_s(Z)& = \sum_{a\in[K]}O_{aa}(Z) = \sum_{i<j}A_{ij}\indc{Z_i=Z_j},\\
    n_s(Z)& = \sum_{a\in[K]}n_{aa}(Z) = \sum_{i<j}\indc{Z_i=Z_j},\\
    \dtO_{ab}& = O_{ab}(Z)- O_{ab}(Z^*),~~\dtO_{s} = \sum_{a\in [K]}\dtO_{aa},\\
    \dtn_{ab}& = n_{ab}(Z)- n_{ab}(Z^*),~~\dtn_{s} = \sum_{a\in [K]}\dtn_{aa},
\end{align*}{}
and we may write $n_{ab}(Z^*)$, $O_{ab}(Z^*)$ as $n_{ab}$, $O_{ab}$ for simplicity.
\begin{proof}[Proof of Lemma \ref{lm:postlargemistakes}]
For any positive sequences $\gamma = \gamma_n$ and $\theta=\theta_n$ such that $\gamma\goto 0$, $\gamma^2nI\goto \infty$, and $\theta^2\gamma nI\goto\infty$, we can construct the event $\mathcal E(\beps, \gamma,\theta)$ defined in \eqref{eq:defE} by setting $\beps=\gamma$, and perform analysis on a large mistake region and a small mistake region separately. 
\begin{description}[leftmargin = 0cm,labelsep = 0.5cm]
\item[Small mistake region.] For $m\leq\gamma n$, by some calculations, it follows that
\begin{equation}\label{eq:ZZstarsmall}
    \begin{split}
    &Q_{LM}(Z,A)-Q_{LM}(Z^*,A)\\
    = &\log \Pi_0(Z|A)-\log \Pi_0(Z^*|A)
    -\sum_{a\leq b}n_{ab} \cdot \KL{\frac{O_{ab}}{n_{ab}}}{\frac{O_{ab}(Z)}{n_{ab}(Z)}}+\\
    &\underbrace{\sum_{a\leq b}\dtO_{ab} \bp{\log \frac{O_{ab}(Z)}{n_{ab}(Z)}-\log B_{ab}}+ (\dtn_{ab}-\dtO_{ab}) \bp{\log \frac{n_{ab}(Z)-O_{ab}(Z)}{n_{ab}(Z)}-\log(1- B_{ab})}}_{(Error)},
\end{split}
\end{equation}
where by \eqref{eq:postknown}
\begin{equation}\label{eq:knownPostEq}
    \log \Pi_0(Z|A)-\log\Pi_0(Z^*|A) \ignore{= \log \frac{p(1-q)}{q(1-p)}\bp{\dtO_s - \lambda^* \dtn_s}}=2t^*\bp{\dtO_s - \lambda^* \dtn_s},
\end{equation}
and
\begin{equation}\label{eq:deflambda_t}
    \lambda^* = \log \frac{1-q}{1-p}\big/ \log \frac{p(1-q)}{q(1-p)},~~~~t^* = \frac{1}{2}\log \frac{p(1-q)}{q(1-p)}.
\end{equation}
Under the event $\mathcal{E}(\beps,\gamma,\theta)$, by Lemma \ref{lm:bound_ec}, we have $(Error)\leq C\gamma mn I$ for some constant $C$. Hence, for any fixed $Z\in S_\alpha$, by Lemma \ref{lm:likemode}, we have that
\begin{align*}
  &\log\frac{\Pi(Z|A)}{\Pi(Z^*|A)}-\log\frac{\Pi_0(Z|A)}{\Pi_0(Z^*|A)}\\
  \leq &Q_{LM}(Z,A)-Q_{LM}(Z^*,A)+C_{LM}-\log\frac{\Pi_0(Z|A)}{\Pi_0(Z^*|A)}\\
  \leq &C\gamma m nI+C_{LM}.  
\end{align*}
Thus, there exists some constant $C_2$ such that under the event $\cal E$, 
\begin{align*}
    {\max_{Z\in S_\alpha:m<\gamma n}}\log\frac{\Pi(Z|A)}{\Pi(Z^*|A)}-\log\frac{\Pi_0(Z|A)}{\Pi_0(Z^*|A)}-C_2\gamma m nI\leq 0,
    \end{align*}
    which proves the second statement in Lemma \ref{lm:postlargemistakes}.

\item[Large mistake region.] For $m>\gamma n$, we have that
\begin{equation}\label{eq:pqunknownLargeZZstar}
    \begin{split}
    Q_{LM}(Z,A)-Q_{LM}(Z^*,A) &=\sum_{a\leq b} n_{ab}(Z)\tau \bp{\frac{O_{ab}(Z)}{n_{ab}(Z)}}-n_{ab}(Z^*)\tau \bp{\frac{O_{ab}(Z^*)}{n_{ab}(Z^*)}}\\
    &= \bp{G(Z) + \Delta(Z) }- \bp{G(Z^*) +\Delta(Z^*)},
\end{split}
\end{equation}
     where we write
     \begin{align*}
    G(\cdot)& = \sum_{a\leq b} n_{ab}(\cdot)\tau \bp{\frac{\EE{O_{ab}(\cdot)}}{n_{ab}(\cdot)}},\\
    \Delta(\cdot) & = \sum_{a\leq b} n_{ab}(\cdot )\bp{\tau \bp{\frac{{O_{ab}(\cdot)}}{n_{ab}(\cdot)}}-\tau \bp{\frac{\EE{O_{ab}(\cdot)}}{n_{ab}(\cdot)}}}.
    \end{align*}
Let $\wt B_{ab} = \EE{O_{ab}(Z)}/n_{ab}(Z)$ for any $a,b\in[K]$. By \eqref{eq:RcalB}, it follows that
    \begin{equation*}
    \begin{split}
        2G(Z)   =& 2\sum_{a\leq b}n_{ab}(Z)\tau \bp{\tB_{ab}} \\
    =& 2\sum_{a\leq b} \EE{O_{ab}(Z)}\log \tB_{ab}+ (n_{ab}(Z)-\EE{O_{ab}(Z)})\log (1-\tB_{ab})\\
    =& \sum_{a,b,k,l}R_{ak}R_{bl}(B_{kl}\log \tB_{ab}+(1-B_{kl})\log(1-\tB_{ab})) - \sum_{a}n_a(Z)(p\log \tB_{aa}+(1-p)\log(1-\tB_{aa})).
    \end{split}
    \end{equation*}
    Then, we have that
    \begin{equation}\label{eq:GZ-GZ*}
    \begin{split}
        &2G(Z)-2G(Z^*)\\
    =& \sum_{a,b,k,l}R_{ak}R_{bl}(B_{kl}\log \tB_{ab}+(1-B_{kl})\log(1-\tB_{ab})) - \sum_{a,b,k,l}R_{ak}R_{bl}(B_{kl}\log B_{kl}+(1-B_{kl})\log(1-B_{kl}))\\
    -& \sum_{a}n_a(Z)\bp{(p\log \tB_{aa}+(1-p)\log(1-\tB_{aa}))-(p\log p+(1-p)\log(1-p))}\\
    = & -\sum_{a,b,k,l}R_{ak}R_{bl}\KL{B_{kl}}{\tB_{ab}} + \sum_{a}n_a(Z)\KL{p}{\tB_{aa}}.
    \end{split}
    \end{equation}
    By Lemma \ref{lm:post_boundGe_1} and Lemma \ref{lm:post_boundGe_2} that bound the above two terms separately, we have
    \[2G(Z)-2G(Z^*)\leq -CmnI,\]
    for some constant $C$. Under the events $\mathcal E_1(\beps), \mathcal E_3(\beps), \mathcal E_4(\gamma, \theta)$ in \eqref{eq:allevents}, by Lemma \ref{lm:prooflikemode} and Lemma \ref{lm:boundDeltaec}, we have that
    \[
        \max_{Z\in S_\alpha}\abs{\log\frac{\Pi(Z|A)}{\Pi(Z^*|A)} - (Q_{LM}(Z,A)-Q_{LM}(Z^*,A))}\leq C_{LM},
    \]
    and
    \[\max_{Z\in S_\alpha: m>\gamma n}|\Delta(Z)-\Delta(Z^*)|\leq \eps mnI,\]
    for some $\eps\goto 0$. Hence, it follows that there exists some constant $C_1$ such that
    \[\PP{\max_{Z\in S_\alpha: m>\gamma n}\log\frac{\Pi(Z|A)}{\Pi(Z^*|A)}>-C_1 mnI}\leq 4\exp(-n).\]
    Combining the result of two regions directly gives Lemma \ref{lm:postlargemistakes}.
\end{description}
\end{proof}

\subsection{Proof of Lemma \ref{lm:PathIncrease} with known connectivity probabilities}
\label{sec:pqknown}

In order to distinguish from the case where probabilities are unknown, we use $\Pi_0(\cdot|A)$ to denote the posterior distribution in this case, and define $\wt \Pi_0(\cdot|A)$ as the scaled distribution proportional to $\Pi_0^\xi(\cdot|A)$. It suffices to prove the following lemma.

\begin{lemma}\label{lm:knownPincrease}
    Recall that $g(Z)$ is the next state of $Z$. Suppose $\gamma_0$ satisfies Condition \ref{con:gamma0}, \ref{con:gamma0xi_small}, or \ref{con:pqknowncon}. Then, there exists some positive sequence $\gamma\goto 0$ such that, with probability at least $1-C_1n^{-C_2}$,
    \begin{align*}
    \frac{\Pi_0(Z|A)}{\Pi_0(g(Z)|A)}\leq \begin{dcases}
    \exp\bp{-\eps\bar nI}, & \text{if~}m\leq \gamma n, \\
    \exp\bp{-{4\bp{1/K\alpha-\gamma_0}nI}(1-o(1))}, &\text{if~} m> \gamma n,
    \end{dcases}
    \end{align*}
    holds uniformly for all $Z\in \cG(\gamma_0)$ defined in \eqref{eq:defineGoodRegion}. Here, $\eps$ is any constant satisfying $\eps<2\eps_0$ with $\eps_0$ defined in Condition \ref{con:pqknowncon}, and $C_1,C_2$ are two constants depending on $\eps$.
    Furthermore, if $\xi$ satisfies Condition \ref{con:pqknowncon}, then by choosing $\eps\in((1-\eps_0)/\xi,2\eps_0)$, we have
    \begin{align*}
        \max_{Z\in \cG(\gamma_0)}\frac{\tPi_0(Z|A)}{\tPi_0(g(Z)|A)}\leq \exp(-C\bar nI)
    \end{align*}
    for some constant $C>1-\eps_0$ with probability at least $1-C_3n^{-C_4}$.
\end{lemma}

\begin{proof}[Proof of Lemma \ref{lm:knownPincrease}]
Recall the definitions of $\cA(Z)$, $\cB(Z)$, and $\cN(Z)$ in \eqref{eq:defABN}. We introduce some notations first to simplify the proof. For any $a,b\in[K]$ and any two label assignments $Z,Z'$, we write
\begin{align*}
&\dO_{ab} = O_{ab}(Z) - O_{ab}(Z'), && \dO_s = \sum_{a}\dO_{aa},\\
&\dn_{ab} = n_{ab}(Z) - n_{ab}(Z'), && \dn_s = \sum_{a}\dn_{aa},\\
&\Delta_n(Z,Z') = \dO_s-\lambda^* \dn_s,  && \lambda^* = \log\frac{1-q}{1-p}\Big / \log \frac{p(1-q)}{q(1-p)}.
\end{align*}
Suppose the current state is $Z$, and we randomly choose one misclassified sample from group $a$ and move to its true group $b$. Denote the new state as $Z'$. It follows that $R_{Z'}(a,b) = R_{Z}(a,b)-1$, and $R_{Z'}(b,b) = R_{Z}(b,b)+1$. Write $R_{Z}(a,b)$ as $R_{ab}$ for simplicity. Furthermore, let $\bb{x_l}_{l\geq 1},\bb{\wt x_l}_{l\geq 1}$ be i.i.d. copies of $\text{Bernoulli}(q)$ and $\bb{y_l}_{l\geq 1},\bb{\wt y_l}_{l\geq 1}$ be i.i.d. copies of $\text{Bernoulli}(p)$. We have that
\begin{equation}\label{eq:DeltaZZ'knownP}
    \Delta_n(Z,Z') = \sum_{l=1}^{R_{aa}+\sum_{k\neq a,b}R_{ak}}(x_l-\lambda^*)-\sum_{l=1}^{R_{bb}}(y_l-\lambda^*)+\sum_{l=1}^{R_{ab}-1}(\wt y_l-\lambda^*) -\sum_{l=1}^{\sum_{k\neq b}R_{bk}}(\wt x_l-\lambda^*).
\end{equation}
By \eqref{eq:postknown}, it directly follows that
\begin{equation}\label{eq:referLogPi0Ratio}
    \log \frac{\Pi_0(Z|A)}{\Pi_0(Z'|A)} = 2t^* \Delta_n(Z,Z'),~~ t^* = \frac{1}{2}\log\frac{p(1-q)}{q(1-p)}.
\end{equation}
For some positive sequence $\gamma =\gamma_n\goto 0$ to be specified later, we can again divide $S_\alpha$ into two regions.
\begin{description}[leftmargin = 0cm,labelsep = 0.5cm]
\item[Small mistake region.] For $m\leq \gamma n$, we have
\begin{align}
    &\EE{\sqrt{\frac{\Pi_0(Z|A)}{\Pi_0(Z'|A)}}}
      = \EE{e^{t^*\Delta_n(Z,Z')}}\nonumber\\
     = &\mathbb E\exp \left[t^*\bp{\sum_{l=1}^{R_{aa}}(x_l-\lambda^*)-\sum_{l=1}^{R_{bb}}(y_l-\lambda^*)}\right]\label{eq:chernoff}\\
    \cdot &\mathbb E\exp\left[ t^*\bp{\sum_{l=1}^{R_{ab}-1}(\wt y_l-\lambda^*) -\sum_{l=1}^{\sum_{k\neq b}R_{bk}}(\wt x_l-\lambda^*)+\sum_{l=1}^{\sum_{k\neq a,b}R_{ak}}(x_l-\lambda^*)}\right]\label{eq:chernoff2}.
    \end{align}
    Based on Lemma 6.1 in \cite{zhang2016minimax}, for any positive integers $n_1, n_2$, we have
    \[\mathbb E{\exp \left[t^*\bp{\sum_{i=1}^{n_1} (x_i-\lambda^*)-\sum_{i=1}^{n_2}(y_i-\lambda^*)}\right]}=\exp\bp{-\frac{(n_1+n_2)I}{2}},\]
    and it leads to
    \[
        \eqref{eq:chernoff}=\exp\bp{-\frac{(R_{aa}+R_{bb})I}{2}}\leq \exp\bp{-\frac{(n_a+n_b-m)I}{2}}\leq \exp\bp{-(1-c\gamma) \bar nI},
    \]
    for some constant $c$. By Lemma \ref{lm:boundExpKnowP}, we have
    \[
        \eqref{eq:chernoff2}\leq \bp{\exp(C_0I)}^m = \exp(C_0mI)\leq \exp(C_0\gamma nI),
    \]
    for some constant $C_0$. It follows that
    \begin{align}\label{eq:Pi0ratioPi0'}
    \EE{\sqrt{\frac{\Pi_0(Z|A)}{\Pi_0(Z'|A)}}}\leq \exp\bp{-(1-C_1\gamma)\bar nI)}
    \end{align}
    for some constant $C_1$ depending on $C_0$. Let $\eps$ be any small constant satisfying $\eps<2\eps_0$, and $w = \eps\bar nI/2t^*$. By Lemma \ref{lm:size}, since $\gamma<c_{\alpha,\beta}$, we have $\cB(Z)\subset S_\alpha$. Then,
    \begin{align}
    \PP{\min_{Z'\in\cB(Z)\cap S_\alpha} \Delta_n (Z,Z')\geq -w}&\leq \PP{\sum_{Z'\in\cB(Z)}\Delta_n (Z,Z')\geq -mw}\nonumber\\
    &= \PP{\exp\bp{t^*\sum_{Z'\in\cB(Z)}\Delta_n(Z,Z')}\geq \exp\bp{-t^*mw}}\nonumber\\
    &\leq \EE{\exp\bp{t^*\sum_{Z'\in\cB(Z)}\Delta_n(Z,Z')}\cdot \exp\bp{t^*mw}}\nonumber\\
    &= \EE{\exp\bp{t^*\sum_{Z'\in\cB(Z)}\Delta_n(Z,Z')}\cdot \exp\bp{\eps m\bar nI/2}},\label{eq:finalterm}
    \end{align}
    The first inequality holds because the minimum is smaller than the average. The second inequality is due to Markov's inequality. We now proceed to bound \eqref{eq:finalterm} by $\exp(-(1-o(1))m\bar nI)$.

    We first define set $\mathcal C(Z) = \{i:Z_i=Z_i^*\}$, which is the set of samples that are correctly classified. Thus, we have $\abs{\mathcal C(Z)} = \sum_{a\in [K]}R_{aa}=n-m$. Suppose $Z'\in \cB(Z)$ corrects $k$th sample from a misclassified group $a$, where $Z_k=a$, to its true group $b$, where $Z_k'=b$. Then, we must have $k\in [n]\setminus \mathcal C(Z)$, and by Lemma \ref{lm:pathdiffformula}, we can rewrite
    \begin{align*}
        \Delta_n(Z,Z')&= \sum_{i\in [n]}(A_{ik}\indc{Z_i'=Z_k}-\lambda^*) - \sum_{i\in [n]}(A_{ik}\indc{Z_i=Z_k'}-\lambda^*)\\
& = \underbrace{\sum_{i\in \mathcal C(Z)}(A_{ik}\indc{Z_i'=Z_k}-\lambda^*) - \sum_{i\in \mathcal C(Z)}(A_{ik}\indc{Z_i=Z_k'}-\lambda^*)}_{(A_k)} \\
&+ \underbrace{\sum_{i\not\in \mathcal C(Z)}(A_{ik}\indc{Z_i'=Z_k}-\lambda^*) - \sum_{i\not \in \mathcal C(Z)}(A_{ik}\indc{Z_i=Z_k'}-\lambda^*)}_{(B_k)}.
    \end{align*}
    Here, $A_k$, $B_k$ correspond to summations in \eqref{eq:chernoff} and \eqref{eq:chernoff2} respectively. We further have
    \begin{align*}
        \sum_{Z'\in \cB(Z)}\Delta_n(Z,Z') = \sum_{k\in [n]\setminus\mathcal{C}(Z)}\bp{A_{k}+B_{k}}.
    \end{align*}
    It is obvious that $A_{k}\perp A_{j}$ for $k,j\in [n]\setminus\mathcal{C}(Z)$, and $\sum_{k\in [n]\setminus\mathcal C(Z)}A_{k}$ can be written as the independent sum of the random variable $A_{ij}$ for some $i\in [n]\setminus\mathcal C(Z)$ and $j\in C(Z)$. As for $\sum_{k\in [n]\setminus\mathcal C(Z)} B_k$, it is the summation of $A_{ij}$ for some $i,j\in [n]\setminus\mathcal C(Z)$. For each random variable $A_{ij}$, the coefficient is at most 2 (since it can only be added twice or canceled out), and the total number of random variables is at most $m\choose2$. Hence, by the argument from \eqref{eq:chernoff} to \eqref{eq:Pi0ratioPi0'}, we can bound \eqref{eq:finalterm} by 
    \begin{align*}
        \eqref{eq:finalterm}\leq \exp\bp{-\bp{1-C\gamma-\eps/2}m\bar nI},
    \end{align*}
    for some constant $C$. By the definition of $g(Z)$ defined in \eqref{eq:defGZ}, we have
    \begin{equation}\label{eq:PknownPathDiff}
    \begin{split}
    \PP{\max_{Z\in \cG(\gamma_0):m\leq \gamma n}\frac{\Pi_0(Z|A)}{\Pi_0(g(Z)|A)}\geq \exp(-\eps \bar nI)} &= \PP{\max_{Z\in \cG(\gamma_0):m\leq \gamma n}\min_{Z'\in \cB(Z)}\Delta_n(Z,Z')\geq -w}\\
    &\leq \sum_{m=1}^{\gamma n}{n\choose m}(K-1)^m\exp\bp{-\bp{1-C\gamma-\eps/2}m\bar nI}\\
    & \leq \sum_{m=1}^{\gamma n} \bp{{enK}\exp\bp{-\bp{1-C\gamma-\eps/2}\bar nI}}^m\\
    & = n \exp\bp{-\bp{1-\eps/2}\bar nI (1-o(1))},
    \end{split}
    \end{equation}
    where we require $\eps<2\eps_0$ in order for the last equation going to 0 as $n$ tends to infinity.

\item[Large mistake region.] For $m>\gamma n$, recall that $\cG(\gamma_0)\subset S_\alpha$, and $S_\alpha$ is defined in \eqref{eq:S_alpha}. If $Z'$ corrects one sample from group $a$ to group $b$, by \eqref{eq:DeltaZZ'knownP}, we have $\dn_s = n_a'-n_b'-1$. By \ref{eq:DeltaZZ'knownP}, we have $\Delta_n(Z,Z')=\dO_s-\lambda^*\dn_s$. Let $\lambda = (p+q)/2$, $m_b=\sum_{k\neq b}R_{kb}$, $n_a'=n_a(Z)$, and $n_b' = n_b(Z)$ for simplicity. Thus, it follows that
\begin{align*}
\EE{\Delta_n(Z,Z')} &=\EE{\dO_s}-\lambda^* \dn_s = \EE{\dO_s}-\lambda\dn_s + (\lambda-\lambda^*) \dn_s\\
& = -\frac{p-q}{2}(n_a'+n_b'-2m_b-2R_{ab})+ \frac{p+q-2\lambda^*}{2}(n_a'-n_b) - (p-\lambda^*)\\
&\leq  -\frac{p-q}{2}(n_a'+n_b'-2m)+\frac{p+q-2\lambda^*}{2}(n_a'-n_b')\\
& = -\frac{p-q}{2}(n_a'+n_b'-2m-C_\lambda (n_a'-n_b'))\\
& = -\frac{p-q}{2}(n_b'+C_\lambda n_b'+(1-C_\lambda)n_a'-2m)\\
&\leq -\bp{\frac{n}{K\alpha}-m}\bp{p-q},
\end{align*}
where $C_\lambda = 2(\lambda-\lambda^*)/(p-q)$. It is easy to verify that $C_\lambda\in (0,1)$, and $\cl$ tends to 1 (resp. tends to 0) when $(p-q)/p$ tends to 1 (resp. tends to 0). If $\gamma_0$ satisfies that $(1-K\alpha\gamma_0)^2nI\goto \infty$, it follows that
\begin{equation}\label{eq:boundknownlargeeq}
    \max_{Z\in \cG(\gamma_0):m>\gamma n}\max_{Z'\in \cB(Z)\cap S_\alpha}\bp{\EE{\dO_s}-\lambda^* \dn_s}\leq -\bp{\frac{1}{K\alpha}-\gamma_0}n(p-q)(1-o(1)).
\end{equation}
Denote $\delta_0=1/K\alpha-\gamma_0$ for simplicity. Then, it follows that $\eqref{eq:boundknownlargeeq}\leq - \delta_0n(p-q)(1-o(1))$. Since $\delta_0^2nI\goto \infty$, there exist some positive sequences $\gamma$ and $\theta$ such that $\gamma,\theta \goto 0$, $\theta^2\gamma nI\goto \infty$, and $\theta\ll \delta_0$. To be specific, we may take $\gamma = 1/(\delta_0\sqrt{nI})$, $\theta = \delta/(\delta\sqrt{nI})^{1/4}$. Hence, we can construct the event $\mathcal E_6(\gamma, \theta)$ as defined in \eqref{eq:allevents}. Note that $X(Z)-X(Z^*) = \dO(Z)-\EE{\dO(Z)}$. Under the event $\mathcal E_6$, we have
\begin{align*}
    & \max_{Z\in S_\alpha:m>\gamma n}\min_{Z'\in \cB(Z)\cap S_\alpha}\abs{\dO_s-\EE{\dO_s}}\\
\leq & \max_{Z\in S_\alpha:m>\gamma n}\frac{1}{\abs{\cB(Z)\cap S_\alpha}}\sum_{Z'\in \cB(Z)\cap S_\alpha}\abs{\dO_s-\EE{\dO_s}}\\
\leq & \max_{Z\in S_\alpha:m>\gamma n}\frac{1}{\abs{\cB(Z)\cap S_\alpha}}\sum_{Z'\in \cB(Z)\cap S_\alpha}\sum_{a\in[K]}\abs{\dO_{aa}-\EE{\dO_{aa}}}\\
\leq & \max_{Z\in S_\alpha:m>\gamma n}\frac{1}{\abs{\cB(Z)\cap S_\alpha}}\sum_{Z'\in \cB(Z)\cap S_\alpha}\sum_{a\leq a'}\abs{X_{aa'}(Z)-X_{aa'}(Z')}\\
\leq &\theta n(p-q).
\end{align*}
Then, it follows that
\begin{equation}\label{eq:laterRefpqknownlarge}
\begin{split}
    &\max_{Z\in \cG(\gamma_0): m>\gamma n}\min_{Z'\in \cB(Z)\cap S_\alpha}\Delta_n(Z,Z')\\
=&\max_{Z\in \cG(\gamma_0): m>\gamma n}\min_{Z'\in \cB(Z)\cap S_\alpha}\bp{\dO_s-\lambda^*\dn_s}\\
\leq & \max_{Z\in \cG(\gamma_0): m>\gamma n}\min_{Z'\in \cB(Z)\cap S_\alpha}\bp{\dO_s-\EE{\dO_s}}+\max_{Z\in S_\alpha:m>\gamma n}\max_{Z'\in \cB(Z)\cap S_\alpha}\bp{\EE{\dO_s}-\lambda^* \dn_s}\\
\leq &\theta n (p-q) - \delta_0 n(p-q)(1-o(1))\\
= &-\delta_0(1-o(1))n(p-q).
\end{split}
\end{equation}
By the definition of $g(Z)$ and by \eqref{eq:referLogPi0Ratio}, we have that
\begin{align*} 
\max_{Z\in \cG(\gamma_0):m>\gamma n}\log\frac{\Pi_0(Z|A)}{\Pi_0(g(Z)|A)} &= \log\frac{p(1-q)}{q(1-p)}\left[\max_{Z\in \cG(\gamma_0):m>\gamma n}\min_{Z'\in \cB(Z)\cap S_\alpha}\Delta_n(Z,Z')\right]\\
&\leq -\log\bp{\frac{p}{q}}\delta_0n(p-q)(1-o(1)).
\end{align*}
Furthermore, since
\begin{align*}
I  = (\sqrt{p}-\sqrt{q})^2(1+o(1)),~~~\frac{(p-q)\log\bp{\frac{p}{q}}}{\bp{\sqrt{p}-\sqrt{q}}^2}\geq 4,
\end{align*}
we have
\[
    \log\bp{\frac{p}{q}}(p-q)\geq 4(\sqrt{p}-\sqrt{q})^2 = 4I(1-o(1)).
\]
Then, it follows directly
\begin{align*}
\max_{Z\in \cG(\gamma_0): m>\gamma n}\frac{\Pi_0(Z|A)}{\Pi_0(g(Z)|A)} \leq \exp\bp{-{4\delta_0nI}(1-o(1))}.
\end{align*}
By Lemma \ref{lm:pathdifflargeMis}, we have that $\mathcal E_6$ happens with probability at least $1-\exp(-n)$.
\end{description}
Combining results of two regions directly gives the result of Lemma \ref{lm:knownPincrease}.
\end{proof}

\subsection{Proof of Lemma \ref{lm:PathIncrease} with unknown connectivity probabilities}\label{sec:pqunknown}

It suffices to prove the following lemma when the connectivity probabilities are unknown.
\begin{lemma}\label{lm:unknownPincrease}
    Suppose $\gamma_0$ satisfies Condition \ref{con:gamma0}, \ref{con:gamma0xi_small}, or \ref{con:pqknowncon}. Then, there exists some positive sequence $\gamma\goto 0$ such that, with probability at least $1-C_1n^{-C_2}$,
    \begin{align*}
    \frac{\Pi(Z|A)}{\Pi(g(Z)|A)}\leq \begin{dcases}
    \exp\bp{-\eps\bar nI(1-o(1))}, &\text{if~}m\leq \gamma n,\\
    \exp\bp{-\frac{(1-K\gamma_0)^4nI}{2\alpha^2}(1-o(1))}, &\text{if~}m> \gamma n,
    \end{dcases}
    \end{align*}
    holds uniformly for all $Z\in \cG(\gamma_0)$ defined in \eqref{eq:defineGoodRegion}. Here, $\eps$ is any constant satisfying $\eps<\eps_0$, and $C_3,C_4$ are constants depending on $\eps$. Furthermore, if $\xi$ satisfies Condition \ref{con:xi} or \ref{con:gamma0xi_small} correspondingly, then by choosing $\eps\in\bp{(1-\eps_0)/\xi,2\eps_0}$, we have
    \begin{align*}
    \max_{Z\in \cG(\gamma_0)}\frac{\tPi(Z|A)}{\tPi(g(Z)|A)}\leq \exp(-C\bar nI),
    \end{align*}
    for some constant $C>1-\eps_0$ with probability at least $1-C_3n^{-C_4}$.
\end{lemma}
Note that when the connectivity probabilities are unknown, the initial conditions are different for the case of two communities and the case of more than two communities. In order to prove Lemma \ref{lm:unknownPincrease}, we again divide $S_\alpha$ into a small mistake region and a large mistake region, according to whether $m>\gamma n$, where $\gamma\goto 0$ is a positive sequence to be specified later. It is worth noting that we always start from the likelihood modularity, and then bound the exact posterior distribution.

\begin{proof}[Proof of Lemma \ref{lm:unknownPincrease}]
Under the conditions of Lemma \ref{lm:unknownPincrease}, let $\eps_{\gamma_0} = 1-K\gamma_0$ for simplicity, and we have $\eps_{\gamma_0}^4nI\goto\infty$, $\eps_{\gamma_0}(1-K\beta\gamma_0)n\goto\infty$. Then, for any positive sequences $\beps,\gamma,\theta\goto0$ satisfying that $\beps^2nI\goto \infty$, $\theta^2\gamma nI\goto \infty$, $\eps_{\gamma_0}^2\gg \beps$, and $\eps_{\gamma_0}^3\gg \theta\beps$. To be specific, we can set $\beps = \eps_{\gamma_0}^2/(\eps_{\gamma_0}^4nI)^{1/4}$, $\gamma = \eps_{\gamma_0}^2$, $\theta = 1/\sqrt{\gamma nI}$.

All the following analyses are based on the event $\cE(\beps,\gamma,\theta)$.
\begin{description}[leftmargin = 0cm,labelsep = 0.5cm]
\item[Small mistake region.] We write $\cM_s = \bb{Z\in \cG(\gamma_0): m\leq \gamma n}$. By Lemma \ref{lm:size}, since $\gamma<c_{\alpha,\beta}$, we have that for any $Z\in \cM_s$, $\cB(Z)\subset S_\alpha$. By Lemma \ref{lm:likemodePathdiff}, under the event $\mathcal E_1(\beps)$, we have
\begin{equation}\label{eq:boundQfirst}
\begin{split}
&\max_{Z\in M_s}\min_{Z'\in \cB(Z)}\log \frac{\Pi(Z|A)}{\Pi(Z'|A)}\\
\leq& \max_{Z\in M_s}\frac{1}{m}\sum_{Z'\in \cB(Z)}\log \frac{\Pi(Z|A)}{\Pi(Z'|A)}\\
\leq& \max_{Z\in M_s}\frac{1}{m}\sum_{Z'\in \cB(Z)}\bp{Q_{LM}(Z,A)-Q_{LM}(Z',A)}+\eps_{LM}.    
\end{split}
\end{equation}
Thus, we proceed to upper bound $Q_{LM}(Z,A)-Q_{LM}(Z',A)$. By some calculations, we have
\begin{equation}\label{eq:unknownSmall}
    \begin{split}
        &Q_{LM}(Z,A)-Q_{LM}(Z',A)\\
= &\log \frac{\Pi_0(Z|A)}{ \Pi_0(Z'|A)} 
-\sum_{a\leq a'}n_{aa'} \cdot \KL{\frac{O_{aa'}}{n_{aa'}}}{\frac{O_{aa'}(Z)}{n_{aa'}(Z)}}+\\
&\underbrace{\sum_{a\leq a'}\dO_{aa'} \bp{\log \frac{O_{aa'}(Z)}{n_{aa'}(Z)}-\log B_{aa'}}+ (\dn_{aa'}-\dO_{aa'}) \bp{\log \frac{n_{aa'}(Z)-O_{aa'}(Z)}{n_{aa'}(Z)}-\log(1- B_{aa'})}}_{Err(Z,Z')},
    \end{split}
\end{equation}
where $\log\left[\Pi_0(Z|A)/\Pi_0(Z'|A)\right]$ is calculated in \eqref{eq:referLogPi0Ratio}.
\ignore{\[\log \Pi_0(Z|A)-\log \Pi_0(Z'|A) = \log \frac{p(1-q)}{q(1-p)}\bp{\dO_s-\lambda^* \dn_s}.\]}Now suppose we correct one sample from a misclassified group $b$ to its true group $b'$. Then, by Lemma \ref{lm:pathdiffformula}, we have
\begin{align*}
&\dO_{b'b'}+\dO_{bb} + \dO_{bb'} = 0, ~~\dO_s= \dO_{bb}+\dO_{b'b'},\\
&\dO_{ab}+\dO_{ab'}=0,~~\dO_{aa'}=0,~~~\text{for~any~} a,a'\in[K]\setminus\{b,b'\}.
\end{align*}
Denote $\tB_{aa'} = {\EE{O_{aa'}(Z)}}/{n_{aa'}(Z)}$ and $\hat B_{aa'} = {{O_{aa'}(Z)}}/{n_{aa'}(Z)}$ for any $a,a'\in[K]$. By Lemma \ref{boundtildePtrueP}, we have
\begin{equation}\label{eq:bias}
\begin{split}
    \frac{\abs{\tB_{aa'} -B_{aa'}}}{p-q} = \begin{dcases}
\frac{\sum_{k\neq l}R_{ak}R_{al}}{n'_{a}(n'_a-1)}, &\text{if~} a=a',\\[10pt]
\frac{\sum_{k}R_{ak}R_{a'k}}{n'_{a} n'_{a'}}, &\text{if~} a\neq a',
\end{dcases}
\end{split}
\end{equation}
and $\Norm{\tB-B}_{\infty}\leq {2K\alpha m(p-q)}/{n}$. Under the event $\mathcal E_1(\beps)$ defined in \eqref{eq:allevents}, by Lemma \ref{tildePhatP}, we have
\begin{equation}\label{eq:BhatminusB}
\Norm{\hat B-B}_{\infty} \leq \Norm{\hat B-\tB}_{\infty}+\Norm{\tB-B}_{\infty}\leq \bp{C\beps+\frac{2K\alpha m}{n}}(p-q)\lesssim (\gamma+\beps) (p-q).    
\end{equation}

We then bound $Err(Z,Z')$ in \eqref{eq:unknownSmall} under the event $\mathcal E(\beps,\gamma,\theta)$. Since $p\asymp q$, by some calculations, we have
\begin{align}
\sum_{Z'\in\cB(Z)}Err(Z,Z') &= \sum_{a\leq a'}\log\frac{\hat B_{aa'}(1-B_{aa'})}{B_{aa'}(1-\hat B_{aa'})}\sum_{Z'\in\cB(Z)} \bp{\dO_{aa'}-\lambda^*_{aa'}\dn_{aa'}}\\
 &\leq \frac{C}{p}{\Norm{\hat B-B}_{\infty}}\underbrace{\sum_{a\leq a'}\abs{\sum_{Z'\in\cB(Z)}\bp{\dO_{aa'}-\lambda^*_{aa'}\dn_{aa'}}}}_{(A)} \label{eq:A}
\end{align}
for some constant $C$, where
\begin{align*}
\lambda_{aa'}^* = \log \frac{1-B_{aa'}}{1-\hat B_{aa'}}\Big / \log \frac{\hat B_{aa'}(1-B_{aa'})}{B_{aa'}(1-\hat B_{aa'})}\in \left[ B_{aa'}\wedge \hat B_{aa'},B_{aa'}\vee \hat B_{aa'} \right].
\end{align*}
Under the event $\mathcal E(\beps,\gamma, \theta)$, for any $Z\in \cM_s$, we bound the above term $(A)$ by the following three terms separately,
\begin{align}
\sum_{a\leq a'}\abs{\sum_{Z'\in\cB(Z)}\bp{\dO_{aa'}-\EE{\dO_{aa'}}}}&\lesssim mn(p-q),\\
\sum_{a\leq a'}\abs{\sum_{Z'\in\cB(Z)}\bp{\EE{\dO_{aa'}}-B_{aa'}\dn_{aa'}}}&{\leq \sum_{Z'\in\cB(Z)}\sum_{a\leq a'}\abs{\EE{\dO_{aa'}}-B_{aa'}\dn_{aa'}}}\leq 2Kmn(p-q),\label{eq:refsamearg}\\
\sum_{a\leq a'}\abs{\sum_{Z'\in\cB(Z)}(B_{aa'}-\lambda^*_{aa'})\dn_{aa'}} & \leq \sum_{Z'\in\cB(Z)}\sum_{a\leq a'}\abs{\dn_{aa'}}\cdot\Norm{\hat B-B}_{\infty}\lesssim (\gamma+\beps) mn(p-q).\label{eq:finalref}
\end{align}
The first inequality directly follows by Lemma \ref{lm:pathdiffall}. The second inequality is due to that for each fixed $Z'$, there are at most $2K$ pairs of groups contributing to the summations of the absolute values, and for each summation, there are at most $n$ random variables associated. The third inequality follows by \eqref{eq:BhatminusB} and Lemma \ref{lm:pathdiffformula}. Hence, under the event $\mathcal E(\beps,\gamma, \theta)$, we have
\begin{align*}
\sum_{Z'\in\cB(Z)}Err(Z,Z')\leq C(\gamma+\beps) m\bar n I,
\end{align*}
for some constant $C$ where $\gamma,\beps\goto 0$ as defined in the beginning of the proof. Hence, it follows that
\begin{equation}\label{eq:refunknownSmallRef}
\begin{split}
    &\PP{\max_{Z\in \cM_s} \min_{Z'\in\cB(Z)}\log \frac{\Pi(Z|A)}{ \Pi(Z'|A) }>-\eps \bar nI}\\
\leq &\PP{\max_{Z\in \cM_s}\frac{1}{m}\sum_{Z'\in\cB(Z)}\log\frac{\Pi(Z|A)}{\Pi(Z'|A)} >-\eps \bar nI }\\
\leq &\PP{\max_{Z\in \cM_s}\frac{1}{m}\sum_{Z'\in\cB(Z)}Q_{LM}(Z,A)-Q_{LM}(Z',A) >-\eps \bar nI-\eps_{LM}, \mathcal E} + \PP{\mathcal E^c}\\
\leq&\PP{\max_{Z\in \cM_s}\frac{1}{m}\sum_{Z'\in\cB(Z)}\bp{\log \frac{\Pi_0(Z|A)}{\Pi_0(Z'|A)}+ Err(Z,Z')}>-\eps \bar nI-\eps_{LM}, \mathcal E}+ \PP{\mathcal E^c}\\
\leq&\PP{\max_{Z\in \cM_s}\frac{1}{m}\sum_{Z'\in\cB(Z)}\log \frac{\Pi_0(Z|A)}{\Pi_0(Z'|A)}>-\eps\bar nI-C(\gamma+\beps) \bar nI-\eps_{LM}} +  \PP{\mathcal E^c}\\
=&\PP{\max_{Z\in \cM_s}t^*\sum_{Z'\in\cB(Z)}\Delta_n(Z,Z')>-m\eps\bar nI/2-Cm(\gamma+\beps)\bar nI/2-m\eps_{LM}/2} +  \PP{\mathcal E^c}.
\end{split}
\end{equation}
where $\eps_{LM}\goto 0$ is defined in \eqref{eq:likemodeDiffPath}, and $\eps$ is any small constant satisfying $\eps<2\eps_0$. A simple union bound and following the argument from \eqref{eq:finalterm} to \eqref{eq:PknownPathDiff} lead to that
\[\PP{\max_{Z\in \cM_s} \min_{Z'\in\cB(Z)}\log \frac{\Pi(Z|A)}{ \Pi(Z'|A) }>-\eps \bar nI}\leq C'n\exp\bp{-(1-\eps/2)\bar nI(1-o(1))},\]
for some constant $C'$.

\begin{remark}
Before performing the analysis for the large mistake region, it is worth noting that for the small mistake region, the proof works for any sequence $\gamma\goto 0$. Thus, in the case of more than two communities ($K\geq 3$), if $\gamma_0\goto 0$, by Lemma \ref{lm:stayingood}, $\cG(\gamma_0)\subset \cM_s$ for some sequence $\gamma\goto0$, and then the proof is complete. Therefore, we only need to analyze the large mistake region for $K=2$.
\end{remark}

\item[Large mistake region.] We write $\cM_l = \{Z\in \cG(\gamma_0):m>\gamma n\}$. By the same argument in \eqref{eq:boundQfirst}, we start with $Q_{LM}(Z,A)-Q_{LM}(Z',A)$. For $K=2$ and $Z\in \cM_l$, by some calculations, we have
\begin{equation}\label{eq:likemode_unknown}
    \begin{split}
    &Q_{LM}(Z,A)-Q_{LM}(Z',A) \\[10pt]
=&\sum_{a\leq a'}n_{aa'}(Z)\tau\bp{\frac{O_{aa'}(Z)}{n_{aa'}(Z)}}-\sum_{a\leq a'}n_{aa'}(Z')\tau\bp{\frac{O_{aa'}(Z')}{n_{aa'}(Z')}}\\[10pt]
=& \sum_{a\leq a'}\dO_{aa'} \log \frac{O_{aa'}(Z)}{n_{aa'}(Z)}+\bp{\dn_{aa'}-\dO_{aa'}}\log \bp{1-\frac{O_{aa'}(Z)}{n_{aa'}(Z)}} - \sum_{a\leq a'}n_{aa'}(Z')\KL{\frac{O_{aa'}(Z')}{n_{aa'}(Z')}}{\frac{O_{aa'}(Z)}{n_{aa'}(Z)}}\\[10pt]
\leq & \underbrace{\sum_{a\leq a'}\dO_{aa'}\log \tB_{aa'}+ \bp{\dn_{aa'}-\dO_{aa'}}\log \bp{1-\tB_{aa'}}}_{P(Z,Z')+ Err_1(Z,Z')} +\\[10pt]
&\underbrace{\sum_{a\leq a'}\dO_{aa'}\bp{\log \frac{O_{aa'}(Z)}{n_{aa'}(Z)}-\log \tB_{aa'}}+ \bp{\dn_{aa'}-\dO_{aa'}}\bp{\log \bp{1-\frac{O_{aa'}(Z)}{n_{aa'}(Z)}}-\log \bp{1-\tB_{aa'}}}}_{Err_2(Z,Z')},     
    \end{split}
\end{equation}
where we write
\begin{align}
    P(Z,Z') = & \sum_{a\leq a'}\EE{\dO_{aa'}}\log \tB_{aa'}+ \bp{\dn_{aa'}-\EE{\dO_{aa'}}}\log \bp{1-\tB_{aa'}},\label{eq:PZZ'}\\
    Err_1(Z,Z') = & \sum_{a\leq a'}\bp{\dO_{aa'}-\EE{\dO_{aa'}}}\log \frac{\tB_{aa'}}{1-\tB_{aa'}}.\label{eq:Err1} 
\end{align}
Recall $g(Z)$ and $\cB(Z)$ are defined in \eqref{eq:defGZ} and \eqref{eq:defABN} respectively. Let $N=|B(Z)\cap S_{\alpha}|$. By Lemma \ref{lm:size}, we have $N\geq \min\{cn, m\}$ for some constant $c$.

\begin{description}[leftmargin = 0cm,labelsep = 0.5cm]
    
\item[Step 1: bound $P(Z,Z')$.]

    \ignore{
    Suppose $Z'$ corrects one sample from group $b$ to group $b'$, i.e., $R_{Z'}(b,b') = R_Z(b,b')-1$, $R_{Z'}(b',b') = R_{Z}(b',b')+1$. Then, by Lemma \ref{lm:boundPZZ'}, since $K=2$, we have
\begin{equation}
\begin{split}
    &-P(Z,Z') \\[7pt]
     = &\frac{p-q}{n'_{b'}}\sum_{a=1}^K \log\frac{\tB_{ab'}(1-\tB_{ab})}{\tB_{ab}(1-\tB_{ab'})}\cdot \sum_{l\neq b'} R_{b'l} (R_{ab'}-R_{al}) + \sum_{a=1}^K (n'_a-\delta_{ab'})\cdot \KL{\tB_{ab'}}{\tB_{ab}}\\[7pt]
 = &\frac{R_{b'b}(p-q)}{n'_{b'}}\sum_{a\in\{b,b'\}}\left[\log\bp{\frac{\tB_{ab'}(1-\tB_{ab})}{\tB_{ab}(1-\tB_{ab'})}}(R_{ab'}-R_{ab})\right]+\sum_{a\in\{b,b'\}} (n'_a-\delta_{ab'})\cdot \KL{\tB_{ab'}}{\tB_{ab}},
\end{split}
\end{equation}
where $\delta_{ab'}=1$ if and only if $a=b'$, otherwise $\delta_{ab'}=0$. }

By Lemma \ref{lm:boundPZZ'2}, for $Z\in \cM_{l}$ with $m= d(Z,Z^*)$ and any $Z'\in \cB(Z)\cap S_\alpha$, we have
\[
    -P(Z,Z')\geq \frac{nI}{2\alpha^2}\eps_{m}^3\max\{1-\beta+1/\alpha,\eps_{m}\}(1-o(1))\gtrsim \eps_{m}^3nI,
\]
where $\eps_{m}=1-Km/n$. Note that $m<n/K\beta$ under the condition. The last inequality holds because $\eps_{m}\goto 0$ only if $\beta\goto 1$, and $m/n\goto 1/K$. Thus, $1-\beta+1/\alpha\gg \eps_{m}$. 
\ignore{
\bigskip

\item[Conclusion of Step 1.] We have 
\[
    \max_{Z\in \cM_l}\max_{Z'\in \cB(Z)}P(Z,Z')\leq -C (1-Km/n)^3nI,
\]
for some constant $C$.
}
\bigskip

\item[Step 2: bound $Err_1(Z,Z')$.] Recall that $N = |\cB(Z)\cap S_\alpha|$. Under the event $\mathcal E_6(\gamma,\theta)$ defined in \eqref{eq:allevents}, we have 
    \begin{equation*}
    {\max_{Z\in S_\alpha: m> \gamma n} \frac{1}{N}\sum_{Z'\in \cB(Z)\cap S_\alpha}\sum_{a\leq a'}\abs{\dO_{aa'}-\EE{\dO_{aa'}}}\leq \theta n(p-q)},
    \end{equation*}
    for the positive sequence $\theta\goto 0$ defined in the beginning of the proof. When $K=2$, and $Z'$ corrects one sample from group $b$ to $b'$, by Lemma \ref{lm:pathdiffformula}, $Err_1(Z,Z')$ in \eqref{eq:Err1} can be expressed as
    \begin{align*}
    Err_1(Z,Z') & = (\dO_{bb}-\EE{\dO_{bb}})\log \frac{\tB_{bb}(1-\tB_{b'b})}{\tB_{b'b}(1-\tB_{bb})}+ (\dO_{b'b'}-\EE{\dO_{b'b'}})\log \frac{\tB_{b'b'}(1-\tB_{bb'})}{\tB_{bb'}(1-\tB_{b'b'})}.
    \end{align*}
    By $p\asymp q$ and $(n-2m)(n-2\beta m)/n\goto\infty$, it follows by Lemma \ref{lm:tildeBminus} that
    \begin{align*}
    \frac{1}{N}\sum_{Z'\in \cB(Z)\cap S_\alpha} Err_1(Z,Z') &\leq C_1\frac{\theta n(p-q)}{p}\cdot \bp{\abs{\tB_{b'b'}-\tB_{bb'}}+\abs{\tB_{bb}-\tB_{b'b}}}\\[4pt]
    &\leq C_2\frac{\theta n(p-q)^2\det(R)}{p}\bp{\frac{2\alpha}{n}}^3\bp{\abs{R_{b'b'}-R_{b'b}}+\abs{R_{bb}-R_{bb'}}}\\[4pt]
    &\leq C_3\frac{\theta I\det(R)(n-2m)}{n^2},
    \end{align*}
    for some constants $C_1,C_2,C_3$. Hence, under the event $\mathcal E(\beps, \gamma, \theta)$, where the sequences are defined in the beginning of the proof, if $(n-2m)(n-2\beta m)/n\goto \infty$ and $(n-2m)/n\gg \theta$, then by Lemma \ref{Hbbeachm} we have
\begin{align*}
        \frac{1}{N}\sum_{Z'\in \cB(Z)\cap S_\alpha} Err_1(Z,Z')\ll \frac{\theta I\det(R)(n-2m)}{n^2}\ll
        -\frac{1}{N}\sum_{Z'\in\cB(Z)\cap S_\alpha}P(Z,Z'),
\end{align*}
for any $Z\in \cM_l$.

\bigskip

\item[Step 3: bound $Err_2(Z,Z')$.] Recall $N = |\cB(Z)\cap S_\alpha|$. By \eqref{eq:likemode_unknown}, we have
    \begin{align*}
\frac{1}{N}\sum_{Z'\in \cB(Z)\cap S_\alpha}Err_2(Z,Z') &= \frac{1}{N}\sum_{a\leq a'}\log\frac{\hat B_{aa'}(1-\tB_{aa'})}{\tB_{aa'}(1-\hat B_{aa'})}\sum_{Z'\in\cB(Z)\cap S_\alpha} \bp{\dO_{aa'}-\lambda^*_{aa'}\dn_{aa'}}\\
 &\leq \frac{C}{p}{\Norm{\hat B-\tB}_{\infty}}\sum_{a\leq a'}\abs{\frac{1}{N}\sum_{Z'\in\cB(Z)\cap S_\alpha}\bp{\dO_{aa'}-\lambda^*_{aa'}\dn_{aa'}}},
\end{align*}
where
\begin{align*}
\lambda^*_{aa'} = \log \frac{1-\tB_{aa'}}{1-\hat B_{aa'}}\Big / \log \frac{\hat B_{aa'}(1-\tB_{aa'})}{\tB_{aa'}(1-\hat B_{aa'})}\in \left[ \tB_{aa'}\wedge \hat B_{aa'},\tB_{aa'}\vee \hat B_{aa'} \right].
\end{align*}
Under the event $\cal E(\beps,\gamma,\theta)$ defined in \eqref{eq:allevents}, by Lemma \ref{tildePhatP}, we have $\max_{Z\in S_\alpha}\Norm{\hat B-\wt B}_{\infty}\lesssim \beps(p-q)$. By the same argument from \eqref{eq:BhatminusB} to \eqref{eq:finalref}, we have
    \begin{align*}
    \frac{1}{N}\sum_{a\leq a'}\abs{\sum_{Z'\in \cB(Z)\cap S_\alpha}\bp{\dO_{aa'}-\EE{\dO_{aa'}}}}&\leq \theta n (p-q), \\
    \frac{1}{N}\sum_{a\leq a'}\abs{\sum_{Z'\in \cB(Z)\cap S_\alpha}\bp{\EE{\dO_{aa'}}-\tB_{aa'}\dn_{aa'}}}&\leq C(n-2m)(p-q)\lesssim \eps_m n (p-q), \\
    \frac{1}{N}\sum_{a\leq a'}\abs{\sum_{Z'\in \cB(Z)\cap S_\alpha}(\tB_{aa'}-\lambda_{aa'})\dn_{aa'}}&\leq K^2 n\Norm{\hat B-\tB}_{\infty} \lesssim \beps n(p-q).
    \end{align*}
It follows that for any $Z\in \cM_l$ with $d(Z,Z^*) = m$, 
\[
    \frac{1}{N}\sum_{Z'\in \cB(Z)\cap S_\alpha}Err_2(Z,Z') \lesssim \frac{\beps n(p-q)^2}{p}\bp{\theta +\eps_m+\beps},
\]
where $\eps_{m} = 1-2m/n$. Hence, under the event $\mathcal E(\beps, \gamma, \theta)$, where the sequences are defined in the beginning of the proof. If $\eps_{m}^3\gg \theta\beps$, $\eps_{m}^2\gg \beps$, and $\eps_{m}^3\gg \beps^2$, then we have 
\[
    \frac{1}{N}\sum_{Z'\in \cB(Z)\cap S_\alpha}Err_2(Z,Z')\ll -\frac{1}{N}\sum_{Z'\in \cB(Z)\cap S_\alpha}P(Z,Z'),
\]
for any $Z\in \cM_l$.
\end{description}
\bigskip

By combining all three steps, we require that for all $Z\in \cG(\gamma_0)$ with $d(Z,Z^*)=m$, 
\begin{equation}\label{eq:deltaCondition}
    (1-2m/n)^4 nI\goto \infty,~~(n-2m)(n-2\beta m)/n\goto \infty ~(\text{for Lemma \ref{lm:boundPZZ'2}}).
\end{equation}
By the definition of $\cG(\gamma_0)$ in \eqref{eq:defineGoodRegion} and by Lemma \ref{lm:stayingood}, it suffices to require 
\[
    (1-2\gamma_0)^4nI\goto\infty, ~~~(1-2\gamma_0)(1-2\beta\gamma_0)n\goto\infty.
\]

Recall that $\eps_{\gamma_0} = 1-2\gamma_0$ in the beginning of the proof. Then, under the event $\mathcal E(\beps,\gamma,\theta)$ defined in \eqref{eq:allevents}, by the conclusions of three steps and by Lemma \ref{lm:likemodePathdiff}, we have
\begin{equation}\label{eq:finalUnknownRef}
\begin{split}
    &\max_{Z\in \cM_l}\min_{Z'\in\cB(Z)\cap S_\alpha}\log \frac{\Pi(Z|A)}{\Pi(Z'|A)}\\
\leq &\max_{Z\in \cM_l}\frac{1}{N}\sum_{Z'\in\cB(Z)\cap S_\alpha}\log \frac{\Pi(Z|A)}{\Pi(Z'|A)}\\
\leq &\max_{Z\in \cM_l}\frac{1}{N}\sum_{Z'\in\cB(Z)\cap S_\alpha}(Q_{LM}(Z,A)-Q_{LM}(Z',A)+\eps_{LM})\\
=& \max_{Z\in \cM_l}\frac{1}{N}\sum_{Z'\in\cB(Z)\cap S_\alpha}\bb{P(Z,Z')+Err_1(Z,Z')+Err_2(Z,Z')}+\eps_{LM}\\
=& \max_{Z\in \cM_l}\frac{1}{N}\sum_{Z'\in\cB(Z)\cap S_\alpha}P(Z,Z')(1-o(1))+\eps_{LM}\\
\leq& -\frac{nI}{2\alpha^2}\eps_{\gamma_0}^3\max\{1-\beta+1/\alpha,\eps_{\gamma_0}\}(1-o(1))\\
\leq &-\frac{nI}{2\alpha^2}\eps_{\gamma_0}^4(1-o(1)).
\end{split}
\end{equation}

\end{description}
Combine the results of two regions, and Lemma \ref{lm:unknownPincrease} directly follows.
\end{proof}

\subsection{Proof of Lemma \ref{lm:randomwalk_probratio}}
\label{sec:randomwalk_probratio}
For the simplicity of presentation, we first introduce some notations that will be used in the proof. Denote
\begin{align*}
\wt m = n\max\{\gamma_0,n^{-\tau}\}+\log^2 n, ~~~m^* = n^{1-\tau},
\end{align*}
where $\tau$ is a sufficiently small constant defined in \eqref{eq:defineGoodRegion}. For any $Z\in \cG(\gamma_0)$, we define the following set
\begin{equation}
    \mathcal S(Z,\eta) =\bb{S'\subset \mathcal{N}(Z)\cap S_\alpha:|S'|\geq \eta \abs{{\cB}(Z)\cap S_\alpha}},
\end{equation}
where $\eta$ is a small constant satisfying $2\tau<\eta<1/2$, and $\cA(Z),\cB(Z),\mathcal{N}(Z)$ are defined in \eqref{eq:defABN}. Then, we have the following lemma.
\begin{lemma}\label{lm:randomWalkLemma}  
Suppose $\gamma_0,\xi$ satisfy all conditions for Theorem \ref{thm:mixing} or those for Theorem \ref{thm:mixingknown}, with probability at least $1-\exp(-n^{1-\tau})$, we have
\[
    \max_{Z\in S_\alpha:m^* \leq  m\leq \wt m}\max_{S'\in \mathcal S(Z,\eta)}\min\bb{\min_{Z'\in S'\cap \cB(Z)}\frac{\tPi(Z|A)}{\tPi(Z'|A)},\min_{Z'\in S'\cap \cA(Z)}\frac{\tPi(Z'|A)}{\tPi(Z|A)}} \leq \exp(-C\bar nI),
\]
for some constant $C>1-\eps_0$, and $2\tau<\eta<1/2$.
\end{lemma} 

To understand Lemma \ref{lm:randomWalkLemma}, we say that $Z'\in \cN(Z)$ is \textit{making a mistake} if $Z'\in \cB(Z)$ but rejected, or $Z'\in \cA(Z)$ but accepted. Lemma \ref{lm:randomWalkLemma} implies that under all the conditions, for any current state $Z$ with $m\in[m^*,\wt m]$, if we make at least $\eta|\cB(Z)\cap S_\alpha|$ different choices of $Z'$, then there is at least one $Z'$ such that it is not making a mistake with high probability. In other words, it holds with high probability that $Z'$ will make less than $\eta|\cB(Z)\cap S_\alpha|$ mistakes among all possible choices.

Though Lemma \ref{lm:randomWalkLemma} seems very similar to Lemma \ref{lm:knownPincrease}, Lemma \ref{lm:knownPincrease} works for all $Z\in\cG(\gamma_0)$, and focuses on the posterior ratio of the current state and the next possible state in set $\cB(Z)$, while Lemma \ref{lm:randomWalkLemma} works for $Z\in\cG(\gamma_0)$ with $m\in [m^*,\wt m]$, and also bounds the probability of updating to $\cA(Z)$.

\begin{proof}[Proof of Lemma \ref{lm:randomwalk_probratio}]
By Lemma \ref{lm:randomWalkLemma}, with probability at least $1-\exp(-n^{1-\tau})$, for any $Z\in \cG(\gamma_0)$ with $m\in [m^*,\wt m]$, by \eqref{eq:defpmqm}, we have that 
\begin{align*}
p_m(Z)&=p(Z,\cA(Z)) = \frac{1}{2(K-1)n} \sum_{Z'\in \cA(Z)\cap S_\alpha}\min \bb{1,\frac{\tPi_g(Z'|A)}{\tPi_g(Z|A)}}\leq \frac{\eta |\cB(Z)\cap S_\alpha|+\eps}{2(K-1)n},\\
q_m(Z)&=p(Z,\cB(Z)) = \frac{1}{2(K-1)n} \sum_{Z'\in \cB(Z)\cap S_\alpha}\min \bb{1,\frac{\tPi_g(Z'|A)}{\tPi_g(Z|A)}}\geq \frac{(1-\eta)|\cB(Z)\cap S_\alpha|}{2(K-1)n},
\end{align*}
where $\eps=n\exp(-C\bar nI)\goto 0$ for some $C>1-\eps_0$. It follows that with probability at least $1-\exp(-n^{1-\tau})$, ${p_m(Z)}/{q_m(Z)}\leq 2\eta$ holds for any $Z\in \cG(\gamma_0)$ with $m\in [m^*,\wt m]$.
\end{proof}

In order to prove Lemma \ref{lm:randomWalkLemma}, we first state two lemmas according to whether the connectivity probabilities are known or not. In the case of known connectivity probabilities, we use $\Pi_0(\cdot|A)$ to denote the posterior distribution.
\begin{lemma}\label{lm:rw_case1}
    When $p,q$ are both known, given $\tau$ sufficiently small and $\eta$ satisfying $2\tau<\eta<1/2$, if $(1-K\alpha\gamma_0)^2nI\goto \infty$, then we have
    \begin{align*}
    &\max_{S'\in \mathcal S(Z,\eta)}\min\bb{\min_{Z'\in S'\cap \cB(Z)}\frac{\Pi_0(Z|A)}{\Pi_0(Z'|A)},\min_{Z'\in S'\cap \cA(Z)}\frac{\Pi_0(Z'|A)}{\Pi_0(Z|A)}}\\
    \leq &\begin{dcases}
    \exp\bp{-\eps\bar nI}, &\text{if~} m^*\leq m\leq \gamma n, \\
    \exp\bp{-{4(1/K\alpha-\gamma_0)nI}(1-o(1))}, &\text{if~} \gamma n< m\leq \wt m,
    \end{dcases}
    \end{align*}
    holds uniformly for all $Z\in \cG(\gamma_0)$ with $m\in [m^*, \wt m]$ and some sequence $\gamma\goto 0$, with probability at least $1-\exp\bp{-n^{1-\tau}}$. Here, $\eps$ is any small constant satisfying $\eps<2\eps_0$.
\end{lemma}

\begin{lemma}\label{lm:rw_case2}
Given $\tau$ sufficiently small and $\eta$ satisfying $2\tau<\eta<1/2$. Suppose $\gamma_0$ satisfies Condition \ref{con:gamma0} or \ref{con:gamma0xi_small}, there exists some positive sequence $\gamma\goto 0$ such that with probability at least $1-\exp\bp{-n^{1-\tau}}$,
\begin{align*}
&\max_{S'\in \mathcal S(Z,\eta)}\min\bb{\min_{Z'\in S'\cap \cB(Z)}\frac{\Pi(Z|A)}{\Pi(Z'|A)},\min_{Z'\in S'\cap \cA(Z)}\frac{\Pi(Z'|A)}{\Pi(Z|A)}}\\
\leq &\begin{dcases}
\exp\bp{-\eps\bar nI(1-o(1))}, &\text{if~} m^*\leq m\leq \gamma n,\\
\exp\bp{-\frac{(1-K\gamma_0)^4nI}{2\alpha^3}(1-o(1))}, &\text{if~} \gamma n<m \leq \wt m,
\end{dcases}
\end{align*}
    holds uniformly for all $Z\in \cG(\gamma_0)$ with $m\in [m^*, \wt m]$. Here, $\eps$ is any constant satisfying $\eps<2\eps_0$.
\end{lemma}

The proofs of Lemma \ref{lm:rw_case1} and Lemma \ref{lm:rw_case2} will be presented in the sequel. We first proceed to prove Lemma \ref{lm:randomWalkLemma} based on these two lemmas.
\begin{proof}[Proof of Lemma \ref{lm:randomWalkLemma}]
The result directly follows Lemma \ref{lm:rw_case1} and Lemma \ref{lm:rw_case2} by choosing $\xi$ properly in Theorem \ref{thm:mixing} or Theorem \ref{thm:mixingknown}. Then, by choosing $\eps\in\bp{(1-\eps_0)/\xi,2\eps_0}$, we have that
    \begin{equation*}
        \max_{Z\in S_\alpha:m^* \leq  m\leq \wt m}\max_{S'\in \mathcal S(Z,\eta)}\min\bb{\min_{Z'\in S'\cap \cB(Z)}\frac{\tPi(Z|A)}{\tPi(Z'|A)},\min_{Z'\in S'\cap \cA(Z)}\frac{\tPi(Z'|A)}{\tPi(Z|A)}} \leq \exp(-C\bar nI),
    \end{equation*}
    for some constant $C>1-\eps_0$ with probability at least $1-\exp(-n^{1-\tau})$ for the sufficiently small constant $\tau$.
\end{proof}

We finally present the proofs of Lemma \ref{lm:rw_case1} and Lemma \ref{lm:rw_case2} to complete this section.
\begin{proof}[Proof of Lemma \ref{lm:rw_case1}]
We consider any positive sequences $\gamma, \theta$ satisfying $\gamma,\theta\goto0$, $\theta^2\gamma nI\goto\infty$, and $\theta\ll 1-K\alpha\gamma_0$. Suppose $\gamma n\in [m^*,\wt m]$, and we perform analyses for the following cases.
\begin{description}[leftmargin = 0cm,labelsep = 0.5cm]
    \item[Case 1: $m^*\leq m\leq \gamma n$.] Since the minimum is upper bounded by the average, it follows that 
\begin{align}
&\log\max_{Z:m^*\leq m\leq \gamma n}\max_{S'\in \mathcal S(Z,\eta)}\min\bb{\min_{Z'\in S'\cap \cB(Z)}\frac{\Pi_0(Z|A)}{\Pi_0(Z'|A)},\min_{Z'\in S'\cap \cA(Z)}\frac{\Pi_0(Z'|A)}{\Pi_0(Z|A)}}\label{eq:random}\\
 \leq &\max_{Z:m^*\leq m\leq \gamma n}\max_{S'\in \mathcal S(Z,\eta)}\frac{1}{|S'|}\bp{\sum_{Z'\in S'\cap \cB(Z)}\log \frac{\Pi_0(Z|A)}{\Pi_0(Z'|A)}+\sum_{Z'\in S'\cap \cA(Z)}\log \frac{\Pi_0(Z'|A)}{\Pi_0(Z|A)}}\nonumber\\
 = &\max_{Z:m^*\leq m\leq \gamma n}\max_{S'\in \mathcal S(Z,\eta)}\frac{2t^*}{|S'|}\bp{\sum_{Z'\in S'\cap \cB(Z)}\Delta_n(Z,Z')+\sum_{Z'\in S'\cap \cA(Z)}\Delta_n(Z',Z)}\nonumber,
\end{align}
where $\Delta_n(Z,Z')$ is defined in \eqref{eq:DeltaZZ'knownP}. By a similar argument from \eqref{eq:finalterm} to \eqref{eq:PknownPathDiff}, we have
\begin{align*}
\EE{t^*\bp{\sum_{Z'\in S'\cap \cB(Z)}\Delta_n(Z,Z')+\sum_{Z'\in S'\cap \cA(Z)}\Delta_n(Z',Z)}}\leq \exp(-(1-C\gamma)|S'|\bar nI).
\end{align*} 
We also have $|\cB(Z)\cap S_\alpha| = m$ by Lemma \ref{lm:size}. For any small constant $\eps<2\eps_0$, write $C_{\gamma,\eps} = 1-C\gamma -\eps/2$ for simplicity. It follows that
\begin{align}
\PP{\eqref{eq:random}\geq -\eps \bar nI}&\leq \underbrace{\sum_{m^*\leq m \leq \gamma n}{n\choose m}(K-1)^m}_{\text{all possible $Z$}}~~~\underbrace{\sum_{|S'|\geq \eta m}{Kn\choose |S'|}}_{\text{all possible $S'$}}~~~\underbrace{\exp\bp{-C_{\gamma,\eps}|S'|\bar nI}}_{\text{bound for each given $Z$ and $S'$}}\nonumber\\
&\lesssim  \sum_{m^*\leq m\leq \gamma n}{n\choose m}(K-1)^m{Kn\choose \eta m}\exp\bp{-C_{\gamma,\eps}\eta m\bar nI}\nonumber\\
&\lesssim  {n\choose m^*}(K-1)^{m^*}{Kn\choose \eta m^*}\exp\bp{-C_{\gamma,\eps}\eta m^*\bar nI}\nonumber\\
&\lesssim \exp\bp{-C_{\gamma,\eps}\eta m^*\bar nI+(\eta+1) m^* \log \frac{eKn}{\eta m^*}}. \label{eq:exp_eta}
\end{align}
For any sufficiently small $\tau$, when $\eta>2\tau$ and $m^*=n^{1-\tau}$, it is easy to check that
\[
    \eqref{eq:exp_eta} \lesssim \exp\bp{-m^*\bp{C_{\gamma,\eps}\eta \bar nI - (\eta+1)\log\frac{eKn^{\tau}}{\eta}}}\leq e^{-m^*}.
\]

\item[Case 2: $\gamma n<m\leq \wt m$.] Recall that $\dO_s(Z,Z') = O_s(Z)-O_s(Z')$ for any label assignments $Z$ and $Z'$. By \eqref{eq:referLogPi0Ratio} and \eqref{eq:boundknownlargeeq}, we have that 
\begin{align*}
  \log\frac{\Pi_0(Z|A)}{\Pi_0(Z'|A)}&=2t^*\Delta_n(Z,Z'),\\
  \max_{Z\in\cG(\gamma_0)}\max_{Z'\in \cB(Z)\cap S_\alpha}\EE{\Delta_n(Z,Z')}&\leq -\bp{\frac{1}{K\alpha}-\gamma_0}(p-q)(1-o(1)).
\end{align*}
Since $|S'|\geq \eta |\cB(Z)\cap S_\alpha|$, by the same proof in Lemma \ref{lm:pathdifflargeMis}, we have that with probability at least $1-\exp(-n)$, 
\begin{align*}
&\max_{Z\in S_\alpha: \gamma n <m\leq \wt m}\max_{S'\in S(Z,\alpha)}\bb{\frac{1}{|S'|}\sum_{Z'\in S'}\abs{\Delta_n(Z,Z')-\EE{\Delta_n(Z,Z')}}}\\
=&\max_{Z\in S_\alpha: \gamma n <m\leq \wt m}\max_{S'\in S(Z,\alpha)}\bb{\frac{1}{|S'|}\sum_{Z'\in S'}\abs{\dO_s(Z,Z')-\EE{\dO_s(Z,Z')}}}\\
\leq &\theta n(p-q),
\end{align*} 
where the positive sequence $\theta$ is defined in the beginning of the proof. Thus, by a similar argument from \eqref{eq:boundknownlargeeq} to \eqref{eq:laterRefpqknownlarge}, the result directly follows.
\end{description}

Combining two cases gives Lemma \ref{lm:rw_case1} directly. Note that in the case of $m^*\geq \gamma n$ or $\wt m<\gamma n$, the result trivially follows.
\end{proof}
\begin{proof}[Proof of Lemma \ref{lm:rw_case2}]
Consider any positive sequences $\beps,\gamma,\theta\goto0$ satisfying that $\beps^2nI\goto \infty$, $\theta^2\gamma nI\goto \infty$, $(1-K\gamma_0)^2\gg \beps$, and $(1-K\gamma_0)^3\gg \theta\beps$. Note that the second case in Lemma \ref{lm:rw_case2} is only for the case of $K=2$. When $K\geq 3$, we require $\gamma_0\goto 0$, and thus there exists some $\gamma\goto 0$, such that for all $Z\in \cG(\gamma_0)$, $m\leq \gamma n$ always holds.

    The following proof is similar with those of Lemma \ref{lm:unknownPincrease} and Lemma \ref{lm:rw_case1}. Denote $\Delta Q(Z,Z') = Q_{LM}(Z,A)-Q_{LM}(Z',A)$ for any two label assignments $Z,Z'\in S_\alpha$, where $Q_{LM}(Z,A)$ is the likelihood modularity function defined in \eqref{eq:defineLikeMode}. Under the event $\mathcal E_1(\beps)$, by Lemma \ref{lm:likemodePathdiff}, there exists some sequence $\eps_{LM}\goto 0$ such that
    \begin{align}
        &\log \max_{S'\in \mathcal S(Z,\eta)}\min\bb{\min_{Z'\in S'\cap \cB(Z)}\frac{\Pi(Z|A)}{\Pi(Z'|A)},\min_{Z'\in S'\cap \cA(Z)}\frac{\Pi(Z'|A)}{\Pi(Z|A)}}\nonumber\\
        \leq &\max_{S'\in \mathcal S(Z,\eta)}\frac{1}{|S'|}\bb{\sum_{Z'\in S'\cap \cB(Z)}\log\frac{\Pi(Z|A)}{\Pi(Z'|A)}+\sum_{Z'\in S'\cap \cA(Z)}\log \frac{\Pi(Z'|A)}{\Pi(Z|A)}}\nonumber\\
        \leq &\max_{S'\in \mathcal S(Z,\eta)}\frac{1}{|S'|}\bb{\sum_{Z'\in S'\cap \cB(Z)}\Delta Q(Z,Z')+\sum_{Z'\in S'\cap \cA(Z)}\Delta Q(Z',Z)} + \eps_{LM}. \label{eq:hereToBound}
    \end{align}
    The first inequality is because minimum is smaller than the average. 

\begin{description}[leftmargin = 0cm,labelsep = 0.5cm]
    \item[Case 1: $m^*\leq m\leq\gamma n$.] In this case, $B(Z)\subset S_{\alpha}$ by Lemma \ref{lm:size}. By \eqref{eq:unknownSmall}, we have that under the event $\cE_1(\beps)$,
    \begin{align*}  
        \eqref{eq:hereToBound}\leq & \max_{S'\in \mathcal S(Z,\eta)}\frac{1}{|S'|}\bb{\sum_{Z'\in S'\cap \cB(Z)}\log\frac{\Pi_0(Z|A)}{\Pi_0(Z'|A)} +\sum_{Z'\in S'\cap \cA(Z)}\log \frac{\Pi_0(Z'|A)}{\Pi_0(Z|A)} + \sum_{Z'\in S'}\abs{Err(Z,Z')}} + \eps_{LM}.
    \end{align*}    
    By a similar argument from \eqref{eq:unknownSmall} to \eqref{eq:refunknownSmallRef}, in order to prove Lemma \ref{lm:rw_case2}, it suffices to show that 
    \begin{equation}\label{eq:eqS'}
        \max_{Z\in\cG(\gamma_0):m^*\leq m\leq \gamma n}\max_{S'\in \mathcal S(Z,\eta)}\frac{1}{|S'|}\sum_{Z'\in S'}\abs{Err(Z,Z')} = o(\bar nI).
    \end{equation}
    Recall that $Err(Z,Z')$ is defined in \eqref{eq:unknownSmall}. It follows by \eqref{eq:A} that under the event $\cE_1(\beps)$, for any $Z\in \cG(\gamma_0)$ with $m\in [m^*, \gamma n]$, 
    \begin{align*}
       & \max_{S'\in \mathcal S(Z,\eta)}\frac{1}{|S'|}\sum_{Z'\in S'}\abs{Err(Z,Z')}\\
\lesssim & (\gamma+\beps)\frac{p-q}{p} \max_{S'\in \mathcal S(Z,\eta)} \frac{1}{|S'|} \sum_{Z'\in S'} \sum_{a\leq a'}\abs{\dO_{aa'}-\lambda_{aa'}^*\dn_{aa'}}\\
        \lesssim &(\gamma+\beps)\frac{p-q}{p} \bb{(A)+(B)+(C)},
    \end{align*}
    where
    \begin{align*}
        (A) &= \max_{S'\in \mathcal S(Z,\eta)} \frac{1}{|S'|} \sum_{Z'\in S'} \sum_{a\leq a'}\abs{\dO_{aa'}-\EE{\dO_{aa'}}}\lesssim n(p-q),\\
        (B) &= \max_{S'\in \mathcal S(Z,\eta)} \frac{1}{|S'|} \sum_{Z'\in S'} \sum_{a\leq a'}\abs{\EE{\dO_{aa'}}-B_{aa'}\dn_{aa'}}\lesssim n(p-q),\\
        (C) &= \max_{S'\in \mathcal S(Z,\eta)} \frac{1}{|S'|} \sum_{Z'\in S'} \sum_{a\leq a'}\abs{(B_{aa'}-\lambda_{aa'}^*)\dn_{aa'}}\ll n(p-q).\\
    \end{align*}
    The first inequality holds with probability at least $1-e^{-m^*}$ by the same proof of Lemma \ref{lm:rw_case1} and Lemma \ref{lm:pathdiffall}. The second and the third inequalities hold due to the same arguments for \eqref{eq:refsamearg} and \eqref{eq:finalref}. Hence, the proof is complete for the small mistake region.

\item[Case 2: $\gamma n<m\leq \wt m$.] We only analyze this case for $K=2$. By \eqref{eq:likemode_unknown}, we have that under the event $\cE_{1}(\beps)$,
\begin{equation*}
\begin{split}
    \eqref{eq:hereToBound}\leq & \max_{S'\in \mathcal S(Z,\eta)}\frac{1}{|S'|}\bb{\sum_{Z'\in S'\cap \cB(Z)}P(Z,Z')+\sum_{Z'\in S'\cap \cA(Z)}P(Z',Z)} + \\
    &\max_{S'\in \mathcal S(Z,\eta)}\frac{1}{|S'|}\sum_{Z'\in S'}\bp{\abs{Err_1(Z,Z')}+\abs{Err_2(Z,Z')}} + \eps_{LM},
\end{split}
\end{equation*}
where $Err_1(Z,Z')$ and $Err_2(Z,Z')$ are defined in \eqref{eq:likemode_unknown} and \eqref{eq:Err1}. According to arguments in Lemma \ref{lm:PathIncrease}, we only need to bound $Err_1(Z,Z')$ and $Err_2(Z,Z')$ in order to upper bound bound \eqref{eq:hereToBound} as well as the posterior ratio. The only term inside $Err_1(Z,Z')$ and $Err_2(Z,Z')$ needed to be treated specially is 
\[
    \max_{Z\in S_\alpha: \gamma n<m\leq \wt m} \max_{S'\in \mathcal S(Z,\eta)}\frac{1}{|S'|}\sum_{Z'\in S'}\sum_{a\leq a'}\abs{\dO_{aa'}(Z,Z')-\EE{\dO_{aa'}(Z,Z')}},
\]
denoted by $D$ for simplicity. By the same proof of Lemma \ref{lm:pathdifflargeMis}, we have that
\[
    \PP{D\geq \theta n(p-q)} \leq e^{-n},
\]
for the positive sequence $\theta$ defined in the beginning of the proof. Hence, following the same arguments from \eqref{eq:likemode_unknown} to \eqref{eq:finalUnknownRef}, the proof of Lemma \ref{lm:rw_case2} is completed for the large mistake region.
\end{description}
Combining two cases gives Lemma \ref{lm:rw_case2} directly. Note that in the case of $m^*\geq \gamma n$ or $\wt m<\gamma n$, the result trivially follows.
\end{proof}

\subsection{Proof of Lemma \ref{lm:minpost}}
In this section, we proceed to lower bound the posterior distribution.
\subsubsection{When connectivity probabilities are known}\label{sec:knownPlowerbound}
Let $\bb{x_i}_{i\geq 1},\bb{y_j}_{j\geq 1}$ be i.i.d. copies of $\text{Bernoulli}(q)$ and Bernoulli$(p)$. According to \eqref{eq:postknown} and \eqref{eq:knownPostEq}, for any $Z\in S_\alpha$, we have that
    \begin{align*}
    \log \frac{\Pi_0(Z|A)}{\Pi_0(Z^*|A)}\ignore{&= \log\frac{p(1-q)}{q(1-p)}\bp{\dtO_s-\lambda^* \dtn_s}\\}
    & = \log\frac{p(1-q)}{q(1-p)}\bp{\dtO_s-\EE{\dtO_s}+\EE{\dtO_s}-\lambda^* \dtn_s},
    \end{align*}
    where $\dtO_s = \sum_{i=1}^{N_d} x_i- \sum_{i=1}^{N_s}y_i$, and 
    \begin{equation}\label{eq:NsNd}
    N_d  = \sum_{i<j}\indc{Z_i=Z_j,Z^*_i\neq Z_j^*},~~N_s  = \sum_{i<j}\indc{Z_i^*=Z_j^*,Z_i\neq Z_j}.
    \end{equation} 
    Recall that $m_k=\sum_{i=1}^n\indc{Z_i=k,~Z_i^*\neq k}$, and it follows that $m = \sum_{k\in [K]}m_k$. Write $\beta_k = \sum_{i<j} \indc{Z_i=Z_j=k,Z_i^*\neq Z_j^*}$ for simplicity, and we have
\begin{align*}
\beta_k &= \sum_{a<b} R_{ka}R_{kb} \leq R_{kk}\cdot m_k + m_k^2 = n_k'm_k\leq \frac{\alpha n m_k}{K}.
\end{align*}
It follows that $N_d=\sum_{k=1}^K \beta_k\leq {\alpha mn}/{K}$. Similarly, we have $N_s\leq {\beta mn}/{K}$, and
    \begin{align*}
    \EE{\dtO_s}-\lambda^*\dtn_s &=N_d q -N_s p - \lambda^*(N_d-N_s) \\
    & = -(N_d\cdot (\lambda^*-q)+N_s\cdot (p-\lambda^*))\\
    & \geq -(N_d+N_s)\cdot (p-q) \geq -(\alpha+\beta)mn(p-q)/K.
    \end{align*}
    Furthermore, by Lemma \ref{boundXabe_xabc}, with probability at least $1-n\exp(-(1-o(1))\bar nI)$, for any $Z\in S_\alpha$ with $m=d(Z,Z^*)$, we have 
    \begin{align*}
        \abs{\dtO_s - \EE{\dtO_{s}}}&\leq \sum_{a\in[K]}\abs{\dtO_{aa} - \EE{\dtO_{aa}}}
        = \sum_{a\in[K]}\abs{X_{aa}(Z) - X_{aa}(Z^*)}\\
        &\leq K\Norm{X(Z)-X(Z^*)}_{\infty} \leq (\alpha+\beta)mn(p-q).
    \end{align*}
Hence, it follows that
    \begin{align*}
    \log \frac{\Pi_0(Z|A)}{\Pi_0(Z^*|A)} &\geq -\log \frac{p(1-q)}{q(1-p)}\cdot Cmn(p-q)\\
    & \geq - Cmn\cdot \frac{(p-q)^2}{q(1-p)}\\
    & \geq - C'nmI\cdot (1+o(1)),
    \end{align*}
    for some constants $C$ and $C'$ with probability at least $1-n\exp(-(1-o(1))\bar nI)$.
\subsubsection{When connectivity probabilities are unknown}
In order to simplify the proof, we first define some events, and all the following analysis are conditioning on the given events. For any positive sequences $\gamma,\theta\goto 0$ with $\gamma ^2 nI\goto \infty$, and $\theta^2\gamma nI\goto \infty$, let $\beps=\gamma$ for simplicity. Consider events $\mathcal E_1(\beps)$, $\mathcal E_2$, $\mathcal E_3(\beps)$, $\mathcal E_4(\gamma,\theta)$ defined in \eqref{eq:allevents}. Under the events $\cE_1(\beps)$ and $\cE_3(\beps)$, by Lemma \ref{lm:prooflikemode}, we have that for any $Z\in S_\alpha$,
\[
    \log\frac{\Pi(Z|A)}{\Pi(Z^*|A)}-\Delta Q(Z,Z^*)\geq -C_{LM},
\]
where $\Delta Q(Z,Z^*) = Q_{LM}(Z,A)-Q_{LM}(Z^*,A)$. Thus, it suffices to lower bound $\Delta Q(Z,Z^*)$.

\begin{description}[leftmargin = 0cm,labelsep = 0.5cm]
\item[Case 1: $m\leq\gamma n$.] By \eqref{eq:ZZstarsmall}, we have
\begin{align*}
    &\Delta Q(Z,Z^*)\\[10pt]
    = &\underbrace{\log \frac{\Pi_0(Z|A)}{\Pi_0(Z^*|A)}}_{(A)}-\underbrace{\sum_{a\leq b}n_{ab}(Z^*) \cdot \KL{\frac{O_{ab}}{n_{ab}}}{\frac{O_{ab}(Z)}{n_{ab}(Z)}}}_{{(B)}}\\[10pt]
    &+\underbrace{\sum_{a\leq b}\dtO_{ab} \bp{\log \frac{O_{ab}(Z)}{n_{ab}(Z)}-\log B_{ab}}+ (\dtn_{ab}-\dtO_{ab}) \bp{\log \frac{n_{ab}(Z)-O_{ab}(Z)}{n_{ab}(Z)}-\log(1- B_{ab})}}_{{(C)}},
\end{align*}
where $\Pi_0(\cdot|A)$ is defined in \eqref{eq:knownPostEq}. We proceed to bound each term above separately. Under the event $\cE_2$ and by the same argument in Section \ref{sec:knownPlowerbound}, we have
\[
    A\gtrsim  -mnI.
\]
By Lemma \ref{boundmsmallunknown}, under the events $\cE_1(\beps)$ and $\cE_2$, we have
\[
    B\lesssim m^2I.
\]
By Lemma \ref{lm:bound_ec}, under the events $\cE_1(\beps)$ and $\cE_2$, we have
\[
    C\gtrsim -(\beps +\gamma)mnI.
\]
Hence, under the events $\mathcal E_1(\beps),\mathcal E_2, \mathcal E_3(\beps)$, we have that for any $Z\in S_\alpha$ with $m\leq \gamma n$, 
    \[
        \log\frac{\Pi(Z|A)}{\Pi(Z^*|A)}\geq \Delta Q(Z,Z^*)-C_{LM}\geq -CmnI,
    \]
    for some constant $C$.
    
    \item[Case 2: $m>\gamma n$.] By \eqref{eq:pqunknownLargeZZstar}, we have
    \begin{align*}
    \Delta Q(Z,Z^*) = G(Z) - G(Z^*) + \Delta(Z) -\Delta(Z^*),
    \end{align*}
    where by \eqref{eq:GZ-GZ*} and Lemma \ref{minPiR}, we have
    \begin{align*}
    G(Z) - G(Z^*) \geq -\frac{1}{2}\sum_{a,b,k,l} R_{ak}R_{bl}\KL{B_{kl}}{\tB{ab}}
    \gtrsim - mnI.
    \end{align*}
    By Lemma \ref{lm:boundDeltaec}, under the events $\mathcal E_1(\beps)$ and $\mathcal E_4(\gamma,\theta)$, for any $Z\in S_\alpha$ with $m>\gamma n$, we have
    \begin{align*}
        \Delta(Z)-\Delta(Z^*)\geq -\eps mnI
    \end{align*}
    for some sequence $\eps\goto 0$. Hence, under the events $\cE_1(\beps)$, $\cE_3(\beps)$, and $\cE_4(\gamma, \theta)$, we have that for any $Z\in S_\alpha$ with $m>\gamma n$, 
    \begin{align*}
    \log\frac{\Pi(Z|A)}{\Pi(Z^*|A)}\geq \Delta Q(Z,Z^*)-C_{LM}\geq -C'mnI,
    \end{align*}
    for some constant $C'$.
\end{description}

Combining two cases, we have that with probability at least $1-n\exp(-(1-o(1))\bar nI)$, 
\[
    \min_{Z\in S_\alpha}\bp{\log\frac{\Pi(Z|A)}{\Pi(Z^*|A)}+CmnI}\geq 0,
\]
for some constant $C$. By Theorem $\ref{thm:postcontract}$, we have $\log\Pi(Z^*|A)\geq C$ for some constant $C$ with high probability. To conclude, there exists some constants $C_3$, $C_4$ and $C_5$ such that for any $Z\in S_\alpha$, 
\[
    \min_{Z\in S_\alpha}\bp{\log{\Pi(Z|A)}+C_3mnI}\geq 0.
\]


\subsection{Bounding probability of events}

We first introduce some notations. For any $Z\in S_\alpha$ and any $a\in [K]$, we use $n_a'$ to denote $n_a(Z)$, $m_a$ to denote $n_a(Z)-R_{aa}$ for simplicity.

\begin{lemma}\label{lm:boundhatPtildeP}
    Denote $X_{ab}(Z) = O_{ab}(Z)-\EE{O_{ab}(Z)}$ for any $a,b\in[K]$, then for a general $K\geq 2$, we have
    \begin{align*}
    \PP{\max_{Z\in S_\alpha} {\Norm{X(Z)}_\infty}\geq \bar\eps n^2(p-q)}&\leq \exp(-n),
    \end{align*}
    as long as $\beps^2nI\goto\infty$.
\end{lemma}

\begin{proof}
    For $K\geq 2$, we have $\VV{O_{ab}(Z)}\leq n_{ab}(Z) p\leq {\alpha^2 n^2 p}/{K^2}$. Then, by a union bound and Bernstein inequality, it follows that
    \begin{align*}
        \PP{\max_{Z\in S_\alpha} {\Norm{X(Z)}_\infty}\geq \beps n^2(p-q)}&\leq 2K^2 K^n \bp{\exp\bp{-\frac{\beps^2 n^2 (p-q)^2K^2}{2\alpha^2p}}+\exp\bp{-\frac{3\beps n^2 (p-q)}{2}}}\\
        &\leq 2K^{n+2} \exp\bp{-(1-o(1))\frac{\beps^2K^2n^2I}{2\alpha^2}}\\
        &\leq \exp(-n),
    \end{align*}
    for any $\beps$ satisfying that $\beps^2nI\goto \infty$.
\end{proof}

The conclusion of Lemma \ref{lm:boundhatPtildeP} directly leads to the following lemma.

\begin{lemma}\label{tildePhatP}
    For $K\geq 2$ and any $a,b\in [K]$, let $\hat B_{ab} = {O_{ab}(Z)}/{n_{ab}(Z)}$, and $\tB_{ab}={\EE{O_{ab}(Z)}}/{n_{ab}(Z)}$. Under the event $\mathcal{E}$ defined in \eqref{eq:defE}, we have
    \[\max_{Z\in S_\alpha}\Norm{\hat B-\wt B}_{\infty}=\max_{Z\in S_\alpha}\max_{a,b\in [K]}\frac{X_{ab}(Z)}{n_{ab}(Z)}\leq \frac{\beps n^2 (p-q)}{({K\alpha}/{n})^2} = C\beps(p-q),\]
    for some constant C depending on $K$ and $\alpha$.
\end{lemma}

We state the following lemmas whose proofs will be given together.

\begin{lemma}\label{boundXabe_xabc}
\begin{align*}
    \PP{\max_{Z\in S_\alpha}\Norm{X(Z)-X(Z^*)}_{\infty}-(\alpha+\beta) mn(p-q)/K\geq 0}\leq n\exp(-(1-o(1))nI/K).
\end{align*}
\end{lemma}
\begin{lemma}\label{boundXabe_xabc_mlarge}
    \begin{align*}
    \PP{\max_{Z\in S_\alpha}\Norm{X(Z)-X(Z^*)}_{\infty}\geq \bar\eps n^2(p-q)}\leq \exp(-n),
    \end{align*}
    as long as $\beps^2nI\goto \infty$.
\end{lemma}
\begin{lemma}\label{lm:bounddtOLargeMis}
For any positive sequences $\gamma \goto 0$, $\theta\goto0$ satisfying $\theta^2\gamma nI\goto \infty$, we have
\[
    \PP{\max_{Z\in S_\alpha: m>\gamma n}\Norm{X(Z)-X(Z^*)}_{\infty}\geq \theta mn(p-q)}\leq \exp(-n).
\]
\end{lemma}

\begin{proof}[Proofs of Lemmas \ref{boundXabe_xabc}, \ref{boundXabe_xabc_mlarge}, and \ref{lm:bounddtOLargeMis}]
Recall that $X_{ab}(Z)-X_{ab}(Z^*) = \dtO_{ab}-\EE{\dtO_{ab}}$ for $a,b\in [K]$. We first consider the case where $a\neq b$, and it follows that
    $\dtO_{ab} = \sum_{i,j\in S_1}A_{ij}-\sum_{i,j\in S_2}A_{ij}$, where
    \[|S_1|=\sum_{i,j\in [n]}\bp{\indc{Z_i=a,Z_j=b}-\indc{Z_i=Z^*_i=a,Z_j=Z^*_j=b}},\]
    \[|S_2| = \sum_{i,j\in[n]}\bp{\indc{Z^*_i=a,Z^*_j=b}-\indc{Z_i=Z^*_i=a,Z_j=Z^*_j=b}}.\]
    Therefore, we have
    \[|S_1| = n_a'n_b'-R_{aa}R_{bb}\leq m_an_b'+m_bn_a'\leq \frac{\alpha mn}{K}.\]
    Similarly, we have $|S_2|\leq{\beta mn}/{K}$. Thus, $\VV{\dtO_{ab}}\leq {(\alpha+\beta)mnp}/{K}$. A similar argument gives that for any $a\in[K]$, $\VV{\dtO_{ab}}\leq {(\alpha+\beta)mnp}/{K}$ also holds. Then, by a union bound and Bernstein inequality, we have
    \begin{align*}
        &\PP{\max_{Z\in S_\alpha }\Norm{X(Z)-X(Z^*)}_{\infty}- (\alpha+\beta) mn(p-q)/K\geq 0}\\
        \leq & 2K^2 \sum_{Z:m<cn/K} {n\choose m} (K-1)^m \exp\bp{-(1-o(1))mnI/K}\\
        \leq & n\exp(-(1-o(1))nI/K),
    \end{align*}
    for some constant $c$, which leads to Lemma \ref{boundXabe_xabc}. 

    Since $\VV{\dtO_{ab}} \leq {(\alpha+\beta)n^2 p}/{K}$ for any $a,b\in[K]$, we also have
    \begin{align*}
    &\PP{\max_{Z\in S_\alpha}\Norm{X(Z)-X(Z^*)}_{\infty}\geq \beps n^2(p-q)}\\
    \leq & 2K^2 K^n {\exp\bp{-(1-o(1))\frac{\beps^2Kn^2(p-q)^2}{2(\alpha+\beta)^2p}}}\\
    \leq & \exp(-n),
    \end{align*}
    as long as $\beps^2nI\goto \infty$, which leads to Lemma \ref{boundXabe_xabc_mlarge}.

For any sequences $\gamma,\theta\goto 0$ satisfying $\theta^2\gamma nI\goto\infty$, we also have
\begin{align*}
\PP{\Norm{X(Z)-X(Z^*)}_{\infty}\geq \theta mn(p-q)}\leq 2K^2 \bp{\exp\bp{-\frac{\theta^2 m^2n^2(p-q)^2}{2(\alpha+\beta)mnp/K}}+\exp\bp{-\frac{3 \theta mn(p-q)}{2}}}.
\end{align*}
It follows that
\begin{equation}\label{eq:largeBoundTheta}
\begin{split}
&\PP{\max_{Z\in S_\alpha: m>\gamma n}\Norm{X(Z)-X(Z^*)}_{\infty}
\leq \theta mn(p-q)}\\
\leq &C'\sum_{m>\gamma n}{n\choose m}K^m\exp(-C\theta^2 m nI)\\
\leq & C'\sum_{m>\gamma n}\bp{\frac{Ke}{\gamma}}^m\exp(-C\theta^2 m nI)\\
\leq & C'\sum_{m>\gamma n}\exp\bp{-m\bp{C\theta^2nI -\log \frac{Ke}{\gamma}}}\\
\leq & \exp(-n).
\end{split}    
\end{equation}
The last inequality holds since $\theta^2\gamma nI\goto \infty$ and thus $\theta^2nI\gg 1/\gamma \gg \log(1/\gamma)$. It directly leads to Lemma \ref{lm:bounddtOLargeMis}.
\end{proof}

Recall that for any $a,a'\in [K]$, 
\[
    X_{aa'}(Z) = O_{aa'}(Z)-\EE{O_{aa'}(Z)},~~\dO_{aa'} = O_{aa'}(Z)-O_{aa'}(Z'),
\]
it follows that
\[
    X_{aa'}(Z)-X_{aa'}(Z') = \dO_{aa'}-\EE{\dO_{aa'}}.
\]

We state the following two lemmas and the proofs will be presented together.
\begin{lemma}\label{lm:pathdiffall}
Let $N = \abs{\cB(Z)\cap S_{\alpha}}$, and we have
\[
    \PP{\max_{Z\in S_\alpha} \frac{1}{N}\sum_{Z'\in \cB(Z)\cap S_\alpha}\sum_{a\leq a'}\abs{\dO_{aa'}-\EE{\dO_{aa'}}}\geq 10n(p-q)} \leq n\exp(-2\bar nI).
\]
\end{lemma}
\begin{lemma}\label{lm:pathdifflargeMis}
Let $N = \abs{\cB(Z)\cap S_{\alpha}}$. For any positive sequence $\gamma \goto 0$, $\theta\goto 0$, and $\theta^2 \gamma nI \goto \infty$, we have
\[
    \PP{\max_{Z\in S_\alpha: m> \gamma n} \frac{1}{N}\sum_{Z'\in \cB(Z)\cap S_\alpha}\sum_{a\leq a'}\abs{\dO_{aa'}-\EE{\dO_{aa'}}}\geq \theta n(p-q)} \leq \exp(-n).
\]
\end{lemma}
\begin{proof}[Proofs of Lemma \ref{lm:pathdiffall} and Lemma \ref{lm:pathdifflargeMis}]
In order to apply Bernstein inequality, we proceed to eliminate the absolute function. Note that $\dO_{aa'}$ depends on both the current state $Z$ and the next state $Z'$. For any $Z\in S_\alpha$, we can rewrite 
\begin{equation}\label{eq:addhboundV}
\begin{split}
    &\sum_{Z'\in \cB(Z)\cap S_\alpha}\sum_{a\leq a'}\abs{\dO_{aa'}-\EE{\dO_{aa'}}}\\
     \ignore{= &\max_{h\in \{1,-1\}^{K\times K \times N}}\sum_{Z'\in \cB(Z)\cap S_\alpha}\sum_{a\leq a'} h_{aa'}(Z')\bp{\dO_{aa'}-\EE{\dO_{aa'}}}\\}
     =& \max_{h\in \{-1,1\}^{(2K-1)\times N}}\sum_{Z'\in \cB(Z)\cap S_\alpha}\bb{\sum_{a\in[K]\setminus\{b'\}}{h_{a}(Z')\bp{\dO_{ab}-\EE{\dO_{ab}}}}+\right.\\
     &\left.\sum_{a\in [K]\setminus\{b\}}^K{h_{a+K-1}(Z')\bp{\dO_{ab'}-\EE{\dO_{ab'}}}}+ h_{2K-1}(Z')\bp{\dO_{bb'}-\EE{\dO_{bb'}}}}\\
    \triangleq &\max_{h\in \{-1,1\}^{(2K-1)\times N}} S(h)
\end{split}
\end{equation}
where $h\in \{-1,1\}^{(2K-1)\times N}$ is a matrix whose elements are either $1$ or $-1$. We use $h(Z')$ to denote a column vector of $h$ corresponding to $Z'$, and $h_a(Z')$ is the $a$th element of $h(Z')$. To prove the equality in \eqref{eq:addhboundV} holds, we suppose the next state $Z'$ updates one sample from a group $b$ to another group $b'$. Then, $\dO_{aa'}=0$ for any $a,a'\in [K]\setminus\{b,b'\}$. Hence, for any possible choice of $Z'$, $h(Z')\in \{-1,1\}^{2K-1}$, and then the equality holds.

\textbf{Claim:} for any samples $i,j$, the random variable $A_{ij}$ appears at most four times in \eqref{eq:addhboundV}. This is because $A_{ij}$ can only appear when we update sample $i$ or sample $j$. Suppose the sample $i$ is corrected, then $A_{ij}$ appears in $\dO_{Z_i,Z_j}$ and $\dO_{Z_i',Z_j}$. Thus, the claim holds, which means the same Bernoulli variable appears at most four times in \eqref{eq:addhboundV}.

By the above claim, for any matrix $h$, it trivially follows that
\[
    V(h) = \VV{S(h)}\leq 16Nnp.
\]

Hence, for any $Z\in S_\alpha$ and $w>0$, by a union bound and Bernstein inequality, we have
\begin{align*}
    &\PP{\frac{1}{N}\sum_{Z'\in \cB(Z)\cap S_\alpha}\sum_{a\leq a'}\abs{\dO_{aa'}-\EE{\dO_{aa'}}}\geq w}\\[7pt]
= & \PP{\max_{h\in\{-1,1\}^{(2K-1)\times N}}S(h)\geq Nw}\\[7pt]
\leq & \sum_{h\in\{-1,1\}^{(2K-1)\times N}}\PP{S(h)\geq Nw}\\[7pt]
    \leq & 2^{2KN}\bp{\exp\bp{-\frac{N^2w^2}{2V}}+\exp\bp{-\frac{3Nw}{2\cdot 4}}},
\end{align*}
where $V=16Nnp$. Thus, by a union bound, there exists some constant $C$ such that for $w=Cn(p-q)$, we have
\begin{align*}
    &\PP{\max_{Z\in S_\alpha} \frac{1}{N}\sum_{Z'\in \cB(Z)\cap S_\alpha}\sum_{a\leq a'}\abs{\dO_{aa'}-\EE{\dO_{aa'}}}\geq w}\\[7pt]
    \leq & \sum_{m\geq 1} {n\choose m} (K-1)^m \PP{\frac{1}{N}\sum_{Z'\in \cB(Z)\cap S_\alpha}\sum_{a\leq a'}\abs{\dO_{aa'}-\EE{\dO_{aa'}}}\geq w}\\[7pt]
    \leq & \sum_{m\geq 1} {n\choose m} (K-1)^m 2^{2KN}\bp{\exp\bp{-\frac{N^2w^2}{2V}}+\exp\bp{-\frac{3Nw}{2\cdot 4}}}\\[10pt]
    \leq & n\exp(-2\bar nI).
\end{align*}
The last inequality holds since $N \geq \min\{m,~c_{\alpha,\beta}n\}$ by Lemma \ref{lm:size}. It is easy to verify that when $C=10$, the result holds, which leads to Lemma \ref{lm:pathdiffall}.

For any positive sequences $\gamma,\theta\goto 0$ satisfying $\gamma^2\theta nI\goto\infty$, let $w = \theta n (p-q)$. Then, it follows that
\begin{align*}
    &\PP{\max_{Z\in S_\alpha:m>\gamma n} \frac{1}{N}\sum_{Z'\in \cB(Z)\cap S_\alpha}\sum_{a\leq a'}\abs{\dO_{aa'}-\EE{\dO_{aa'}}}\geq w}\\[7pt]
    \leq & \sum_{m>\gamma n} {n\choose m} (K-1)^m \PP{\frac{1}{N}\sum_{Z'\in \cB(Z)\cap S_\alpha}\sum_{a\leq a'}\abs{\dO_{aa'}-\EE{\dO_{aa'}}}\geq w}\\[7pt]
    \leq & C'\sum_{m\geq 1} {n\choose m} (K-1)^m 2^{2KN}\exp\bp{-C\theta^2NnI}\\[7pt]
    \leq & \exp(-n),
\end{align*}
where the last inequality holds by the same argument for \eqref{eq:largeBoundTheta} and Lemma \ref{lm:size}. Thus, the proof of Lemma \ref{lm:pathdifflargeMis} is complete.
\end{proof}


\subsection{Proofs of auxiliary lemmas}

\begin{lemma}\label{lm:prooflikemode}
    When the connectivity probabilities are unknown, under the events $\mathcal E_1(\beps)$ and $\mathcal E_3(\beps)$ defined in \eqref{eq:allevents}, if $p\asymp q$, then we have
    \begin{align*}
    \max_{Z\in S_\alpha}\abs{\log\frac{\Pi(Z|A)}{\Pi(Z^*|A)}-(Q_{LM}(Z,A)-Q_{LM}(Z^*,A))}\leq C_{LM}
    \end{align*}
    for some constant $C_{LM}$.
\end{lemma}
\begin{proof}
    When the connectivity probabilities are unknown, by \eqref{eq:postunknown} we have
\begin{align*}
\log\Pi(Z|A) &= \sum_{a\leq b}\log \text{Beta}( O_{ab}(Z)+\kappa_1, n_{ab}(Z)- O_{ab}(Z)+\kappa_2)+Const.
\end{align*}
By Lemma \ref{lm:boundprior}, we have 
\begin{align}
&\log\frac{\Pi(Z|A)}{\Pi(Z^*|A)} \\
\leq &Q_{LM}(Z,A)-Q_{LM}(Z^*,A) + \nonumber\\
&(\kappa_1+\kappa_2)^2\cdot\sum_{a\leq b} \bp{\frac{1}{O_{ab}(Z^*)+\kappa_1}+ \frac{1}{n_{ab}(Z^*)-O_{ab}(Z^*)+\kappa_2}+\frac{1}{n_{ab}(Z)+\kappa_1+\kappa_2}}+\label{eq:ecdiff1}\\
& (\kappa_1+\kappa_2+2)\sum_{a\leq b}\bp{\abs{\log \frac{O_{ab}(Z)+\kappa_1}{O_{ab}(Z^*)+\kappa_1}}+\abs{\log \frac{n_{ab}(Z)-O_{ab}(Z)+\kappa_2}{n_{ab}(Z^*)-O_{ab}(Z^*)+\kappa_2}}+\abs{\log \frac{n_{ab}(Z)+\kappa_1+\kappa_2}{n_{ab}(Z^*)+\kappa_1+\kappa_2}}}\label{eq:ecdiff2}.
\end{align}
Recall that $Z^*$ is the true label assignment, and $\dtO_{ab} = O_{ab}(Z)-O_{ab}(Z^*)$ for any $a,b\in [K]$. Under the events $\mathcal E_1(\beps)$ and $\mathcal E_3(\beps)$, we have
\begin{align*}
{\max_{Z\in S_\alpha}\max_{a,b}\abs{O_{ab}(Z)-\EE{O_{ab}(Z)}}\leq \beps n^2(p-q)},\\
{\max_{Z\in S_\alpha}\max_{a,b}\abs{\dtO_{ab}-\EE{\dtO_{ab}}}\leq \beps n^2(p-q)}.
\end{align*}
It is easy to check that $\sum_{a\leq b}\frac{1}{O_{ab}(Z^*)+\kappa_1}$ is the dominant term in \eqref{eq:ecdiff1}. Thus, we have
\begin{align*}
(\ref{eq:ecdiff1})\leq C K^2\cdot \frac{1}{pn^2/K^2\alpha^2}\asymp \frac{1}{n^2p},
\end{align*}
for some constant $C$. Since $\abs{\log(a/b)}\leq {|a-b|}/{|\min\bb{a,b}|}$, by a similar argument, we have
\begin{align*}
(\ref{eq:ecdiff2})\leq C'K^2 \cdot \bp{\frac{n^2p+\beps n^2p}{n^2p/K^2\alpha^2}+\frac{n^2\alpha^2/K^2}{n^2/K^2\alpha^2}} \asymp 1,
\end{align*}
for some constant $C'$. By symmetry, the same argument also applies to upper bound $\log{\Pi(Z^*|A)}-\log {\Pi(Z|A)}$. Hence, the result of Lemma \ref{lm:likemode} holds.
\end{proof}

\begin{lemma}\label{lm:likemodePathdiff}
When the connectivity probabilities are unknown, under the event $\mathcal E_1(\beps)$ defined in \eqref{eq:allevents}, if $p\asymp q$, then we have that
\begin{equation}\label{eq:likemodeDiffPath}
\max_{Z\in S_\alpha} \max_{Z'\in \cN(Z)}\left|\log \frac{\Pi(Z|A)}{\Pi(Z'|A)}-\bp{Q_{LM}(Z,A)-Q_{LM}(Z',A)}\right| \leq \eps_{LM},
\end{equation}
for some positive sequence $\eps_{LM}\goto 0$.
\end{lemma}
\begin{proof}
    Recall the definition of $\cN(Z)$ in \eqref{eq:defABN}. For any label assignment $Z\in S_\alpha$ and any $Z'\in \cN(Z)$, by Lemma \ref{lm:boundprior}, we have
\begin{align}
&\log\frac{\Pi(Z|A)}{\Pi(Z'|A)} \nonumber\\[10pt]
\ignore{\leq&\sum_{a\leq b}n_{ab}(Z)\tau\bp{\frac{O_{ab}(Z)}{n_{ab}(Z)}}-n_{ab}(Z')\tau\bp{\frac{O_{ab}(Z')}{n_{ab}(Z')}}+\nonumber\\}
\leq&Q_{LM}(Z,A)-Q_{LM}(Z',A)+\nonumber\\[10pt]
&(\kappa_1+\kappa_2)^2\cdot\sum_{a\leq b} \bp{\frac{1}{O_{ab}(Z')+\kappa_1}+ \frac{1}{n_{ab}(Z')-O_{ab}(Z')+\kappa_2}+\frac{1}{n_{ab}(Z)+\kappa_1+\kappa_2}}+\label{eq:pathdiff1}\\[10pt]
& (\kappa_1+\kappa_2+2)\sum_{a\leq b}\bp{\abs{\log \frac{O_{ab}(Z)+\kappa_1}{O_{ab}(Z')+\kappa_1}}+\abs{\log \frac{n_{ab}(Z)-O_{ab}(Z)+\kappa_2}{n_{ab}(Z')-O_{ab}(Z')+\kappa_2}}+\abs{\log \frac{n_{ab}(Z)+\kappa_1+\kappa_2}{n_{ab}(Z')+\kappa_1+\kappa_2}}}.\label{eq:pathdiff2}
\end{align}
Under the event $\mathcal E_1(\beps)$ defined in \eqref{eq:allevents}, we have
\begin{align*}
\max_{Z\in S_\alpha}\max_{a,b\in[K]}\abs{O_{ab}(Z)-\EE{O_{ab}(Z)}}\leq \beps n^2(p-q).
\end{align*}
Since $\max_{a,b\in[K]}\EE{O_{ab}(Z)}\leq n_{ab}(Z)p = \O{n^2p/K^2}$, it follows that under the event $\mathcal E_1(\beps)$, $O_{ab}(Z)=\O{n^2p}$. Since $\sum_{a\leq b}\frac{1}{O_{ab}(Z')+\kappa_1}$ is the dominant term in \eqref{eq:pathdiff1}, there exists some constant $C$ such that
\[\eqref{eq:pathdiff1}\leq \frac{C}{n^2/K^2 \cdot p}K^2\lesssim\frac{1}{n^2p}.\]
By Lemma \ref{lm:pathdiffformula}, $\abs{\dO_{ab}}\leq {2n\alpha}/{K}$ and $\abs{\dn_{ab}}\leq 2n\alpha/K$ always hold. Since $\abs{\log(x/y)}\leq \abs{x-y}/\min\{x,y\}$ for $x,y>0$, we have that
\begin{align*}
\eqref{eq:pathdiff2}\ignore{&\leq C'\sum_{a\leq b}\bp{\abs{\frac{\max_{Z\in S_\alpha}\dO_{ab}(Z)}{\min_{Z\in S_\alpha}O_{ab}(Z)+\kappa_1}}+\abs{\frac{\max_{Z\in S_\alpha}{\dn_{ab}(Z)}}{\min_{Z\in S_\alpha}(n_{ab}(Z)-O_{ab}(Z))+\kappa_2}}+\abs{\frac{\max_{Z\in S_\alpha}\dn_{ab}(Z)}{\min_{Z\in S_\alpha}n_{ab}(Z')+\kappa_1+\kappa_2}}}\\
&}\leq C'K^2\bp{\frac{n}{n^2p/K^2}+\frac{n}{n^2/K^2}}\lesssim \frac{1}{np},
\end{align*}
for some constant $C'$. Hence, under the event $\mathcal E_1(\beps)$, we have that
\begin{equation*}
\max_{Z\in S_\alpha} \max_{Z'\in \cN(Z)}\left|{\log\frac{\Pi(Z|A)}{\Pi(Z'|A)}-\bp{Q_{LM}(Z,A)-Q_{LM}(Z',A)}}\right| \leq \eps_{LM},
\end{equation*}
for some positive sequence $\eps_{LM}\goto 0$. The absolute sign is due to the symmetry.
\end{proof}

\begin{lemma}\label{boundtildePtrueP}
    For a general $K\geq 2$, define $\tB_{ab} = \EE{O_{ab}(Z)}/n_{ab}(Z)$, and we have
    \begin{equation*}
    \begin{split}
        \frac{\abs{\tB_{aa'} -B_{aa'}}}{p-q} = \begin{dcases}
    \frac{\sum_{k\neq l}R_{ak}R_{al}}{n'_{a}(n'_a-1)}, &\text{if~} a=a',\\[10pt]
    \frac{\sum_{k}R_{ak}R_{a'k}}{n'_{a} n'_{a'}}, &\text{if~} a\neq a'.
    \end{dcases}
    \end{split}
\end{equation*}
It follows that
    \begin{align*}
        \Norm{\tB-B}_{\infty} = \max_{a,b\in[K]}\abs{\tB_{ab} -B_{ab}}\leq \frac{2K\alpha m}{n}(p-q)
    \end{align*}
    for some constant $C$ depending on $\alpha,\beta$.
\end{lemma}
\begin{proof}
We split the proof of Lemma \ref{boundtildePtrueP} into two cases and calculate the results based on \eqref{eq:RcalB}.
\begin{description}[leftmargin = 0cm,labelsep = 0.5cm]
    \item[Case 1: $a=b$.] It follows that
    \begin{align*}
        \frac{{\abs{\tB_{aa} -B_{aa}}}}{p-q}&= \abs{\frac{(RBR^T)_{aa}-p n_a'}{n'_a(n'_a-1)}-p}\Big / (p-q) = \frac{\sum_{k\neq l}R_{ak}R_{al}}{n'_a(n'_a-1)}\leq \frac{2m_a}{n_a'}
        \leq \frac{2K\alpha m_a}{n}\leq \frac{2K\alpha m}{n},
    \end{align*}
    where the first inequality holds trivially by analyzing the cases with $m_a=0$, $m_a=1$, and $m_a\geq 2$, respectively. The second inequality is by the definition of $S_\alpha$.
    
    \item[Case 2: $a\neq b$.] In this case, we have
        \begin{align*}
    \frac{\abs{\tB_{ab} -B_{ab}}}{p-q}&= \abs{\frac{(RPR^T)_{ab}}{n'_a\cdot n'_{b}}-q}\Big / (p-q) = \frac{\sum_{k}R_{ak}R_{bk}}{n'_a\cdot n'_{b}}\leq \frac{R_{aa}R_{ba}+R_{bb}R_{ab}+m_a m_{b}}{n_a'\cdot n'_{b}}\\
    &\leq \frac{n_a'\cdot m_{b}+n'_{b}\cdot m_a}{n_a'\cdot n'_{b}}\leq \frac{K\alpha m_a+K\alpha m_{b}}{n}\leq \frac{K\alpha m}{n}.
    \end{align*}
\end{description}
    The result simply follows by combining two cases.
\end{proof}

\begin{lemma}\label{lm:bound_ec}
    Let $\beps,\gamma$ be any two positive sequences satisfying that $\gamma,\beps\goto 0$ and $\beps^2nI\goto\infty$. Denote $\hat B_{ab} = O_{ab}(Z)/n_{ab}(Z)$. Under the events $\mathcal{E}_1(\beps)$ and $\mathcal E_2$ defined in (\ref{eq:allevents}), for any $Z\in S_\alpha$ with $m \leq \gamma n$, we have
    \begin{align}
    &\sum_{a\leq b}\abs{\dtO_{ab}\log \bp{\frac{\hat B_{ab}}{B_{ab}}}+ (\dtn_{ab}-\dtO_{ab}) \log\bp{\frac{1-\hat B_{ab}}{1- B_{ab}}}} \lesssim (\beps+\gamma) mnI.\label{eq:bounderrorsmall}
    \end{align}
\end{lemma}
\begin{proof}
    By Lemma \ref{lm:techError}, we have 
    \[
        \eqref{eq:bounderrorsmall} \leq \underbrace{\sum_{a\leq b}\abs{\dtO_{ab}-B_{ab}\dtn_{ab}}\cdot \abs{\log \frac{\hat B_{ab} (1-B_{ab})}{B_{ab}(1-\hat B_{ab})}}}_{(A)} + \underbrace{\sum_{a\leq b}\abs{\dtn_{ab}}\cdot \KL{B_{ab}}{\hat B_{ab}}}_{(B)}.
    \]
    For any $a,b\in[K]$, we have
    \begin{align*}
    \abs{\dtO_{ab}-\dtn_{ab}\cdot B_{ab}}&\leq \abs{\dtO_{ab}-\EE{\dtO_{ab}}}+\abs{\EE{\dtO_{ab}}-\dtn_{ab}\cdot B_{ab}}\\
    &=\abs{X_{ab}(Z)-X_{ab}(Z^*)}+n_{ab}(Z)\abs{\tB_{ab}-B_{ab}}\\
    &\leq \Norm{X(Z)-X(Z^*)}_{\infty} + \bp{\frac{\alpha n}{K}}^2\cdot \Norm{\tB-B}_{\infty}.
    \end{align*}
    Thus, under the events $\mathcal E_1(\beps)$ and $\mathcal{E}_2$, by Lemma \ref{boundtildePtrueP}, it follows that
    \[
        \max_{a,b}\abs{\dtO_{ab}-\dtn_{ab}\cdot B_{ab}}\lesssim mn(p-q).
    \]
    Under the events $\mathcal E_1(\beps)$, by Lemma \ref{tildePhatP} and Lemma \ref{boundtildePtrueP}, we further have that for any $Z\in S_\alpha$,
    \[
        \Norm{\hat B-B}_{\infty}\leq \Norm{\hat B-\tB}_{\infty}+\Norm{\tB-B}_{\infty}\lesssim (\beps+\gamma)(p-q),
    \]
    and thus it follows that
    \[
        (A)\lesssim (\beps+\gamma)mnI.
    \]
    For the term $(B)$, under the events $\cE_1(\beps)$, since $\abs{\dtn_{ab}} = \abs{n_{ab}\bp{Z}-n_{ab}(Z^*)}\leq 2mn$ trivially holds, by bounding the Kullback-Leibler divergence with $\chi^2$-divergence, it follows that
    \begin{align*}
        (B)\leq K^2\max_{a,b\in[K]}\bb{\abs{\dtn_{ab}}\cdot \KL{B_{ab}}{\hat B_{ab}}}\lesssim mn \cdot \frac{\Norm{\hat B-B}_{\infty}^2}{p}\lesssim (\beps+\gamma)^2mnI.
    \end{align*}
    By combining $(A)$ and $(B)$, the result directly follows.
\end{proof}

\begin{lemma}\label{lm:post_boundGe_1}
Suppose $p\asymp q$. Then, we have
\begin{align*}
 \sum_{a}n_a(Z)\KL{p}{\tB_{aa}}\leq C nI,
 \end{align*}
 for some constant $C$ depending on $\alpha$.

\end{lemma}
\begin{proof}
        Recall that $n_a'=n_a(Z)$. By Lemma \ref{boundtildePtrueP}, since $\tB_{aa}\leq p$ for any $a\in [K]$, then we have
    \[p-\tB_{aa} = \frac{\sum_{k\neq l}R_{ak}R_{al}}{n_a'(n_a'-1)}(p-q)\leq \frac{2m_a}{n_a'}(p-q).\]
We upper bound the Kullback-Leibler divergence by $\chi^2$-divergence, and it follows that
    \begin{align*}
         \sum_{a}n_a(Z)\KL{p}{\tB_{aa}}&\leq \sum_{a}n_a'\frac{(p-\tB_{aa})^2}{q(1-q)}\\
         &\leq \sum_{a}n_a'\frac{1}{q(1-q)}\frac{4m_a^2(p-q)^2}{(n_a')^2}\\
         &\leq CnI,
    \end{align*}
    for some constant $C$ depending on $\alpha$.
\end{proof}

\begin{lemma}\label{LBKL}
    Suppose $p,q\goto 0$, $p\asymp q$. For any $x,y\in [q,p]$, we have
    \[\KL{x}{y}\geq \frac{(x-y)^2}{2p}.\]
\end{lemma}
\begin{proof}
    Suppose we fix $x$, and construct
\begin{align*}
    f(y) = x\log\frac{x}{y}+(1-x)\log \frac{1-x}{1-y} - \frac{(x-y)^2}{2p},
\end{align*}
and
\begin{align*}
f'(y) = (y-x)\bp{\frac{1}{y(1-y)}-\frac{1}{p}},
\end{align*}
where $1/y(1-y)>1/p$ always holds. Thus $f(y)\geq f(x)=0$.
\end{proof}

\begin{lemma}\label{lm:post_boundGe_2}
    Let $\gamma\goto 0$ be any positive sequence with $\gamma^2nI\goto\infty$. For $m\geq \gamma n$, we have that
\begin{align*}
    \sum_{a,b,k,l}R_{ak}R_{bl}\KL{B_{kl}}{\tB_{ab}}\geq C(\alpha,\beta,K)mnI,
\end{align*}
for some constant $C$ depending on $\alpha,\beta,K$.
\end{lemma}

\begin{proof}
    By Lemma \ref{LBKL}, $\KL{x}{y}\geq \frac{(x-y)^2}{2p}$. Then, we have
    \begin{align}
    &\sum_{a,b,k,l}R_{ak}R_{bl}\KL{B_{kl}}{\tB_{ab}}\nonumber\\
    \geq &\sum_{a}\sum_{k,l}R_{ak}R_{al} \KL{B_{kl}}{\tB_{aa}}\nonumber\\
    \geq &\frac{1}{2p}\sum_a \sum_{k} R_{ak}^2(p-\tB_{aa})^2+ \frac{1}{2p}\sum_a \sum_{k\neq l} R_{ak}R_{al}[(p-q)-(p-\tB_{aa})]^2.\label{eq:LowerBoundCrossKL}
    \end{align}
    The first inequality holds since we only keep the terms with $b=a$. The second inequality is using Lemma \ref{LBKL}. By Lemma \ref{boundtildePtrueP}, it follows that
    \[
        p-\tB_{aa} = \frac{\sum_{k\neq l}R_{ak}R_{al}}{n_a'(n_a'-1)}(p-q).
    \]
    For simplicity, let
    \[T_a =n_a'(n_a'-1),~~~B_a = \sum_{k\neq l}R_{ak}R_{al},~~~\sum_{k}R_{ak}^2 = T_a+n_a'-B_a.\]
    Then we have that
    \begin{align*}
        \eqref{eq:LowerBoundCrossKL}= &\frac{(p-q)^2}{2p}\sum_{a}\frac{1}{T_a^2}\bp{(T_a+n_a'-B_a) B_a^2+ B_a (T_a-B_a)^2}\\
        =& \frac{(p-q)^2}{2p}\sum_{a}\frac{1}{T_a^2}\bp{T_a(T_a-B_a)B_a+n_a'B_a^2}\\
        \geq & \frac{(p-q)^2}{2p}\sum_{a}\frac{B_a(T_a-B_a)}{T_a}.
    \end{align*}
    Let $x=\sum_{k\neq a} R_{ak}^2$, and thus $0\leq x\leq m_a^2$. Then, we can write 
\begin{align*}
    B_a &= 2R_{aa}m_a+m_a^2-x,\\
    T_a-B_a & = R_{aa}^2 + x-n_a'.
\end{align*}
Since $B_a(T_a-B_a)$ is a quadratic function of $x$ that is concave, it follows that 
\begin{align*}
    \eqref{eq:LowerBoundCrossKL}\geq \frac{(p-q)^2}{2p}\sum_{a}\frac{1}{T_a}\min\bb{2R_{aa}m_a\bp{R_{aa}^2+m_a^2-n_a'},~(2R_{aa}m_a+m_a^2)(R_{aa}^2-n_a')}.
\end{align*}

\textbf{Claim}: there exists an $a'\in [K]$ such that $R_{a'a'}\geq Cn$ and $m_{a'}\geq C'm$ for some constants $C$ and $C'$. This is because of the following argument. Since $m=\sum_a m_a$, there must exist some $a$ such that $m_a\geq m/K$. Without loss of generality, suppose $m_1\geq m/K$. Then, there are two cases we need to consider next. Case 1: if $R_{11}\geq\frac{n}{2\beta K^2}$, then we take $a' = 1$. Case 2: if $R_{11}<\frac{n}{2\beta K^2}$, since $n_1\geq \frac{n}{K\beta}$, it follows that $\sum_{i\neq 1}R_{i1}\geq \frac{(2K-1)n}{2K^2\beta}$. Then, there must exists some $a\neq 1$ such that $R_{a1}\geq \frac{(2K-1)n}{2K^2(K-1)\beta}$. Without loss of generality, suppose $R_{21}\geq \frac{(2K-1)n}{2K^2(K-1)\beta}\geq \frac{n}{K^2\beta}$. Then, we have $m_2\geq R_{21}\geq \frac{n}{K^2\beta}\geq \frac{m}{K^2\beta}$, and thus $R_{22}\geq R_{21}-R_{11}>\frac{n}{2K^2\beta}$ by the definition of discrepancy matrix $R$. Then, we take $a'=2$. Hence, the claim always holds.

Based on the above claim, we have that
\begin{align*}
    \eqref{eq:LowerBoundCrossKL}&\geq \frac{(p-q)^2}{2p}\frac{1}{T_{a'}}\min\bb{2R_{a'a'}m_a\bp{R_{a'a'}^2+m_{a'}^2-n_{a'}'},~(2R_{a'a'}m_{a'}+m_{a'}^2)(R_{a'a'}^2-n_{a'}')}\\
    &\gtrsim mnI.
\end{align*}
The proof is complete.
\end{proof}

\begin{lemma}\label{lm:boundDeltaec}
    Recall that
\begin{align*}
        \Delta(\cdot) & = \sum_{a\leq b} n_{ab}(\cdot)\bp{\tau \bp{\frac{{O_{ab}(\cdot)}}{n_{ab}(\cdot)}}-\tau \bp{\frac{\EE{O_{ab}(\cdot)}}{n_{ab}(\cdot)}}},
\end{align*}
where $\tau(x)=x\log x+(1-x)\log (1-x)$. For any positive sequences $\beps=\gamma\goto 0$, $\theta\goto 0$ with $\gamma^2nI\goto\infty$ and $\theta^2\gamma nI\goto\infty$, under the events $\mathcal E_1(\beps)$ and $\mathcal E_4(\gamma,\theta)$ defined in \eqref{eq:allevents}, we have that for any $Z\in S_\alpha$ with $m>\gamma n$, 
\begin{align*}
    \abs{\Delta(Z)-\Delta(Z^*)}\leq \eps mnI,
\end{align*}
for some positive sequence $\eps\goto 0$.
\end{lemma}
\begin{proof}
Recall that for any $a,b\in [K]$, $\hat B_{ab}(Z) = O_{ab}(Z)/n_{ab}(Z)$ and $\wt B_{ab}(Z) = \EE{O_{ab}(Z)}/n_{ab}(Z)$. Then, $\tB_{ab}(Z^*)=B_{ab}$. Note that $\tau'(x) = \log\frac{x}{1-x},$ and $\tau''(x)=\frac{1}{x(1-x)}$.
    By Taylor expansion, it follows that for any $Z\in S_\alpha$ and any $a,b\in[K]$, there exists some $\xi_{ab}(Z)\in[\hat B_{ab}(Z), \tB_{ab}(Z)]$ such that
    \begin{align*}
    \Delta(Z) = \sum_{a\leq b}\tau'\bp{\tB_{ab}(Z)}\cdot X_{ab}(Z)+ \underbrace{\sum_{a\leq b}\frac{X_{ab}(Z)^2}{2 n_{ab}(Z)}\cdot \frac{1}{\xi_{ab}(Z)(1-\xi_{ab}(Z))}}_{Err(Z)}.
    \end{align*}
    Similarly, since $\tB_{ab}(Z^*)=B_{ab}$, we have $\xi_{ab}(Z^*)\in [\hat B_{ab}(Z^*), B_{ab}]$ such that
 \begin{align*}
    \Delta(Z^*) = \sum_{a\leq b}\tau'\bp{B_{ab}}\cdot X_{ab}(Z^*)+ Err(Z^*).
    \end{align*}
    We write $\tB(Z)=\tB$ for simplicity. Then, we have
    \begin{align*}
    &\abs{\Delta(Z)-\Delta(Z^*)} \\
     = &\abs{\sum_{a\leq b}\tau'(\tB_{ab}){X_{ab}(Z)} + \sum_{a\leq b}\tau'(B_{ab})X_{ab}(Z^*) + Err(Z)-Err(Z^*)}\\
    = &\underbrace{\sum_{a\leq b}\abs{\tau'(\tB_{ab})-\tau'(B_{ab})}\abs{X_{ab}(Z)}}_{(A)} + \underbrace{\abs{\sum_{a\leq b}\tau'(B_{ab}) \bp{X_{ab}(Z)-X_{ab}(Z^*)}}}_{(B)}+\underbrace{\abs{Err(Z)-Err(Z^*)}}_{(C)},
    \end{align*}
    and we proceed to bound each term.

    Under the event $\cE_1(\gamma)$, by Lemma \ref{boundtildePtrueP}, we have that for any $Z\in S_\alpha$,
    \begin{align*}
        (A)\leq \sum_{a\leq b}\abs{\log\frac{\tB_{ab}(1-B_{ab})}{B_{ab}(1-\tB_{ab})}}\cdot \Norm{X(Z)}_{\infty}\lesssim \frac{\Norm{\tB-B}_{\infty}}{p}\cdot \Norm{X(Z)}_{\infty}\lesssim \gamma mn I.
    \end{align*}
Under the event $\cE_4(\gamma, \theta)$, we have that for any $Z\in S_\alpha$ with $m> \gamma n$,
\begin{align*}
    (B) &= \abs{\tau'(p)\sum_{a}(X_{aa}(Z)-X_{aa}(Z^*))+\tau'(q)\sum_{a<b}(X_{ab}(Z)-X_{ab}(Z^*))}\\
    & = \abs{(\tau'(p)-\tau'(q))\cdot \sum_{a}(X_{aa}(Z)-X_{aa}(Z^*))}\\
    &\leq  \log \frac{p(1-q)}{q(1-p)} \cdot K\Norm{X(Z)-X(Z^*)}_{\infty}\\
    &\lesssim \theta mnI,
\end{align*}
where the second equality holds since $\sum_{a\leq b}O_{ab}(Z)=\sum_{a\leq b}O_{ab}(Z^*)$, and thus $\sum_{a\leq b}(X_{ab}(Z)-X_{ab}(Z^*))=0$ always holds. Under the event $\cE_1(\gamma)$, we have that for any $Z\in S_\alpha$,
\begin{align*}
    (C)&\leq 2\max_{Z\in S_\alpha}|Err(Z)|\leq \max_{Z\in S_\alpha}\sum_{a\leq b}\abs{\frac{X_{ab}(Z)^2}{n_{ab}(Z)}\cdot \frac{1}{\xi_{ab}(Z)(1-\xi_{ab}(Z))}}\\
    &\lesssim K^2 \frac{\Norm{X(Z)}_{\infty}}{n^2}\cdot \frac{1}{p} \lesssim \gamma^2n^2I.
\end{align*}
Since $m>\gamma n$, it follows that under the events $\cE_1(\gamma)$ and $\cE_4(\gamma,\theta)$, there exists some sequence $\eps\goto 0$, such that for any $Z\in S_\alpha$ with $m>\gamma n$, 
\begin{align*}
    \abs{\Delta(Z)-\Delta(Z^*)}&\leq \eps mnI.
\end{align*}
The proof is complete.
\end{proof}

\begin{lemma}\label{lm:size}
    Recall that $\cB(Z)$ is defined in \eqref{eq:defABN}. We have
    $$|S_\alpha \cap\cB(Z)|\geq \min\left\{\frac{n}{\beta K}-\frac{n}{\alpha K}-1, \frac{\alpha n-\beta n}{K}-1, m\right\} \geq \min\{{c_{\alpha,\beta} n}, m\},$$
    for some constant $c_{\alpha, \beta}$.
\end{lemma}
\begin{proof}
We split the proof of Lemma \ref{lm:size} into three cases.
\begin{description}[leftmargin = 0cm,labelsep = 0.5cm]
\item[Case 1.] Suppose there are totally $k'$ groups with size $\ceil{{n}/{K\alpha}}$ (reaching the small size boundary), denoted as the set $\mathcal{K}'$, and $|\mathcal{K}'| = k'$. Then,
    \begin{align*}
    k'\ceil{\frac{n}{K\alpha}} = \sum_{a\in \mathcal{K}'}n_a' &= \sum_{a\in \mathcal{K}'} R_{aa}+\sum_{a,b\in \mathcal{K}', a\neq b}R_{ab} +\sum_{a \in \mathcal{K}', b\not \in \mathcal{K}'}R_{ab}\\
    &=\sum_{a\in \mathcal{K}'}\bp{n_a -\sum_{b\in \mathcal{K}'\setminus \{a\}}R_{ba}-\sum_{b\not \in \mathcal{K}'}R_{ba}} + \sum_{a,b\in \mathcal{K}', a\neq b}R_{ab} + \sum_{a\in \mathcal{K}', b\not \in \mathcal{K}'}R_{ab}\\
    & = \sum_{a\in \mathcal K'}n_a-\sum_{a\in \mathcal K'}\sum_{b\not\in \mathcal K'}R_{ba}+\sum_{a\in K',b\not \in \mathcal K'}R_{ab}\\
    &\geq k'{\frac{n}{K\beta}} -\sum_{a\not \in \mathcal{K}', b\in \mathcal{K}'}R_{ab},
    \end{align*}
    and thus $|\cB(Z)\cap S_\alpha|\geq \sum_{a\not \in \mathcal{K}', b\in \mathcal{K}'}R_{ab}\geq k'{\frac{n}{K\beta}}-\ceil{\frac{n}{K\alpha}}\geq \frac{n}{K\beta}-\frac{n}{K\alpha}-1$.
 \item[Case 2.]  If there is at least one group with size $\floor{\alpha n/K}$ (reaching the large size boundary), denoted as group $a$. It directly follows that 
 \[
    m_a = n_a'-R_{aa}\geq n_a'-n_a \geq \floor{\frac{\alpha n}{K}}-\frac{\beta n}{K} \geq \frac{\alpha n}{K}-\frac{\beta n}{K}-1.
 \]
    
 \item[Case 3.]  If $\cB(Z)\subset S_\alpha$, then $|S_\alpha\cap \cB(Z) |=|\cB(Z)| = m$.
\end{description}
\end{proof}

\begin{lemma}\label{lm:pathdiffformula}
Suppose the current label assignment is $Z$, and $Z'$ corrects the $k$th sample from a misclassified group $b$ to its true group $b'$. Then, for any $a,a'\in [K]\setminus\{b,b'\}$, we have 
\begin{align*} 
\dO_{ab} &= \sum_{i,j}A_{ij}\bp{\indc{Z_i=a,Z_j=b}-\indc{Z'_i=a,Z'_j=b}} = \sum_{i}A_{ik}\indc{Z_i = a},\\
\dO_{ab'} & = \sum_{i,j}A_{ij}\bp{\indc{Z_i=a,Z_j=b'}-\indc{Z'_i=a,Z'_j=b'}} = - \sum_{i} A_{ik} \indc{Z_i=a},\\
\dO_{bb} &= \sum_{i<j}A_{ij}\bp{\indc{Z_i=Z_j=b}-\indc{Z'_i=Z'_j=b}}=\sum_{i}A_{ik}\indc{Z_i'=b},\\
\dO_{b'b'} &= \sum_{i<j}A_{ij}\bp{\indc{Z_i=Z_j=b'}-\indc{Z'_i=Z'_j=b'}}=-\sum_{i}A_{ik}\indc{Z_i=b'},\\
\dO_{bb'} & = \sum_{i,j}A_{ij}\bp{\indc{Z_i=b,Z_j=b'}-\indc{Z'_i=b,Z'_j=b'}} = -\dO_{b'b'}-\dO_{bb},\\
\dO_{aa'} &= 0.
\end{align*}
\end{lemma}

\begin{lemma}\label{lm:tildeBminus}
    Recall that $\tB_{aa'}=\EE{O_{aa'}(Z)}/n_{aa'}(Z)$ for any $a,a'\in[K]$. When $K=2$, if $(n-2m)(n-2\beta m)/n\goto \infty$, we have that
    \[
        \tB_{11}-\tB_{12} = \frac{\det(R)(R_{11}-R_{12})(p-q)}{n_1(n_1-1)n_2}(1-o(1)),
    \]
    and by symmetry,
    \[
        \tB_{22}-\tB_{21} = \frac{\det(R)(R_{22}-R_{21})(p-q)}{n_2(n_2-1)n_1}(1-o(1)).
    \]
\end{lemma}
\begin{proof}
     We use $b$ to denote one group label, and use $b'$ to denote the other. By Lemma \ref{boundtildePtrueP}, 
    \begin{equation*}
    \begin{split}
    \tB_{b'b'}-\tB_{b'b} &= \bp{p-\frac{\sum_{k\neq l}R_{b'k}R_{b'l}}{n'_{b'}(n_{b'}'-1)}(p-q)}-\bp{q+\frac{\sum_{k}R_{bk}R_{b'k}}{n'_b \cdot n'_{b'}}(p-q)} \\[5pt]
    &= \frac{\det(R)(R_{b'b'}-R_{b'b})(p-q)}{n'_{b'}(n'_{b'}-1) n'_{b}} - \frac{\bp{R_{bb}R_{b'b'}+R_{bb'}R_{b'b}}(p-q)}{n_{b'}'(n'_{b'}-1)n'_b},
    \end{split}
    \end{equation*}
    where $R_{b'b'}-R_{b'b} = n_{b'}'-m\geq {n}/{2\beta}-m$. By Lemma \ref{detR}, $\det(R)\gtrsim n(n-2m)$. Since $R_{bb}R_{b'b'}+R_{bb'}R_{b'b}\lesssim n^2$, under the condition that $(n-2m)(n-2\beta m)/n\goto\infty$, the second term is negligible compared to the first term.    
\end{proof}

\begin{lemma}\label{lm:boundPZZ'}
For any $Z\in S_\alpha$, suppose $Z'$ corrects one sample from a misclassified group $b$ to its true group $b'$. Recall $P(Z,Z')$ is defined in \eqref{eq:PZZ'}. Then, we have
\begin{align*}
        &P(Z,Z') \\
        =& -\sum_{a,l=1}^K\tR_{al}\bp{B_{lb'}\log\frac{\tB_{ab'}}{\tB_{ab}} + (1-B_{lb'})\log\frac{1-\tB_{ab'}}{1-\tB_{ab}}}\\
         = &-\frac{p-q}{n_{b'}(Z)}\sum_{a\in [K]}\log\frac{\tB_{ab'}(1-\tB_{ab})}{\tB_{ab}(1-\tB_{ab'})}\bp{\sum_{k\neq b'}R_{b'k}(R_{ab'}-R_{ak})} - \sum_{a\in[K]}(n_a(Z)-\delta_{ab'})\KL{\tB_{ab'}}{\tB_{ab}},
\end{align*}
where $\wt R_{b'b'} =R_{b'b'}-1$, otherwise $\wt R_{ab}=R_{ab}$. Here, $\delta_{ab'}=1$ when $a=b'$, otherwise $\delta_{ab'}=0$.
\end{lemma}
\begin{proof}
We write $P(Z,Z') = \sum_{a,a'} P_{a,a'}(Z,Z')$. Since $Z'$ updates one sample from a misclassified group $b$ to its true group $b'$, by Lemma \ref{lm:pathdiffformula}, we have for $a\in [K]\setminus\{b,b'\}$,
\begin{align*}
    P_{a,b}(Z,Z') + P_{a,b'}(Z,Z') & = -\bp{\EE{\dO_{ab}}\log \frac{\tB_{ab'}}{\tB_{ab}}+ (\dn_{ab}-\EE{\dO_{ab}})\log \frac{1-\tB_{ab'}}{1-\tB_{ab}}},
\end{align*}
where $\dn_{ab} = n_a(Z)=n_a(Z')$, and $\EE{\dO_{ab}} = \sum_{l\in[K]}R_{al}B_{lb'}$. Then, it follows that
\[
    P_{a,b}(Z,Z')+P_{a,b'}(Z,Z') = -\bp{\log\frac{\tB_{ab'}}{\tB_{ab}}\sum_{l\in[K]} R_{al}B_{lb'} + \log\frac{1-\tB_{ab'}}{1-\tB_{ab}}\sum_{l\in [K]} R_{al}(1-B_{lb'})}.
\]
Furthermore, we have
\begin{align*}
    P_{b,b}(Z,Z')+P_{b',b'}(Z,Z')+P_{b,b'}(Z,Z') &= -\bp{\EE{\dO_{bb}}\log \frac{\tB_{bb'}}{\tB_{bb}}+ (\dn_{bb}-\EE{\dO_{bb}})\log \frac{1-\tB_{bb'}}{1-\tB_{bb}}} \\
    &-\bp{\EE{\dO_{b'b'}}\log \frac{\tB_{bb'}}{\tB_{b'b'}}+ (\dn_{b'b'}-\EE{\dO_{b'b'}})\log \frac{1-\tB_{bb'}}{1-\tB_{b'b'}}},
\end{align*}
where $\dn_{bb} = n_b(Z')=n_b(Z)-1$, $\dn_{b'b'} = -n_{b'}(Z)$, $\EE{\dO_{bb}} = \sum_{l}R_{bl}B_{lb'} - B_{b'b'}$, and $\EE{\dO_{b'b'}} = \sum_{l}R_{b'l}B_{lb'}$. It follows that
\begin{align}
    & P_{b,b}(Z,Z')+P_{b',b'}(Z,Z')+P_{b,b'}(Z,Z') \nonumber\\
    = &-\sum_{a\in \{b,b'\}}\bp{\log \frac{\tB_{ab'}}{\tB_{ab}}\sum_{l}R_{al}B_{lb'}+ \log \frac{1-\tB_{ab'}}{1-\tB_{ab}}\sum_{l}R_{al}(1-B_{lb'})} \\
   & + B_{b'b'}\log \frac{\tB_{b'b'}}{\tB_{bb}} + (1-B_{b'b'})\log \frac{1-\tB_{b'b'}}{1-\tB_{bb}}\label{eq:extraPunknown}.
\end{align}
Thus we have
\[
    P(Z,Z') = -\sum_{a,l=1}^K\tR_{al}\bp{B_{lb'}\log\frac{\tB_{ab'}}{\tB_{ab}} + (1-B_{lb'})\log\frac{1-\tB_{ab'}}{1-\tB_{ab}}},
\]
where $\tR_{b'b'} = R_{b'b'}-1$, otherwise $\tR_{aa'}=R_{aa'}$. It leads to the first equality of Lemma \ref{lm:boundPZZ'}.

By some calculations, it follows that
    \begin{align*}
        -P(Z,Z') = \underbrace{\sum_{a,l\in [K]} \tR_{al}\bp{(B_{lb'}-\tB_{ab'})\log \frac{\tB_{ab'}}{\tB_{ab}}+(\tB_{ab'}-B_{lb'})\log\frac{1-\tB_{ab'}}{1-\tB_{ab}}}}_{H_1} + \sum_{a,l\in[K]}\tR_{al}\KL{\tB_{ab'}}{\tB_{ab}},
    \end{align*}
    where $H_1$ can be written as 
    \begin{align*}
        H_1 = \sum_{a\in [K]} \log \frac{\tB_{ab'}(1-\tB_{ab})}{\tB_{ab}(1-\tB_{ab'})}\sum_{l\in [K]} \tR_{al}(B_{lb'}-\tB_{ab'}).
    \end{align*}
     By Lemma \ref{boundtildePtrueP}, we have for $a\neq b'$,
     \begin{align*}
         \sum_{l\in [K]}R_{al}(B_{lb'}-\tB_{ab'}) &=\sum_{l\in [K]} R_{al}(B_{lb'}-q -\frac{\sum_{k\in [K]}R_{ak}R_{b'k}}{n_a'n_{b'}'}(p-q))\\
& = R_{ab'}(p-q)-\sum_{l\in [K]}R_{al}\frac{\sum_{k\in [K]}R_{ak}R_{b'k}}{n_a'n_{b'}'}(p-q)\\
& = \frac{p-q}{n_{b'}'}\bp{R_{ab'}\sum_{k\in [K]}R_{b'k}-\sum_{k\in [K]}R_{ak}R_{b'k}}\\
&=\frac{p-q}{n_{b'}'}\bp{\sum_{k\neq b'}R_{b'k}(R_{ab'}-R_{ak})},
     \end{align*}
     and for $a=b'$, we have
     \begin{align*}
         \sum_{l\in [K]}\tR_{b'l}(B_{lb'}-\tB_{b'b'}) &= \sum_{l\in [K]}\tR_{b'l}\bp{B_{lb'}-q-\frac{\sum_{k\in [K]}R_{b'k}^2-n'_{b'}}{n_{b'}'(n_{b'}'-1)}}\\
         &=\tR_{b'b'}(p-q)-\sum_{l\in [K]}\tR_{b'l}\frac{\sum_{k\in [K]}R_{b'k}^2-n_{b'}'}{n_{b'}'(n_{b'}'-1)}\\
         & =\frac{p-q}{n_{b'}'}\bp{(R_{b'b'}-1)n_{b'}'-\sum_{k\in [K]}R_{b'k}^2+n_{b'}'}\\
         & = \frac{p-q}{n_{b'}'}\bp{\sum_{k\neq b'}R_{b'k}(R_{b'b'}-R_{b'k})}.
     \end{align*}
Thus, the result follows by plugging the results into $H_1$.
\end{proof}

\begin{lemma}\label{detR}
    For $K=2$, when $m\leq {n}/{K\beta}$, we have $\det(R)\geq {n(n-2m)}/{4\alpha}$.
\end{lemma}
\begin{proof}
    Suppose $n_1=n_1(Z^*)$, $n_2=n_2(Z^*)$, and $m=d(Z,Z^*)$. Then, we can write $R$ as 
    \[R = \bp{\begin{matrix}
        n_1-\bp{\frac{m}{2}-x} & \frac{m}{2}+x\\
        \frac{m}{2}-x & n_2-\bp{\frac{m}{2}+x}
        \end{matrix}},\]
        for some $x\in[-m/2, m/2]$. Therefore,
    \[\det(R)=n_1n_2-\frac{m}{2}n-x(n_1-n_2)= (n-n_2)n_2-x(n-2n_2)-mn/2.\]
    Without loss of generality, we assume $\frac{\beta n}{2}\geq n_1\geq \frac{n}{2}\geq n_2\geq \frac{n}{2\beta}$. Then, we have the following conditions that
    \begin{align*}
    &n_1+2x=n-n_2+2x\leq \frac{n\alpha}{2},~~~n_2-2x\geq \frac{n}{2\alpha}.
    \end{align*}
    Combine all these conditions, and it directly follows that
    \begin{align*}
        \det(R)&\geq \frac{n}{2}\bp{1-\frac{1}{\alpha}}n_2+\frac{n^2}{4\alpha}-\frac{mn}{2}\\
        &\geq \frac{n}{2}\bp{\frac{n}{2\beta}+\frac{n}{2\alpha}\bp{1-\frac{1}{\beta}}-m}\\
        &= \frac{n}{2}\bp{\frac{n}{2\beta}\bp{1-\frac{1}{\alpha}}+\frac{n}{2\alpha}-m}\\
        &\geq \frac{n(n-2m)}{4\alpha},
    \end{align*}
    where the last inequality holds since $m\leq \frac{n}{2\beta}$.
\end{proof}

\begin{lemma}\label{lm:boundPZZ'2} When $K=2$, for any $Z\in S_\alpha$ with $d(Z,Z^*)=m$, denote $\eps_{m} = 1-2m/n$ for simplicity. If $(n-2m)(n-2\beta m)/n\goto\infty$ and $(1-2m/n)^3nI\goto \infty$, then we have
\begin{align*}
        \min_{Z'\in B(Z)\cap S_\alpha}-P(Z,Z') \ignore{\geq &\frac{(n-2m)^3I}{\alpha^2 n^3}\left[\bp{m-\frac{\beta n}{2}+\frac{n}{2\alpha}}_++\frac{1}{2}(n-2m)\right](1-o(1)).}
        \geq \frac{\eps_{m}^3nI}{2\alpha^2}\max\bb{1-\beta+1/\alpha, \eps_{m}}(1-o(1)).
\end{align*}
\end{lemma}
\begin{proof}
Suppose $Z'\in \cB(Z)\cap S_\alpha$ is corrects one sample from a misclassified group $b$ to its true group $b'$. Decompose $-P(Z,Z') = (A) + (B)$ by Lemma \ref{lm:boundPZZ'}, where
\begin{align*}
    (A) &= \frac{p-q}{n_{b'}(Z)}\sum_{a\in [K]}\log\frac{\tB_{ab'}(1-\tB_{ab})}{\tB_{ab}(1-\tB_{ab'})}\bp{\sum_{k\neq b'}R_{b'k}(R_{ab'}-R_{ak})}\\
    & = \frac{R_{b'b}(p-q)}{n'_{b'}}\sum_{a\in\{b,b'\}}\left[\log\bp{\frac{\tB_{ab'}(1-\tB_{ab})}{\tB_{ab}(1-\tB_{ab'})}}(R_{ab'}-R_{ab})\right]
\end{align*}
and 
\[
    (B) = \sum_{a\in[K]}(n_a(Z)-\delta_{ab'})\KL{\tB_{ab'}}{\tB_{ab}}
\]
 By Lemma \ref{lm:tildeBminus}, when $a = b'$, under the condition that $(n-2m)(n-2\beta m)/n\goto\infty$, since $R_{b'b'}-R_{b'b} = n_{b'}-m>0$, we have
    \begin{align*}
        \log\bp{\frac{\tB_{b'b'}(1-\tB_{b'b})}{\tB_{b'b}(1-\tB_{b'b'})}}(R_{b'b'}-R_{b'b})&\geq \frac{\tB_{b'b'}-\tB_{b'b}}{\tB_{b'b'}(1-\tB_{b'b'})} \cdot (R_{b'b'}-R_{b'b})\\
        &\geq (1-o(1))\frac{\det(R)(R_{b'b'}-R_{b'b})^2(p-q)}{n_{b'}'(n_{b'}'-1)n_b'}\cdot \frac{1}{\tB_{b'b'}(1-\tB_{b'b'})}\\
        &\geq (1-o(1))\frac{\det(R)(R_{b'b'}-R_{b'b})^2(p-q)}{n_{b'}'(n_{b'}'-1)n_b'}\cdot \frac{1}{p},
    \end{align*}
    and a similar argument also applies to the case with $a=b$. Hence, it follows that
\begin{align*}
    (A)&\geq \frac{R_{b'b}(p-q)^2\det(R)}{n_b'}\bp{\frac{(R_{b'b'}-R_{b'b})^2}{n'^2_{b'} n_b'}+\frac{(R_{bb}-R_{bb'})^2}{n'^2_{b} n_{b'}'}}(1-o(1))\\
    &\geq \frac{R_{b'b}(p-q)^2\det(R)}{n_b'}\frac{(n-2m)^2}{n_{b'}'n_b'n}(1-o(1))\\
    &\geq \frac{8R_{b'b}(p-q)^2\det(R)(n-2m)^2}{\alpha n^4}(1-o(1)),
\end{align*}
where the second inequality holds since by Cauchy-Schwarz inequality，
\[
    \bp{\frac{(R_{b'b'}-R_{b'b})^2}{n'_{b'}}+\frac{(R_{bb}-R_{bb'})^2}{n_{b}'}}\cdot (n_{b'}'+n_b')\geq (R_{b'b'}-R_{b'b}+R_{bb}-R_{bb'})^2 = (n-2m)^2.
\]
Third inequalities holds since $n_{b'}'n_{b'}\leq (n_{b'}'+n_{b}')^2/4$ and $n_{b}'\leq \alpha n/2$. We proceed to lower bound the term $(B)$. By Lemma \ref{LBKL}, we have $\KL{x}{y}\geq {(x-y)^2}/{2p}$, and
    \begin{align*} 
    &n'_{b'} \KL{\tB_{b'b'}}{\tB_{b'b}}+ n'_b \KL{\tB_{bb'}}{\tB_{bb}}\\[10pt]
    \geq& \det(R)^2\frac{(p-q)^2}{2p}\bp{\frac{(R_{b'b'}-R_{b'b})^2}{n'^3_{b'}n'^2_{b}}+\frac{(R_{bb}-R_{bb'})^2}{n'^3_{b}n'^2_{b'}}}(1-o(1))\\[10pt]
    \geq & \frac{(p-q)^2}{2p}\det(R)^2 \frac{(n-2m)^2}{n'^2_{b}n'^2_{b'}n}(1-o(1))\\[10pt]
    \geq& \frac{8(p-q)^2{\det(R)^2(n-2m)^2}}{{n^5}p}(1-o(1)).
    \end{align*}
Upper bound Kullback-Leibler divergence by $\chi^2$-divergence, and we have that $\KL{\tB_{bb'}}{\tB_{bb}}\leq CI$ for some constant $C$. Under the condition that $(1-2m/n)^3nI\goto\infty$, we have
\[
    (B)\geq \frac{8(p-q)^2{\det(R)^2(n-2m)^2}}{{n^5}p}(1-o(1)).
\]
It directly follows that
\begin{equation}\label{eq:Pfirst}
    -P(Z,Z') = (A)+(B)\geq \frac{8(p-q)^2(n-2m)^2\det(R)}{n^4p}\bp{\frac{R_{b'b}}{\alpha} +\frac{\det(R)}{n}}(1-o(1)).
\end{equation}

Note that
    \[\frac{n}{2\alpha}\leq R_{b'b}+R_{b'b'}\leq R_{b'b}+\frac{\beta n}{2}-(m-R_{b'b}),\]
and thus 
\begin{equation}\label{eq:boundRb'b}
R_{b'b}\geq \max\left\{\frac{m}{2}-\frac{\beta n}{4}+\frac{n}{4\alpha},~0\right\}.    
\end{equation}
By Lemma \ref{detR}, we have
\begin{align*}
    -P(Z,Z')&\geq \frac{(p-q)^2(n-2m)^3}{\alpha^2 n^3p}\left[\max\bb{m-\frac{\beta n}{2}+\frac{n}{2\alpha},0}+\frac{n-2m}{2}\right](1-o(1))\\
    &= \frac{(p-q)^2(n-2m)^3}{\alpha^2 n^3p}\max\bb{\frac{n}{2}(1-\beta+1/\alpha), \frac{n-2m}{2}}(1-o(1)),
\end{align*}
where $(p-q)^2/p\geq I$, and thus the result follows.
\end{proof}

\begin{lemma}\label{Hbbeachm}
    Under the conditions of Lemma \ref{lm:boundPZZ'2}, for $Z\in S_\alpha$ with $d(Z,Z')=m$, we have that
    \[
        -P(Z,Z')\gtrsim \frac{(n-2m)^2\det(R) I}{n^3}.
    \]
    for any $Z'\in \cB(Z)\cap S_\alpha$.
\end{lemma}
\begin{proof}
    By \eqref{eq:Pfirst}, \eqref{eq:boundRb'b} and Lemma \ref{detR}, we have
    \[
        -P(Z,Z')\gtrsim \frac{(n-2m)^2\det(R)I}{n^3}\max\bb{1-\beta+1/\alpha,1-2m/n}.
    \]
    Since $m\leq n/2\beta$, then $1-2m/n\goto 0$ only if $\beta\goto 1$, and thus $1-\beta+1/\alpha$ is a constant. Hence, the result follows.
\end{proof}
\ignore{
\begin{lemma}\label{Hbbpro}
Let $m_{\max}$ denote the max number of mistakes and $m_{\max}=\frac{n}{2\beta}(1-\delta_{\min})$. Then for any positive sequence $\gamma\goto 0$, we have
\[
    \min_{Z\in S_\alpha: \gamma n<d(Z,Z^*)\leq m_{\max}} \min_{Z'\in \cB(Z)\cap S_\alpha}-P(Z,Z')\geq \frac{nI}{2\alpha^2}\eps_{\beta,\delta_{\min}}^3\max\{1-\beta+1/\alpha,\eps_{\beta,\delta_{\min}}\}(1-o(1)),
\]
where $\eps_{\beta,\delta_{\min}} = 1-1/\beta+\delta_{\min}/\beta$.
\end{lemma}
\begin{proof}
    Note that $m\leq {n}/{2\beta}$. Denote 
    \begin{align*}
        f(m) &= (n-2m)^3\bp{m-\frac{n}{2}(\beta-1/\alpha)}_++(n-2m)^4/2\\
        & = (n-2m)^3/2\cdot \max\bb{n(1-\beta+1/\alpha),(n-2m)}.
    \end{align*}
    It is obvious that $f(m)$ is a decreasing function of $m$ when $m<{n}/{2\beta}$. Let $m_{\max}$ denote the maximum number of mistakes, and $m_{\max} = \frac{n}{2\beta}(1-\delta_{\min})$. Hence, it follows that
    \begin{align*}
    -P(Z,Z')&\geq \frac{f(m)I}{\alpha^2n^3}(1-o(1))\geq \frac{f\bp{m_{\max}}I}{\alpha^2n^3}(1-o(1))\\
    & = \frac{nI}{2\alpha^2}(1-1/\beta+\delta_{\min}/\beta)^3\max\{(1-\beta+1/\alpha),(1-1/\beta+\delta_{\min}/\beta)\}(1-o(1))\ignore{\\
    & = \frac{nI}{2\alpha^2}(1-1/\beta+\delta_{0}/\beta)^3\max\{(1-\beta+1/\alpha),(1-1/\beta+\delta_{0}/\beta)\}(1-o(1))}.
    \end{align*}
\end{proof}
}
\begin{lemma}\label{dOminusdn}
    For $K=2$, we have
    \[\max_{Z'\in \cB(Z)\cap S_\alpha}\max_{a,a'\in \{1,2\}}\abs{\EE{\dO_{aa'}}-\tB_{aa'} \dn_{b'b'}}\leq C(n-2m)(p-q)\]
    for some constant $C$ depending on $\alpha$.
\end{lemma}
\begin{proof}
        Without loss of generality, suppose the current state is $Z$, and $Z'$ corrects one sample from misclassified group 1 to true group 2. We write $n_1 = n_1(Z^*)$, $n_2=n_2(Z^*)$ for simplicity. Then, we have
    \[R_{Z} = \bp{\begin{matrix}
        n_1-s & t\\
        s & n_2-t
        \end{matrix}},~~R_{Z'} = \bp{\begin{matrix}
        n_1-s & t-1\\
        s & n_2-t+1
        \end{matrix}},\]
        where $t = \sum_{i=1}^n\indc{Z_i=1,~Z_i^*=2}$, and $s=\sum_{i=1}^n\indc{Z_i=2,~Z_i^*=1}$. Thus, $t+s=m$. By Lemma \ref{lm:pathdiffformula}, we have
        \begin{align*}
            \dO_{11}-\tB_{11}\dn_{11} & = \frac{-(n_1-s)(n_1-m)}{n_1'},\\
            \dO_{22}-\tB_{22}\dn_{22} & = \frac{-s(n_2-m+1)}{n_2'-1},\\
            \dO_{11}-\tB_{12}\dn_{11} & = \frac{-s(n_1-m)}{n_2'}-\frac{n_1s+n_2t-m}{n_1'n_2'},\\
            \dO_{22}-\tB_{12}\dn_{22} & = \frac{-(n_1-s)(n_2-m)}{n_1'},
        \end{align*}
    and it directly follows that
    \[\max_{a,a'\in \{1,2\}}\abs{\EE{\dO_{aa'}}-\tB_{aa'} \dn_{aa'}}\lesssim \bp{n-2m}(p-q).\]
\end{proof}

\begin{lemma}\label{minPiR}
    For any $Z\in S_\alpha$, write $R_Z(a,b)$ as $R_{ab}$ for simplicity. Then, we have
    \begin{align*}
    \sum_{a,b,k,l}R_{ak}R_{bl}\KL{B_{kl}}{\tB_{ab}}\lesssim mnI.
    \end{align*}
\end{lemma}
\begin{proof}
    Since $\KL{x}{y}\leq \frac{(x-y)^2}{y(1-y)}$ for any $x,y\in(0,1)$, we have
    \begin{align*}
        \sum_{a,b,k,l}R_{ak}R_{bl}\KL{B_{kl}}{\tB_{ab}}&\leq \frac{1}{q(1-q)}\sum_{a,b,k,l}R_{ak}R_{bl}(B_{kl}-\tB_{ab})^2\\
        & = \frac{1}{q(1-q)}\sum_{a,b,k,l}R_{ak}R_{bl}((B_{kl}-B_{ab})+(B_{ab}-\tB_{ab}))^2\\
        &\leq \frac{2}{q(1-q)}\sum_{a,b,k,l}R_{ak}R_{bl}\bp{(B_{kl}-B_{ab})^2+(B_{ab}-\tB_{ab})^2}\\
       & = (A) + (B),
\end{align*}
where by Lemma \ref{boundtildePtrueP},
\begin{align*}
   (B) = \frac{1}{q(1-q)}\sum_{a,b,k,l}R_{ak}R_{bl}(\tB_{ab}-B_{ab})^2 & \leq \frac{2(p-q)^2}{q(1-q)}n^2\bp{\frac{2K\alpha m}{n}}^2 \lesssim mn I,
\end{align*}
and 
\begin{align*}
    (A) & \leq \frac{2(p-q)^2}{q(1-q)}\bp{\sum_{a}\sum_{k\neq l}R_{ak}R_{al}+\sum_{a\neq b}\sum_{k}R_{ak}R_{bk}}\lesssim mnI,
\end{align*}
since $\sum_{a}\sum_{k\neq l}R_{ak}R_{al}\leq \sum_{a}n_a(Z)m_a\leq mn$, and a similar bound holds for $\sum_{a\neq b}\sum_{k}R_{ak}R_{bk}$. By combining $(A)$ and $(B)$, the proof is complete.
\end{proof}

\begin{lemma}\label{boundmsmallunknown}
    Let $\gamma\goto 0$ be any positive sequence. Under the events $\mathcal E_1(\beps)$ and $\mathcal E_2$ defined in \eqref{eq:allevents}, for any $Z\in S_\alpha$ with $m\leq \gamma n$, we have
    \begin{align*}
    \sum_{a<b}n_{ab}\cdot \KL{\frac{O_{ab}}{n_{ab}}}{\frac{O_{ab}(Z)}{n_{ab}(Z)}}&\leq Cm^2I
    \end{align*}
    for some constant $C$ only depending on $K, \beta,\alpha$.
\end{lemma}
\begin{proof} For any $a,b\in [K]$, we have
    \begin{align*}
    \abs{\frac{O_{ab}}{n_{ab}}-\frac{O_{ab}(Z)}{n_{ab}(Z)}} &=
    \abs{\frac{X_{ab}(Z^*)}{n_{ab}}-\frac{X_{ab}(Z)}{n_{ab}(Z)}+B_{ab}-\tB_{ab}}\\ &\leq \underbrace{\abs{B_{ab}-\tB_{ab}}}_{(A)}+\underbrace{\abs{\frac{X_{ab}(Z^*)\bp{n_{ab}(Z)-n_{ab}}}{n_{ab}(Z)\cdot n_{ab}}}}_{(B)}+\underbrace{\abs{\frac{X_{ab}(Z^*)-X_{ab}(Z)}{n_{ab}(Z)}}}_{(C)}.
    \end{align*}
    By Lemma \ref{boundtildePtrueP}, we have
    \begin{align*}
    (A)\leq \frac{2K\alpha m}{n}(p-q)\asymp \frac{m}{n}(p-q).
    \end{align*}
    Under the event $\mathcal E_1(\beps)$, we have
    \begin{align*}
(B)\leq \frac{\beps n^2(p-q)\cdot 2mn}{(n/K\alpha)^4} \asymp \beps\frac{m}{n}(p-q),
    \end{align*}
    where $\beps$ is the positive sequence that defines the event $\mathcal E_1(\beps)$. Under the event $\mathcal E_2$, we have that
    \begin{align*}
    (C)&\leq \frac{(\alpha+\beta)mn(p-q)/K}{(n/K\alpha)^2} \asymp \frac{m}{n}(p-q).
    \end{align*}
    Hence, it follows that
    \[\abs{\frac{O_{ab}}{n_{ab}}-\frac{O_{ab}(Z)}{n_{ab}(Z)}}\lesssim \frac{m}{n}(p-q),\]
    for all $Z\in S_\alpha$ with $m\leq \gamma n$. Furthermore, under the event $\mathcal E_1(\beps)$, we have
    \begin{align*}
    \abs{\frac{O_{ab}(Z)}{n_{ab}(Z)}-B_{ab}}&\leq \abs{\frac{X_{ab}(Z)}{n_{ab}(Z)}}+\abs{\tB_{ab}-B_{ab}}\lesssim (\beps+\gamma) (p-q),
    \end{align*}
    and thus $\frac{O_{ab}(Z)}{n_{ab}(Z)}\gtrsim p$. Hence, by $\KL{x}{y}\leq \frac{(x-y)^2}{y(1-y)}$ for any $x,y\in (0,1)$, we have that,
    \begin{align*}
    \sum_{a\leq b}n_{ab}\cdot \KL{\frac{O_{ab}}{n_{ab}}}{\frac{O_{ab}(Z)}{n_{ab}(Z)}}&\lesssim K^2 n^2 \bp{\frac{m}{n}}^2 \frac{(p-q)^2}{p}\lesssim m^2I.
    \end{align*}
\end{proof}


\subsection{Proofs of technical lemmas}

\begin{lemma}\label{supp-lm:TV}
    $P,\wt P$ are the probability measures defined in set $\Omega$. Suppose there exists a subset $A\subset\Omega$ such that $\wt P(B) = {P(B\cap A)}/{P(A)}$ for any set $B\subset \Omega$. Then, we have
    \[\TVdiff{\wt P}{P}\leq 2P(A^c).\]
\end{lemma}
\begin{proof}
It is obvious that $P(B) = P(B\cap A)+P(B\cap A^c)$. Then, 
\begin{align*}
\TVdiff{\wt P}{P}&= \max_B\left|\frac{P(B\cap A)-P(A)P(B)}{P(A)}\right|\\
& = \max_B\left|\frac{P(B\cap A)P(A^c)-P(A)P(B\cap A^c)}{P(A)}\right|\\
&\leq 2P(A^c).
\end{align*}
\end{proof}

\begin{lemma}\label{lm:boundprior}
    For any positive integers $x$ and $y$, for any constant $\beta>0$, we have
    \[\log\frac{\Gamma(x+\beta)}{\Gamma(y+\beta)}\leq x\log x-y\log y-(x-y)+\frac{\beta^2}{y+\beta}+(\beta+2)\abs{\log\frac{y+\beta}{x+\beta}}.\]
\end{lemma}
\begin{proof}
For any positive constants $a$ and $b$, if $b-a$ is a positive integer, then it is easy to verify that
    \begin{align*}
    \log \frac{\Gamma(b)}{\Gamma(a)} &= \sum_{k=a}^{b-1} \log k\leq b\log b-a\log a-(b-a),\\
    \log \frac{\Gamma(b)}{\Gamma(a)}&\geq b\log b-a\log a-(b-a) -\bp{\frac{1}{a}-\frac{1}{b}+\log {\frac{b}{a}}}.\\
    &\geq b\log b-a\log a-(b-a) -2{\log {\frac{b}{a}}}
    \end{align*}
    Then, for any $a,b\geq 1$, if $a-b$ is an integer, then we have
    \[\log \frac{\Gamma(a)}{\Gamma(b)}\leq a\log a-b\log b-(a-b)+2\bp{\log {\frac{b}{a}}}_+.\]
\ignore{    For any $x,y\in {\cal N}$, since
    \[x\log(x+\beta)-x\log x\leq \beta\]
    and
    \[y\log(y+\beta)-y\log y\geq
    \begin{dcases}
    y\bp{\frac{\beta}{y}-\frac{\beta^2}{2y^2}}= \beta-\frac{\beta^2}{2y}\geq \beta-\beta^2 & y\in {\cal N}_+\\
    0& y=0
    \end{dcases}\]}
Now, let $a=x+\beta$ and $b=y+\beta$. It follows that
    \begin{align*}
    &\log \frac{\Gamma(x+\beta)}{\Gamma(y+\beta)}\\
    \leq& x\log (x+\beta)-y\log (y+\beta) -(x-y)+\beta (\log (x+\beta)-\log (y+\beta))+2\bp{\log \frac{y+\beta}{x+\beta}}_+\\
    \leq& x\log x -y\log y-(x-y)+(\beta+2)\abs{\log\frac{y+\beta}{x+\beta}}+Err,
    \end{align*}
    where we write
\begin{align*}
    Err = &\bp{x\log(x+\beta)-y\log(y+\beta)} -\bp{x\log x-y\log y}\\
    =& x \log\bp{\frac{x+\beta}{x}}+y\log\bp{\frac{y}{y+\beta}}\\
    \leq & \beta-  \frac{\beta y}{y+\beta}\\
    = & \frac{\beta^2}{y+\beta}.
\end{align*}
Hence, the result follows.
\end{proof}

\begin{lemma}\label{lm:boundExpKnowP}
Suppose $x\sim \text{Bernoulli}(q)$ and $y\sim \text{Bernoulli}(p)$ with $p,q\goto 0,p\asymp q$. Then, for any constant $C$, we have that
\[
   \max\bb{\EE{e^{Ct^*(x-\lambda^*)}},\EE{e^{Ct^*(y-\lambda^*)}}}\leq e^{C'I} 
\]
for some constant $C'$, and $t^*,\lambda^*$ are defined in \eqref{eq:deflambda_t}.
\end{lemma}
\begin{proof}
Since $p,q\goto 0$ and $p\asymp q$, we have
\[
    \log \frac{1-q}{1-p}\leq \log \frac{p}{q}\leq \frac{p}{q}-1, ~~ \log\frac{1-q}{1-p}\geq \frac{p-q}{1-q}\geq p-q.
\]
Suppose $C>0$, and then it follows that
\begin{align*}
\EE{\exp\bp{Ct^*(x-\lambda^*)}} &= \exp(-Ct^*\lambda^*)\bp{q\exp(Ct^*)+1-q}\\
&\leq \exp(-Ct^*\lambda^*+q\exp(Ct^*)-q)\\
&\leq \exp\bp{\frac{-\frac{C}{2}(p-q)+q(p/q)^C-q}{(\sqrt{p}-\sqrt{q})^2}(1-o(1))\cdot I}\\
&\leq \exp\bp{C'\cdot I},
\end{align*}
for some constant $C'$ depending on $C$. The other cases follow by a similar argument.
\end{proof}

\begin{lemma}\label{lm:techError}
For any $a,b\in (0,1)$ and $x,y\in \mathbb{R}$, we have
    \begin{align*}
    \abs{x\log\frac{a}{b} + (y-x)\log \frac{1-a}{1-b}}\leq \abs{x-by}\cdot \abs{\log \frac{a(1-b)}{b(1-a)}}+\abs{y}\cdot\KL{b}{a}.
    \end{align*}    
\end{lemma}
\begin{proof}
    It follows that
    \begin{align*}
        x\log \frac{a}{b}+ (y-x)\log\frac{1-a}{1-b} &= (x-by+by)\log \frac{a}{b}+ (y-by+by-x)\log\frac{1-a}{1-b}\\
        &=(x-by)\log\frac{a(1-b)}{b(1-a)} - y\cdot \KL{b}{a}.
    \end{align*}
    Thus, the result directly follows.
    \ignore{
    Thus, by $\abs{\log\frac{s}{t}}\leq \frac{|s-t|}{\min\{s,t\}}$ for $s,t>0$ and bounding Kullback-Leibler divergence by $\mathcal X^2$-divergence, we have
    \begin{align*}
        \abs{x\log \frac{a}{b}+ (y-x)\log\frac{1-a}{1-b}}&\leq \abs{x-by}\cdot \abs{\log \frac{a(1-b)}{b(1-a)}}+\abs{y}\cdot\KL{b}{a}\\
        &\leq \abs{x-by}\frac{|a-b|}{\min\{a(1-b), b(1-a)\}}+\abs{y} \cdot \frac{|a-b|^2}{a(1-a)}
    \end{align*}}
\end{proof}

\ignore{
\begin{lemma}\label{lm:boundExpKnowP}
Suppose $x\sim \text{Bernoulli}(q)$ and $y\sim \text{Bernoulli}(p)$ with $p,q\goto 0,p\asymp q$. Then, there exists some constant $C$ such that 
\[
    \max\bb{\EE{e^{t^*(x-\lambda^*)}},\EE{e^{-t^*(x-\lambda^*)}},\EE{e^{t^*(y-\lambda^*)}},\EE{e^{-t^*(y-\lambda^*)}}}\leq e^{CI},    
\]
for some constant $C$, where $t^*,\lambda^*$ are defined in \eqref{eq:deflambda_t}.
\end{lemma}
\begin{proof}
    \begin{align*}
\EE{\exp\bp{t^*(x-\lambda^*)}} &= \exp(-t^*\lambda^*)\bp{q\exp(t^*)+1-q}\\
&\leq \exp(-t^*\lambda^*+q\exp(t^*)-q)\\
&\leq \exp\bp{\frac{-\frac{1}{2}(p-q)+q\sqrt{p/q}-q}{(\sqrt{p}-\sqrt{q})^2}(1-o(1))\cdot I}\\
&\leq \exp\bp{-(1/2-o(1))\cdot I},\\[10pt]
\EE{\exp\bp{-t^*(x-\lambda^*)}} &= \exp(t^*\lambda^*)\bp{q\exp(-t^*)+1-q}\\
&\leq \exp\bp{\frac{\frac{1}{2}(p-q)+q\sqrt{q/p}-q}{(\sqrt{p}-\sqrt{q})^2}(1-o(1))\cdot I}\\
&\leq \exp((1/2-o(1))\cdot I),\\[10pt]
\EE{\exp\bp{t^*(y-\lambda^*)}} &= \exp(-t^*\lambda^*)\bp{p\exp(t^*)+1-p}\\
&\leq \exp\bp{\frac{-\frac{1}{2}(p-q)+p\sqrt{p/q}-p}{(\sqrt{p}-\sqrt{q})^2}(1+o(1))\cdot I}\\
&\leq \exp(O(1)\cdot I),\\[10pt]
\EE{\exp\bp{-t^*(y-\lambda^*)}} &= \exp(t^*\lambda^*)\bp{p\exp(-t^*)+1-p}\\
&\leq \exp\bp{\frac{\frac{1}{2}(p-q)+p\sqrt{q/p}-p}{(\sqrt{p}-\sqrt{q})^2}(1+o(1))\cdot I}\\
&\leq \exp(-(1/2-o(1))\cdot I).\\
\end{align*}
\end{proof}
}
\clearpage

\bibliographystyle{plainnat}
\bibliography{MCMC}


\end{document}